\documentclass[twoside,11pt,reqno]{amsart}
\usepackage{amsmath,amssymb,amscd,mathrsfs,epic,wasysym,latexsym,tikz,mathrsfs,cite,hyperref}
\usepackage{pb-diagram}
\usepackage[matrix,arrow]{xy}

\usepackage[skip=3pt,font=footnotesize,labelfont=bf]{caption}

\usepackage[enableskew]{youngtab}
\usepackage{ytableau}
\usepackage{color}
\usepackage{mathdots}
\usepackage{enumerate}
\usepackage[margin=1.2in]{geometry}

\usepackage{young}

\makeatletter

\numberwithin{equation}{section}

\newtheorem{proposition}[equation]{Proposition}
\newtheorem{lemma}[equation]{Lemma}
\newtheorem{theorem}[equation]{Theorem}
\newtheorem{corollary}[equation]{Corollary}
\theoremstyle{definition}  
\newtheorem{definition}[equation]{Definition}
\newtheorem{remark}[equation]{Remark}

\newtheorem{example}[equation]{Example}

\newtheorem*{thmA}{Theorem A}
\newtheorem*{thmB}{Theorem B}
\newtheorem*{thmC}{Theorem C}

\let\<\langle
\let\>\rangle

\newcommand\Comment[2][\relax]{\space\par\medskip\noindent%
   \fbox{\begin{minipage}{\textwidth}\textbf{Comment\ifx\relax#1\else---#1\fi}\newline%
        #2\end{minipage}}\medskip
}

\def\bc{\text{\boldmath$c$}}

\def\b1{\text{\boldmath$1$}}

\newcommand{\im}{\operatorname{im}}

\newcommand{\res}{\operatorname{res}}
\newcommand{\height}{\operatorname{ht}}

\newcommand{\infl}{\operatorname{dil}}

\newcommand{\Z}{\mathbb{Z}}

\newcommand{\N}{\mathcal{N}}
\newcommand{\SSSS}{\mathcal{S}}
\newcommand{\SSSSS}{\boldsymbol{\mathcal{S}}}

\def\phi{{\varphi}}

\newcommand{\ttt}{{\tt t}}

\newcommand{\NEarrow}{\mathbin{\rotatebox[origin=c]{45}{$\Rightarrow$}}}
\newcommand{\SEarrow}{\mathbin{\rotatebox[origin=c]{-45}{$\Rightarrow$}}}

\newcommand{\Ind}{{\mathrm {Ind}}}

\newcommand{\diag}{{\mathrm {diag}}}
\newcommand{\dist}{{\mathrm {dist}}}
\newcommand{\Res}{{\mathrm {Res}}}

\newcommand{\QQ}{{\mathbb Q}}
\newcommand{\ZZ}{{\mathbb Z}}
\newcommand{\NN}{{\mathbb N}}

\newcommand{\D}{{\mathcal D}}

\newcommand{\so}{{\sf S}}
\newcommand{\no}{{\sf N}}
\newcommand{\ea}{{\sf E}}
\newcommand{\we}{{\sf W}}
\newcommand{\T}{{\sf T}}
\newcommand{\nnee}{{\sf NE}}
\newcommand{\nnww}{{\sf NW}}
\newcommand{\ssee}{{\sf SE}}
\newcommand{\ssww}{{\sf SW}}

\newcommand{\blam}{\boldsymbol{\lambda}}
\newcommand{\bkap}{\boldsymbol{\kappa}}
\newcommand{\btau}{\boldsymbol{\tau}}
\newcommand{\bbep}{\boldsymbol{\varepsilon}}
\newcommand{\bmu}{\boldsymbol{\mu}}
\newcommand{\bnu}{\boldsymbol{\nu}}
\newcommand{\bbeta}{\boldsymbol{\beta}}

\def\rev{{\operatorname{rev}}}
\def\cont{{\operatorname{cont}}}
\def\init{{\operatorname{init}}}
\def\height{{\operatorname{ht}}}

\def\b{\mathfrak{b}}
\def\k{\Bbbk}

\def\height{{\operatorname{ht}}}

\def\re{{\mathrm{re}}}
\def\im{{\mathrm{im}\,}}

{\catcode`\|=\active
  \gdef\set#1{\mathinner{\lbrace\,{\mathcode`\|"8000%
  \let|\midvert #1}\,\rbrace}}
}
\def\midvert{\egroup\mid\bgroup}

\colorlet{darkgreen}{green!50!black}
\tikzset{dots/.style={very thick,loosely dotted},
         greendot/.style={fill,circle,color=darkgreen,inner sep=1.5pt,outer sep=0}
}
\def\greendot(#1,#2){\node[greendot] at(#1,#2){}}

\newenvironment{braid}{
  \begin{tikzpicture}[baseline=6mm,blue,line width=1pt, scale=0.4,
                      draw/.append style={rounded corners},
                      every node/.append style={font=\fontsize{5}{5}\selectfont}]%
  }{\end{tikzpicture}
}

\def\Grid(#1,#2){
  \draw[very thin,gray,step=2mm] (0,0)grid(#1,#2);
  \draw[very thin,darkgreen,step=10mm] (0,0)grid(#1,#2);
}

\newcommand\Tableau[2][\relax]{
  \begin{tikzpicture}[scale=0.5,draw/.append style={thick,black}]
    \ifx\relax#1\relax%
    \else 
      \foreach\box in {#1} { \filldraw[blue!30]\box+(-.5,-.5)rectangle++(.5,.5); }
    \fi
    \newcount\row\newcount\col
    \row=0
    \foreach \Row in {#2} {
       \col=1
       \foreach\k in \Row {
          \draw(\the\col,\the\row)+(-.5,-.5)rectangle++(.5,.5);
          \draw(\the\col,\the\row)node{\k};
          \global\advance\col by 1
       }
       \global\advance\row by -1
    }
  \end{tikzpicture}
}

\newcommand{\hackcenter}[1]{
 \xy (0,0)*{#1}; \endxy}

\newcommand\YoungDiagram[2][\relax]{
  \begin{tikzpicture}[scale=0.5,draw/.append style={thick,black}]
    \ifx\relax#1\relax%
    \else 
    \foreach\box in {#1} {
      \filldraw[blue!30]\box rectangle ++(1,1);
    }
    \fi
    \newcount\row
    \row=0
    \foreach \col in {#2} {
       \draw(1,\the\row)grid ++(\col,1);
       \global\advance\row by -1
    }
  \end{tikzpicture}
}

\colorlet{darkgreen}{green!50!black}
\tikzset{dots/.style={very thick,loosely dotted},
         greendot/.style={fill,circle,color=darkgreen,inner sep=1.5pt,outer sep=0},
         blackdot/.style={fill,circle,color=black,inner sep=2pt,outer sep=0},
         graydot/.style={fill,circle,color=gray,inner sep=1.1pt,outer sep=0}
}
\def\greendot(#1,#2){\node[greendot] at(#1,#2){}}
\def\blackdot(#1,#2){\node[blackdot] at(#1,#2){}}
\def\graydot(#1,#2){\node[graydot] at(#1,#2){}}

\def\Grid(#1,#2){
  \draw[very thin,gray,step=2mm] (0,0)grid(#1,#2);
  \draw[very thin,darkgreen,step=10mm] (0,0)grid(#1,#2);
}

\begin{document}


\title[Cuspidal ribbon tableaux in affine type A]{{\bf Cuspidal ribbon tableaux in affine type A}}

\author{\sc Dina Abbasian}
\address{Washington \& Jefferson College\\ Washington\\ PA~15301, USA}
\email{abbasiand@washjeff.edu}

\author{\sc Lena Difulvio}
\address{Washington \& Jefferson College\\ Washington\\ PA~15301, USA}
\email{difulviole@washjeff.edu}

\author{\sc Robert Muth}
\address{Dept. of Mathematics\\ Washington \& Jefferson College\\ Washington\\ PA~15301, USA}
\email{rmuth@washjeff.edu}

\author{\sc Gabrielle Pasternak}
\address{Washington \& Jefferson College\\ Washington\\ PA~15301, USA}
\email{pasternakgm@washjeff.edu}

\author{\sc Isabella Sholtes}
\address{Washington \& Jefferson College\\ Washington\\ PA~15301, USA}
\email{sholtesim@washjeff.edu}

\author{\sc Frances Sinclair}
\address{Washington \& Jefferson College\\ Washington\\ PA~15301, USA}
\email{fiee9808@gmail.com}

\renewcommand{\shortauthors}{D. Abbasian, L. Difulvio, R. Muth, G. Pasternak, I. Sholtes, and F. Sinclair}
\renewcommand{\shorttitle}{Cuspidal ribbon tableaux in affine type A}



\begin{abstract}
For any convex preorder on the set of positive roots of affine type A, we classify and construct all associated cuspidal and semicuspidal skew shapes. These combinatorial objects correspond to cuspidal and semicuspidal skew Specht modules for the Khovanov-Lauda-Rouquier algebra of affine type A. Cuspidal skew shapes are ribbons, and we show that every skew shape has a unique ordered tiling by cuspidal ribbons. This tiling data provides an upper bound, in the bilexicographic order on Kostant partitions, for labels of simple factors of Specht modules.
\end{abstract}

\maketitle

\section{Introduction}

We begin by briefly describing the combinatorial data studied herein, referring the reader to the body of the paper for detailed exposition and to Example~\ref{BigEx} for demonstrative visuals.

Fix \(e > 1\), and let \(\Phi_+ = \Phi_+^\re \sqcup \Phi_+^\im\) be the positive root system of type \({\tt A}^{(1)}_{e-1}\); \(I = \{\alpha_0, \ldots, \alpha_{e-1}\}\) be the set of simple roots; \(Q_+ = \ZZ_{\geq 0}I\) be the root lattice;  \(\Phi_+^\re\) be the set of real roots; \(\Phi_+^\im = \{m \delta \mid m \in \NN\}\) be the set of imaginary roots, and \(\delta = \alpha_0 + \cdots + \alpha_{e-1}\) be the null root. We fix a convex preorder \(\succeq\) on \(\Phi_+\), see \S\ref{posrootsec}. This root system data is fundamental in the representation theory of the Kac-Moody Lie algebra \(\widehat{\mathfrak{sl}}_e\) \cite{Kac}, and convex preorders determine PBW bases for the associated quantum group \cite{Beck}.

A skew shape \(\tau\) is a set difference of Young diagrams. Nodes in \(\tau\) have residue in \(\ZZ_e\), and the skew shape \(\tau\) has content \(\cont(\tau) \in Q_+\), see \S\ref{skewshapesec}, \S\ref{contentsec}.
We say a set \(\Lambda\) of non-overlapping skew shapes is a tiling of a skew shape \(\tau\) provided that \(\tau = \bigsqcup_{\lambda \in \Lambda}\lambda\), and we refer to elements of \(\Lambda\) as tiles. A \(\Lambda\)-tableau \(\ttt = (\lambda_1, \ldots, \lambda_{|\Lambda|})\)  is an ordering of the tiles in \(\Lambda\) such that no node in \(\lambda_i\) is (weakly) southeast of any node in \(\lambda_j\) when \(1 \leq i < j \leq |\Lambda|\), see \S\ref{tilingsec}. This is a generalization of the notion of {\em Young} tableaux, in which each \(\lambda_i\) consists of a single node. Young tableaux and their associated residue sequences correspond to bases and associated weight spaces for Specht modules over cyclotomic Hecke algebras, and the symmetric group in particular, see \cite{Mathasbook, Kbook}.

Of particular interest in this paper are tilings whose constituent tiles are ribbons---connected skew shapes that contain no \(2 \times 2\) boxes.
Ribbon tilings and tableaux are well-studied combinatorial objects connected to many areas of algebra and geometry, such as the theory of symmetric functions, the representation theory of finite groups, lie algebras and quantum groups, and the geometry of flag varieties, see \cite{FS, LLT, Shef} for a few examples. The majority of research done on this topic concerns \(r\)-ribbon tableaux, in which all tiles are ribbons of a given cardinality \(r\). We do not include that restriction here, as we are interested in {\em cuspidal} ribbons, which have varying and unbounded cardinality.

\subsection{Cuspidality}\label{cuspint}
Motivated by the representation theory of Khovanov-Lauda-Rouquier (KLR) algebras studied in \cite{KCusp, KM, McN}, we introduce the notion of {\em cuspidal} and {\em semicuspidal} skew shapes. Let \(\tau\) be a skew shape such that \(\cont(\tau) = m \beta\) for some \(m \in \NN\) and \(\beta \in \Phi_+\). We say that \(\tau\) is semicuspidal provided that whenever \((\lambda_1, \lambda_2)\) is a tableau for \(\tau\), we may write \(\cont(\lambda_1)\) as a sum of positive roots \(\preceq \beta\), and \(\cont(\lambda_2)\) as a sum of positive roots \(\succeq \beta\). We say that \(\tau\) is cuspidal provided that \(m=1\) and the comparisons above may be made strict, see \S\ref{cuspdef}. 
We give a complete classification of cuspidal and semicuspidal skew shapes in Theorems~\ref{allcusp} and~\ref{allsemicusp}, summarized (in a slightly weaker form) as Theorems~A and~B below.

\begin{thmA}\label{THMA}
Every cuspidal skew shape is a ribbon.
There exists a unique cuspidal ribbon \(\zeta^\beta\) of content \(\beta\) for all \(\beta \in \Phi_+^\re\), and \(e\) distinct cuspidal ribbons \(\zeta^{\overline{0}}, \ldots, \zeta^{\overline{e-1}}\) of content \(\delta\). 
\end{thmA}

\begin{thmB}\label{THMB}
Let \(m \in \NN\).
There exists a unique semicuspidal skew shape \(\zeta^{m \beta}\) of content \(m \beta\), for all \(\beta \in \Phi_+^\re\). The set of connected semicuspidal skew shapes of content \(m \delta\) is in bijection with \(\ZZ_e \times \SSSS_{\tt c}(m)\), where \(\SSSS_{\tt c}(m)\) is the set of connected skew shapes of cardinality \(m\). 
\end{thmB}

The uniqueness statements in Theorems~A and~B refer to uniqueness up to certain residue-preserving translations, see \S\ref{similaritysec}. In \S\ref{cuspcons} and \S\ref{semicuspcons} we provide explicit constructions of all cuspidal and semicuspidal skew shapes. Imaginary semicuspidal skew shapes are constructed via an \(e\)-dilation process which is in some sense an inversion of the \(e\)-quotients defined in \cite{Lit}.

\subsection{Kostant tilings}\label{Kosint}
Let \(\Lambda\) be a tiling for \(\tau\). We say a \(\Lambda\)-tableau \(\ttt = (\lambda_1, \ldots, \lambda_{|\Lambda|})\) is {\em Kostant} if there exist \(m_1, \ldots, m_k \in \NN\), \(\beta_1\succeq \ldots \succeq \beta_k \in \Phi_+\) such that \(\cont(\lambda_i) = m_i \beta_i\) for all \(1 \leq i \leq |\Lambda|\). We say it is {\em strict} Kostant if the comparisons above are strict. We say \(\Lambda\) is a {\em (strict) Kostant} tiling if a (strict) Kostant \(\Lambda\)-tableau exists.

If \(\cont(\tau) = \theta\), then a Kostant tiling \(\Lambda\) of \(\tau\) may naturally be associated with a Kostant partition \(\bkap^\Lambda\) of \(\theta\), see \S\ref{Kostsec}. The convex preorder \(\succeq\) naturally induces a bilexicographic partial order \(\trianglerighteq\) on the set \(\Xi(\theta)\) of Kostant partitions of \(\theta\). Our main theorem on Kostant tilings is the following, which appears as Theorems~\ref{mainthmcusp} and~\ref{mainthmscKost} in the text:

\begin{thmC}\label{THMC}
Let \(\tau\) be a nonempty skew shape.
\begin{enumerate}
\item There exists a unique cuspidal Kostant tiling \(\Gamma_\tau\) for \(\tau\). 
\item There exists a unique semicuspidal strict Kostant tiling \(\Gamma^{sc}_\tau\) for \(\tau\).
\item If \(\Lambda\) is any Kostant tiling for \(\tau\), then we have \( \bkap^{\Lambda} \trianglelefteq \bkap^{\Gamma_\tau} = \bkap^{\Gamma^{sc}_\tau} \).
\end{enumerate}
\end{thmC}

We establish existence in Theorem~C(i),(ii) constructively;
in \S\ref{resKos} we show the cuspidal Kostant tiling \(\Gamma_\tau\) may be constructed via progressive removal of {\em minimal ribbons}.

\subsection{Application to representation theory} 

The combinatorial study of cuspidality and Kostant tilings for skew shapes discussed in \S\ref{cuspint}, \S\ref{Kosint} is motivated by a connection to the theory of cuspidal systems and Specht modules over Khovanov-Lauda-Rouquier (KLR) algebras, and has application in that setting.

For any field \(\k\) and \(\theta \in Q_+\), there is an associated KLR \(\k\)-algebra \(R_\theta\). This family of algebras categorifies the positive part of the quantum group \(U_q(\widehat{\mathfrak{sl}}_e)\), see \cite{KL1, KL2, Rouq}. Associated to any skew shape \(\tau\) of content \(\theta\) is a {\em (skew) Specht} \(R_\theta\)-module \(S^\tau\), as defined in \cite{KMR, Muth}. As discussed in \S\ref{Spechtmodsec}, these Specht modules are key objects in the representation theory of cyclotomic KLR algebras, Hecke algebras and symmetric groups, via the connection between these objects proved in \cite{BK}.

Under some restrictions on the ground field characteristic (see \cite{KM}), the category of finitely generated \(R_\theta\)-modules is properly stratified, with strata labeled by \(\Xi(\theta)\) and with the simple modules labeled \(L(\bkap, \blam)\), where \(\bkap \in \Xi(\theta)\) and \(\blam\) is an \((e-1)\)-multipartition of the coefficient of \(\delta\) in the Kostant partition \(\bkap\). {\em Cuspidal} and {\em semicuspidal} \(R_\beta\)-modules associated to every \(\beta \in \Phi_+\) are the building blocks of this stratification theory, see \cite{KCusp, McN, KM}.

An interesting objective is to find a combinatorial rule connecting the {\em cellular} structure in cyclotomic KLR algebra representation theory---built on Specht modules with simple modules labeled via multipartitions---to the {\em stratified} representation theory of the affine KLR algebra---built on cuspidal modules with simple modules labeled via Kostant partitions. Theorems~A, B, C describe a rough step in this direction. 
In Proposition~\ref{cuspiscusp}, we show that Theorems~A and~B give a complete classification of all cuspidal and semicuspidal Specht modules over KLR algebras. In Proposition~\ref{propcuspSpechtsimp} we leverage this connection to give a presentation of all simple cuspidal and semicuspidal modules associated to real positive roots. This generalizes a result proved in \cite{Muth} from balanced convex preorders to arbitrary convex preorders. In Proposition~\ref{simplefactors} we show that Theorem~C provides a tight upper bound, in the bilexicographic order on \(\Xi(\theta)\), on the simple factors which occur in the skew Specht module \(S^\tau\).

\begin{example}\label{BigEx}
We conclude the introduction with a demonstrative example of the combinatorial objects studied in this paper.
Take \(e=3\). Following \cite[Example 3.6]{McN}, we may define a convex preorder on \(\Phi_+\) as follows. Set a total order on \(\QQ^2\) via 
\begin{align*}
(x,y) \geq (x',y') \qquad \iff \qquad x>x' \textup{ or } x=x' \textup{ and }y\geq y',
\end{align*}
for all \((x,y), (x',y') \in \QQ^2\). 
Then define a map \(h: \Phi_+ \to \QQ^2\) by setting
\begin{align*}
h(\alpha_0) = (2,1),
\qquad
h(\alpha_1) = (-1,0),
\qquad
h(\alpha_2) = (-1,-1),
\end{align*}
and extend by \(\ZZ\)-linearity to all of \(\Phi_+\). For \(\beta, \gamma \in \Phi_+\), set 
\begin{align*}
\beta \succeq \gamma \qquad \iff \qquad
\frac{h(\beta)}{\height(\beta)} \geq \frac{h(\gamma)}{\height(\gamma)},
\end{align*}
where \(\height(\beta)\) is the {\em height} of \(\beta\), the sum of the coefficients of simple roots in \(\beta\).

At the maximum end of the preorder \(\succeq\), we have the real positive roots:
\begin{align*}
\alpha_0 \succ \alpha_0 + \alpha_1 \succ \delta + \alpha_0 \succ \alpha_2 + \alpha_0 \succ 2 \delta + \alpha_0 \succ \delta + \alpha_0 + \alpha_1 \succ \cdots
\end{align*}
Cuspidal ribbons \(\zeta^\beta\) associated to these roots, as constructed in \S\ref{cuspcons}, appear in Figure~\ref{fig:BigEx1}, with residues depicted within each node.

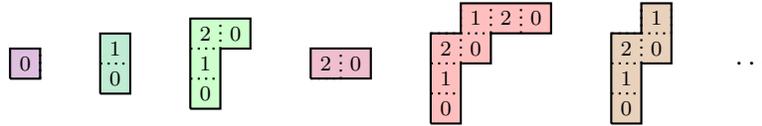
\begin{figure}[h]
\begin{align*}
{}
\hackcenter{
\begin{tikzpicture}[scale=0.4]
%
\draw[thick, fill=violet!25]  (0,0)--(1,0)--(1,1)--(0,1)--(0,0);
%
\draw[thick, dotted] (1,0)--(1,1);
%
%
%
\node at (0.5,0.5){$\scriptstyle0 $};
\end{tikzpicture}
}
\qquad
\hackcenter{
\begin{tikzpicture}[scale=0.4]
%
\draw[thick, fill=blue!30!green!25]  (0,0)--(1,0)--(1,2)--(0,2)--(0,0);
%
\draw[thick, dotted] (0,1)--(1,1);
%
%
%
\node at (0.5,0.5){$\scriptstyle0 $};
\node at (0.5,1.5){$\scriptstyle1 $};
\end{tikzpicture}
}
\qquad
\hackcenter{
\begin{tikzpicture}[scale=0.4]
%
\draw[thick, fill=green!20]  (0,-1)--(1,-1)--(1,1)--(2,1)--(2,2)--(0,2)--(0,-1);
%
\draw[thick, dotted] (0,1)--(1,1)--(1,2);
\draw[thick, dotted] (0,0)--(1,0);
%
%
%
\node at (0.5,-0.5){$\scriptstyle0 $};
\node at (0.5,0.5){$\scriptstyle1 $};
\node at (0.5,1.5){$\scriptstyle2 $};
\node at (1.5,1.5){$\scriptstyle0$};
\end{tikzpicture}
}
\qquad
\hackcenter{
\begin{tikzpicture}[scale=0.4]
%
\draw[thick, fill=purple!25]  (0,1)--(2,1)--(2,2)--(0,2)--(0,1);
%
\draw[thick, dotted] (0,1)--(1,1)--(1,2);
%
%
%
\node at (0.5,1.5){$\scriptstyle2 $};
\node at (1.5,1.5){$\scriptstyle0 $};
\end{tikzpicture}
}
\qquad
\hackcenter{
\begin{tikzpicture}[scale=0.4]
%
\draw[thick, fill=pink]  (0,0)--(1,0)--(1,2)--(2,2)--(2,3)--(4,3)--(4,4)--(1,4)--(1,3)--(0,3)--(0,0);
%
\draw[thick, dotted] (0,1)--(1,1);
\draw[thick, dotted] (0,2)--(1,2);
\draw[thick, dotted] (0,3)--(2,3);
\draw[thick, dotted] (1,2)--(1,4);
\draw[thick, dotted] (2,3)--(2,4);
\draw[thick, dotted] (3,3)--(3,4);
%
%
%
\node at (0.5,0.5){$\scriptstyle0 $};
\node at (0.5,1.5){$\scriptstyle1 $};
\node at (0.5,2.5){$\scriptstyle2 $};
\node at (1.5,2.5){$\scriptstyle0 $};
\node at (1.5,3.5){$\scriptstyle1 $};
\node at (2.5,3.5){$\scriptstyle2 $};
\node at (3.5,3.5){$\scriptstyle0 $};
\end{tikzpicture}
}
\qquad
\hackcenter{
\begin{tikzpicture}[scale=0.4]
%
\draw[thick,  fill=brown!35]  (0,0)--(1,0)--(1,2)--(2,2)--(2,4)--(1,4)--(1,3)--(0,3)--(0,0);
%
\draw[thick, dotted] (0,1)--(1,1);
\draw[thick, dotted] (0,2)--(1,2);
\draw[thick, dotted] (0,3)--(2,3);
\draw[thick, dotted] (1,2)--(1,4);
\draw[thick, dotted] (2,3)--(2,4);
%
%
%
\node at (0.5,0.5){$\scriptstyle0 $};
\node at (0.5,1.5){$\scriptstyle1 $};
\node at (0.5,2.5){$\scriptstyle2 $};
\node at (1.5,2.5){$\scriptstyle0 $};
\node at (1.5,3.5){$\scriptstyle1 $};
\end{tikzpicture}
}
\qquad
\cdots
\end{align*}
\caption{Real cuspidal ribbons \(\zeta^{\alpha_0}; \zeta^{\alpha_0 + \alpha_1}; 
\zeta^{\delta + \alpha_0}; \zeta^{\alpha_2 + \alpha_0}; \zeta^{2 \delta + \alpha_0}; \zeta^{\delta + \alpha_0 + \alpha_1}
\)}
\label{fig:BigEx1}       
\end{figure}

At the minimum end of the preorder \(\succeq\), we have the real positive roots:
\begin{align*}
\cdots \succ 2 \delta + \alpha_1 + \alpha_2 \succ \delta+\alpha_2 \succ \delta + \alpha_1 + \alpha_2 \succ \alpha_1 \succ \alpha_1 + \alpha_2 \succ \alpha_2.
\end{align*}
Cuspidal ribbons \(\zeta^\beta\) associated to these roots appear in Figure~\ref{fig:BigEx2}. 
\begin{figure}[h]
\begin{align*}
{}
\cdots
\qquad
\hackcenter{
\begin{tikzpicture}[scale=0.4]
%
\draw[thick, fill=cyan!30] (0,0)--(3,0)--(3,2)--(4,2)--(4,5)--(3,5)--(3,3)--(2,3)--(2,1)--(0,1)--(0,0);
%
\draw[thick, dotted] (1,0)--(1,1);
\draw[thick, dotted] (2,0)--(2,1);
\draw[thick, dotted] (3,2)--(3,3);
\draw[thick, dotted] (2,1)--(3,1);
\draw[thick, dotted] (2,2)--(3,2);
\draw[thick, dotted] (3,3)--(4,3);
\draw[thick, dotted] (3,4)--(4,4);
%
%
%
\node at (0.5,0.5){$\scriptstyle1 $};
\node at (1.5,0.5){$\scriptstyle2 $};
\node at (2.5,0.5){$\scriptstyle0 $};
\node at (2.5,1.5){$\scriptstyle1 $};
\node at (2.5,2.5){$\scriptstyle2 $};
\node at (3.5,2.5){$\scriptstyle0 $};
\node at (3.5,3.5){$\scriptstyle1 $};
\node at (3.5,4.5){$\scriptstyle2 $};
\end{tikzpicture}
}
\qquad
\hackcenter{
\begin{tikzpicture}[scale=0.4]
%
\draw[thick, fill=purple!35] (2,2)--(4,2)--(4,5)--(3,5)--(3,3)--(2,3)--(2,2);
%
\draw[thick, dotted] (3,2)--(3,3);
\draw[thick, dotted] (3,3)--(4,3);
\draw[thick, dotted] (3,4)--(4,4);
%
%
%
\node at (2.5,2.5){$\scriptstyle2 $};
\node at (3.5,2.5){$\scriptstyle0 $};
\node at (3.5,3.5){$\scriptstyle1 $};
\node at (3.5,4.5){$\scriptstyle2 $};
\end{tikzpicture}
}
\qquad
\hackcenter{
\begin{tikzpicture}[scale=0.4]
%
\draw[thick, fill=orange!50] (1,2)--(4,2)--(4,5)--(3,5)--(3,3)--(1,3)--(1,2);
%
\draw[thick, dotted] (2,2)--(2,3);
\draw[thick, dotted] (3,2)--(3,3);
\draw[thick, dotted] (3,3)--(4,3);
\draw[thick, dotted] (3,4)--(4,4);
%
%
%
\node at (1.5,2.5){$\scriptstyle1 $};
\node at (2.5,2.5){$\scriptstyle2 $};
\node at (3.5,2.5){$\scriptstyle0 $};
\node at (3.5,3.5){$\scriptstyle1 $};
\node at (3.5,4.5){$\scriptstyle2 $};
\end{tikzpicture}
}
\qquad
\hackcenter{
\begin{tikzpicture}[scale=0.4]
%
\draw[thick, fill=yellow!40!orange!50]  (0,0)--(1,0)--(1,1)--(0,1)--(0,0);
%
\draw[thick, dotted] (1,0)--(1,1);
%
%
%
\node at (0.5,0.5){$\scriptstyle1 $};
\end{tikzpicture}
}
\qquad
\hackcenter{
\begin{tikzpicture}[scale=0.4]
%
\draw[thick, fill=yellow]  (0,0)--(1,0)--(1,2)--(0,2)--(0,0);
%
\draw[thick, dotted] (0,1)--(1,1);
%
%
%
\node at (0.5,0.5){$\scriptstyle1 $};
\node at (0.5,1.5){$\scriptstyle2 $};
\end{tikzpicture}
}
\qquad
\hackcenter{
\begin{tikzpicture}[scale=0.4]
%
\draw[thick, fill=blue!30]  (0,0)--(1,0)--(1,1)--(0,1)--(0,0);
%
\draw[thick, dotted] (1,0)--(1,1);
%
%
%
\node at (0.5,0.5){$\scriptstyle2 $};
\end{tikzpicture}
}
\end{align*}
\caption{Real cuspidal ribbons \(
\zeta^{2 \delta + \alpha_1 + \alpha_2}; \zeta^{ \delta+\alpha_2}; \zeta^{\delta + \alpha_1 + \alpha_2}; \zeta^{ \alpha_1}; \zeta^{\alpha_1 + \alpha_2}; \zeta^{\alpha_2}
\)}
\label{fig:BigEx2}       
\end{figure}
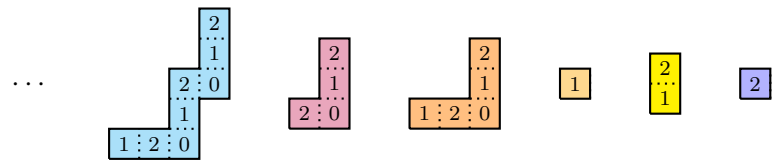

Smack in the middle of the preorder \(\succeq\) we have the imaginary roots \(m \delta\), all of which are equivalent under \(\succeq\). We have an imaginary cuspidal ribbon \(\zeta^{t}\) of content \(\delta\) associated to each \(t \in \ZZ_3\), as shown in Figure~\ref{fig:BigEx3}. 
\begin{figure}[h]
\begin{align*}
{}
\hackcenter{
\begin{tikzpicture}[scale=0.4]
%
\draw[thick, fill=lightgray!50]  (0,0)--(1,0)--(1,3)--(0,3)--(0,0);
%
\draw[thick, dotted] (0,1)--(1,1);
\draw[thick, dotted] (0,2)--(1,2);
%
%
%
\node at (0.5,0.5){$\scriptstyle0 $};
\node at (0.5,1.5){$\scriptstyle1 $};
\node at (0.5,2.5){$\scriptstyle2 $};
\end{tikzpicture}
}
\qquad
\hackcenter{
\begin{tikzpicture}[scale=0.4]
%
\draw[thick, fill=lightgray!50]  (0,0)--(3,0)--(3,1)--(0,1)--(0,0);
%
\draw[thick, dotted] (1,0)--(1,1);
\draw[thick, dotted] (2,0)--(2,1);
%
%
%
\node at (0.5,0.5){$\scriptstyle1 $};
\node at (1.5,0.5){$\scriptstyle2 $};
\node at (2.5,0.5){$\scriptstyle0 $};
\end{tikzpicture}
}
\qquad
\hackcenter{
\begin{tikzpicture}[scale=0.4]
%
\draw[thick, fill=lightgray!50] (0,0)--(2,0)--(2,2)--(1,2)--(1,1)--(0,1)--(0,0);
%
\draw[thick, dotted] (1,0)--(1,1);
\draw[thick, dotted] (2,0)--(2,1);
\draw[thick, dotted] (1,1)--(2,1);
%
%
%
\node at (0.5,0.5){$\scriptstyle2 $};
\node at (1.5,0.5){$\scriptstyle0 $};
\node at (1.5,1.5){$\scriptstyle1 $};
\end{tikzpicture}
}
\end{align*}
\caption{Imaginary cuspidal ribbons \(
\zeta^{\delta,\overline 0}; \zeta^{\delta, \overline 1}; \zeta^{\delta, \overline 2}
\)}
\label{fig:BigEx3}       
\end{figure}
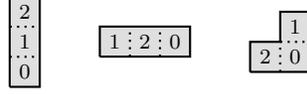

If \(\beta \in \Phi_+^\re\), then the unique semicuspidal skew shape of content \(m \beta\) has \(m\) connected components, each of which is identical to the ribbon \(\zeta^\beta\), per Theorem~\ref{allsemicusp}. On the other hand, semicuspidal connected skew shapes of content \(m \delta\) are in bijection with \(\ZZ_3 \times \SSSS_{\tt c}(m)\) via the dilation map, per Proposition~\ref{conissemicusp}. Roughly speaking, for \(t \in \ZZ_3\), one `\(t\)-dilates' the skew shape \(\lambda \in \SSSS_{\tt c}(m)\) by replacing every node in \(\lambda\) with the cuspidal ribbon \(\zeta^{t}\), see \S\ref{infldef}. In Figure~\ref{fig:Infl} an example of this dilation process is depicted.
\begin{figure}[h]
\begin{align*}
{}
\hackcenter{
\begin{tikzpicture}[scale=0.4]
%
\draw[thick, fill=lightgray!50]  (0,0)--(2,0)--(2,1)--(3,1)--(3,2)--(4,2)--(4,3)--(0,3)--(0,0);
%
\draw[thick, dotted] (0,1)--(2,1);
\draw[thick, dotted] (0,2)--(3,2);
\draw[thick, dotted] (1,0)--(1,3);
\draw[thick, dotted] (2,0)--(2,3);
\draw[thick, dotted] (3,1)--(3,3);
%
%
%
\end{tikzpicture}
}
\;\;\;\;\;
\xrightarrow{\;\;\infl_{\overline 2}\;\;}
\;
\hackcenter{
\begin{tikzpicture}[scale=0.4]
%
\pgfmathsetmacro{\ex}{0}
\pgfmathsetmacro{\ey}{0}
\draw[thick, fill=lightgray!50]  (\ex,\ey)--(\ex+2,\ey)--(\ex+2,\ey+2)--(\ex+1,\ey+2)--(\ex+1,\ey+1)--(\ex,\ey+1)--(\ex,\ey);
\draw[thick, dotted] (\ex+1,\ey)--(\ex+1,\ey+1);
\draw[thick, dotted] (\ex+1,\ey+1)--(\ex+2,\ey+1);
\node at (\ex+0.5,\ey+0.5){$\scriptstyle2 $};
\node at (\ex+1.5,\ey+0.5){$\scriptstyle0 $};
\node at (\ex+1.5,\ey+1.5){$\scriptstyle1 $};
\pgfmathsetmacro{\ex}{2}
\pgfmathsetmacro{\ey}{1}
\draw[thick, fill=lightgray!50]  (\ex,\ey)--(\ex+2,\ey)--(\ex+2,\ey+2)--(\ex+1,\ey+2)--(\ex+1,\ey+1)--(\ex,\ey+1)--(\ex,\ey);
\draw[thick, dotted] (\ex+1,\ey)--(\ex+1,\ey+1);
\draw[thick, dotted] (\ex+1,\ey+1)--(\ex+2,\ey+1);
\node at (\ex+0.5,\ey+0.5){$\scriptstyle2 $};
\node at (\ex+1.5,\ey+0.5){$\scriptstyle0 $};
\node at (\ex+1.5,\ey+1.5){$\scriptstyle1 $};
\pgfmathsetmacro{\ex}{1}
\pgfmathsetmacro{\ey}{2}
\draw[thick, fill=lightgray!50]  (\ex,\ey)--(\ex+2,\ey)--(\ex+2,\ey+2)--(\ex+1,\ey+2)--(\ex+1,\ey+1)--(\ex,\ey+1)--(\ex,\ey);
\draw[thick, dotted] (\ex+1,\ey)--(\ex+1,\ey+1);
\draw[thick, dotted] (\ex+1,\ey+1)--(\ex+2,\ey+1);
\node at (\ex+0.5,\ey+0.5){$\scriptstyle2 $};
\node at (\ex+1.5,\ey+0.5){$\scriptstyle0 $};
\node at (\ex+1.5,\ey+1.5){$\scriptstyle1 $};
\pgfmathsetmacro{\ex}{2}
\pgfmathsetmacro{\ey}{4}
\draw[thick, fill=lightgray!50]  (\ex,\ey)--(\ex+2,\ey)--(\ex+2,\ey+2)--(\ex+1,\ey+2)--(\ex+1,\ey+1)--(\ex,\ey+1)--(\ex,\ey);
\draw[thick, dotted] (\ex+1,\ey)--(\ex+1,\ey+1);
\draw[thick, dotted] (\ex+1,\ey+1)--(\ex+2,\ey+1);
\node at (\ex+0.5,\ey+0.5){$\scriptstyle2 $};
\node at (\ex+1.5,\ey+0.5){$\scriptstyle0 $};
\node at (\ex+1.5,\ey+1.5){$\scriptstyle1 $};
\pgfmathsetmacro{\ex}{3}
\pgfmathsetmacro{\ey}{3}
\draw[thick, fill=lightgray!50]  (\ex,\ey)--(\ex+2,\ey)--(\ex+2,\ey+2)--(\ex+1,\ey+2)--(\ex+1,\ey+1)--(\ex,\ey+1)--(\ex,\ey);
\draw[thick, dotted] (\ex+1,\ey)--(\ex+1,\ey+1);
\draw[thick, dotted] (\ex+1,\ey+1)--(\ex+2,\ey+1);
\node at (\ex+0.5,\ey+0.5){$\scriptstyle2 $};
\node at (\ex+1.5,\ey+0.5){$\scriptstyle0 $};
\node at (\ex+1.5,\ey+1.5){$\scriptstyle1 $};
\pgfmathsetmacro{\ex}{4}
\pgfmathsetmacro{\ey}{5}
\draw[thick, fill=lightgray!50]  (\ex,\ey)--(\ex+2,\ey)--(\ex+2,\ey+2)--(\ex+1,\ey+2)--(\ex+1,\ey+1)--(\ex,\ey+1)--(\ex,\ey);
\draw[thick, dotted] (\ex+1,\ey)--(\ex+1,\ey+1);
\draw[thick, dotted] (\ex+1,\ey+1)--(\ex+2,\ey+1);
\node at (\ex+0.5,\ey+0.5){$\scriptstyle2 $};
\node at (\ex+1.5,\ey+0.5){$\scriptstyle0 $};
\node at (\ex+1.5,\ey+1.5){$\scriptstyle1 $};
\pgfmathsetmacro{\ex}{5}
\pgfmathsetmacro{\ey}{4}
\draw[thick, fill=lightgray!50]  (\ex,\ey)--(\ex+2,\ey)--(\ex+2,\ey+2)--(\ex+1,\ey+2)--(\ex+1,\ey+1)--(\ex,\ey+1)--(\ex,\ey);
\draw[thick, dotted] (\ex+1,\ey)--(\ex+1,\ey+1);
\draw[thick, dotted] (\ex+1,\ey+1)--(\ex+2,\ey+1);
\node at (\ex+0.5,\ey+0.5){$\scriptstyle2 $};
\node at (\ex+1.5,\ey+0.5){$\scriptstyle0 $};
\node at (\ex+1.5,\ey+1.5){$\scriptstyle1 $};
\pgfmathsetmacro{\ex}{6}
\pgfmathsetmacro{\ey}{6}
\draw[thick, fill=lightgray!50]  (\ex,\ey)--(\ex+2,\ey)--(\ex+2,\ey+2)--(\ex+1,\ey+2)--(\ex+1,\ey+1)--(\ex,\ey+1)--(\ex,\ey);
\draw[thick, dotted] (\ex+1,\ey)--(\ex+1,\ey+1);
\draw[thick, dotted] (\ex+1,\ey+1)--(\ex+2,\ey+1);
\node at (\ex+0.5,\ey+0.5){$\scriptstyle2 $};
\node at (\ex+1.5,\ey+0.5){$\scriptstyle0 $};
\node at (\ex+1.5,\ey+1.5){$\scriptstyle1 $};
\pgfmathsetmacro{\ex}{8}
\pgfmathsetmacro{\ey}{7}
\draw[thick, fill=lightgray!50]  (\ex,\ey)--(\ex+2,\ey)--(\ex+2,\ey+2)--(\ex+1,\ey+2)--(\ex+1,\ey+1)--(\ex,\ey+1)--(\ex,\ey);
\draw[thick, dotted] (\ex+1,\ey)--(\ex+1,\ey+1);
\draw[thick, dotted] (\ex+1,\ey+1)--(\ex+2,\ey+1);
\node at (\ex+0.5,\ey+0.5){$\scriptstyle2 $};
\node at (\ex+1.5,\ey+0.5){$\scriptstyle0 $};
\node at (\ex+1.5,\ey+1.5){$\scriptstyle1 $};
%
%
%
%
\end{tikzpicture}
}
\end{align*}
\caption{\(\overline 2\)-dilation of \(\lambda \in \SSSS_{\tt c}(9)\) into a semicuspidal skew shape \(\infl_{\overline 2}(\lambda)\) of content \(9\delta\)}
\label{fig:Infl}       
\end{figure}
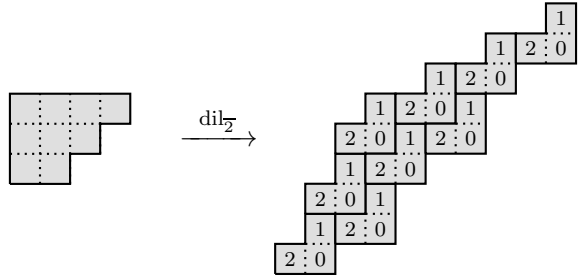

By Theorem~\ref{mainthmcusp}, every skew shape has a unique cuspidal Kostant tiling. In Figure~\ref{fig:Young3ways}, we take \(\tau_0, \tau_1, \tau_2\) to be the Young diagrams associated with the partition \((6,5^4,2^2,1)\) and {\em charges} \(0,1,2\) respectively (see \S\ref{otherformsec}).
The unique cuspidal Kostant tiling is shown for each. Per Theorem~\ref{mainthmscKost}, the unique semicuspidal strict Kostant tiling is achieved by choosing tiles to be unions of same-content cuspidal tiles.

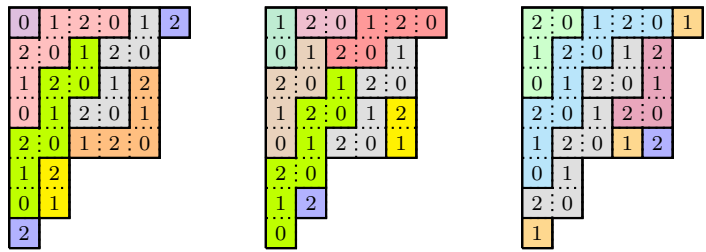
\begin{figure}[h]
\begin{align*}
{}
\hackcenter{
\begin{tikzpicture}[scale=0.4]
%
\draw[thick, fill=lightgray!50]  (0,0)--(1,0)--(1,1)--(2,1)--(2,3)--(5,3)--(5,7)--(6,7)--(6,8)--(0,8)--(0,0);
\draw[thick,  fill=blue!30]  (0,0)--(1,0)--(1,1)--(0,1)--(0,0);
\draw[thick, fill=blue!30]  (5,7)--(6,7)--(6,8)--(5,8)--(5,7);
\draw[thick, fill=lime]  (0,1)--(1,1)--(1,3)--(2,3)--(2,5)--(3,5)--(3,7)--(2,7)--(2,6)--(1,6)--(1,4)--(0,4)--(0,1);
\draw[thick, fill=yellow]  (1,1)--(2,1)--(2,3)--(1,3)--(1,1);
\draw[thick, fill=pink]  (0,4)--(1,4)--(1,6)--(2,6)--(2,7)--(4,7)--(4,8)--(1,8)--(1,7)--(0,7)--(0,4);
\draw[thick, fill=violet!25] (0,7)--(1,7)--(1,8)--(0,8)--(0,7);
\draw[thick, fill=orange!50] (2,3)--(5,3)--(5,6)--(4,6)--(4,4)--(2,4)--(2,3);
\draw[thick, fill=lightgray!50] (2,4)--(4,4)--(4,6)--(3,6)--(3,5)--(2,5)--(2,4);
\draw[thick, fill=lightgray!50] (3,6)--(5,6)--(5,8)--(4,8)--(4,7)--(3,7)--(3,6);
\draw[thick]  (0,0)--(1,0)--(1,1)--(2,1)--(2,3)--(5,3)--(5,7)--(6,7)--(6,8)--(0,8)--(0,0);
%
\draw[thick, dotted] (0,1)--(1,1);
\draw[thick, dotted] (0,2)--(2,2);
\draw[thick, dotted] (0,3)--(2,3);
\draw[thick, dotted] (0,4)--(5,4);
\draw[thick, dotted] (0,5)--(5,5);
\draw[thick, dotted] (0,6)--(5,6);
\draw[thick, dotted] (0,7)--(5,7);
\draw[thick, dotted] (1,0)--(1,8);
\draw[thick, dotted] (2,1)--(2,8);
\draw[thick, dotted] (3,3)--(3,8);
\draw[thick, dotted] (4,3)--(4,8);
\draw[thick, dotted] (5,3)--(5,8);
\node at (0.5,7.5){$\scriptstyle 0$};
\node at (1.5,7.5){$\scriptstyle 1$};
\node at (2.5,7.5){$\scriptstyle 2$};
\node at (3.5,7.5){$\scriptstyle 0$};
\node at (4.5,7.5){$\scriptstyle 1$};
\node at (5.5,7.5){$\scriptstyle 2$};
\node at (0.5,6.5){$\scriptstyle 2$};
\node at (1.5,6.5){$\scriptstyle 0$};
\node at (2.5,6.5){$\scriptstyle 1$};
\node at (3.5,6.5){$\scriptstyle 2$};
\node at (4.5,6.5){$\scriptstyle 0$};
\node at (0.5,5.5){$\scriptstyle 1$};
\node at (1.5,5.5){$\scriptstyle 2$};
\node at (2.5,5.5){$\scriptstyle 0$};
\node at (3.5,5.5){$\scriptstyle 1$};
\node at (4.5,5.5){$\scriptstyle 2$};
\node at (0.5,4.5){$\scriptstyle 0$};
\node at (1.5,4.5){$\scriptstyle 1$};
\node at (2.5,4.5){$\scriptstyle 2$};
\node at (3.5,4.5){$\scriptstyle 0$};
\node at (4.5,4.5){$\scriptstyle 1$};
\node at (0.5,3.5){$\scriptstyle 2$};
\node at (1.5,3.5){$\scriptstyle 0$};
\node at (2.5,3.5){$\scriptstyle 1$};
\node at (3.5,3.5){$\scriptstyle 2$};
\node at (4.5,3.5){$\scriptstyle 0$};
\node at (0.5,2.5){$\scriptstyle 1$};
\node at (1.5,2.5){$\scriptstyle 2$};
\node at (0.5,1.5){$\scriptstyle 0$};
\node at (1.5,1.5){$\scriptstyle 1$};
\node at (0.5,0.5){$\scriptstyle 2$};
\end{tikzpicture}
}
\qquad
\;\;
\hackcenter{
\begin{tikzpicture}[scale=0.4]
%
\draw[thick, fill=red!40]  (0,0)--(1,0)--(1,1)--(2,1)--(2,3)--(5,3)--(5,7)--(6,7)--(6,8)--(0,8)--(0,0);
\draw[thick, fill=lime]  (0,0)--(1,0)--(1,2)--(2,2)--(2,4)--(3,4)--(3,6)--(2,6)--(2,5)--(1,5)--(1,3)--(0,3)--(0,0);
\draw[thick, fill=blue!30] (1,1)--(2,1)--(2,2)--(1,2)--(1,1);
\draw[thick, fill=lightgray!50] (2,3)--(4,3)--(4,5)--(3,5)--(3,4)--(2,4)--(2,3);
\draw[thick, fill=lightgray!50] (3,5)--(5,5)--(5,7)--(4,7)--(4,6)--(3,6)--(3,5);
\draw[thick, fill=yellow] (4,3)--(5,3)--(5,5)--(4,5)--(4,3);
\draw[thick,  fill=brown!35] (0,3)--(1,3)--(1,5)--(2,5)--(2,7)--(1,7)--(1,6)--(0,6)--(0,3);
\draw[thick, fill=blue!30!green!25] (0,6)--(1,6)--(1,8)--(0,8)--(0,6);
\draw[thick, fill=purple!25] (1,7)--(3,7)--(3,8)--(1,8)--(1,7);
\draw[thick]  (0,0)--(1,0)--(1,1)--(2,1)--(2,3)--(5,3)--(5,7)--(6,7)--(6,8)--(0,8)--(0,0);
%
\draw[thick, dotted] (0,1)--(1,1);
\draw[thick, dotted] (0,2)--(2,2);
\draw[thick, dotted] (0,3)--(2,3);
\draw[thick, dotted] (0,4)--(5,4);
\draw[thick, dotted] (0,5)--(5,5);
\draw[thick, dotted] (0,6)--(5,6);
\draw[thick, dotted] (0,7)--(5,7);
\draw[thick, dotted] (1,0)--(1,8);
\draw[thick, dotted] (2,1)--(2,8);
\draw[thick, dotted] (3,3)--(3,8);
\draw[thick, dotted] (4,3)--(4,8);
\draw[thick, dotted] (5,3)--(5,8);
\node at (0.5,7.5){$\scriptstyle 1$};
\node at (1.5,7.5){$\scriptstyle 2$};
\node at (2.5,7.5){$\scriptstyle 0$};
\node at (3.5,7.5){$\scriptstyle 1$};
\node at (4.5,7.5){$\scriptstyle 2$};
\node at (5.5,7.5){$\scriptstyle 0$};
\node at (0.5,6.5){$\scriptstyle 0$};
\node at (1.5,6.5){$\scriptstyle 1$};
\node at (2.5,6.5){$\scriptstyle 2$};
\node at (3.5,6.5){$\scriptstyle 0$};
\node at (4.5,6.5){$\scriptstyle 1$};
\node at (0.5,5.5){$\scriptstyle 2$};
\node at (1.5,5.5){$\scriptstyle 0$};
\node at (2.5,5.5){$\scriptstyle 1$};
\node at (3.5,5.5){$\scriptstyle 2$};
\node at (4.5,5.5){$\scriptstyle 0$};
\node at (0.5,4.5){$\scriptstyle 1$};
\node at (1.5,4.5){$\scriptstyle 2$};
\node at (2.5,4.5){$\scriptstyle 0$};
\node at (3.5,4.5){$\scriptstyle 1$};
\node at (4.5,4.5){$\scriptstyle 2$};
\node at (0.5,3.5){$\scriptstyle 0$};
\node at (1.5,3.5){$\scriptstyle 1$};
\node at (2.5,3.5){$\scriptstyle 2$};
\node at (3.5,3.5){$\scriptstyle 0$};
\node at (4.5,3.5){$\scriptstyle 1$};
\node at (0.5,2.5){$\scriptstyle 2$};
\node at (1.5,2.5){$\scriptstyle 0$};
\node at (0.5,1.5){$\scriptstyle 1$};
\node at (1.5,1.5){$\scriptstyle 2$};
\node at (0.5,0.5){$\scriptstyle 0$};
\end{tikzpicture}
}
\qquad
\;\;
\hackcenter{
\begin{tikzpicture}[scale=0.4]
%
\draw[thick,fill=cyan!25]  (0,0)--(1,0)--(1,1)--(2,1)--(2,3)--(5,3)--(5,7)--(6,7)--(6,8)--(0,8)--(0,0);
\draw[thick, fill=lightgray!50] (0,1)--(2,1)--(2,3)--(1,3)--(1,2)--(0,2)--(0,1);
\draw[thick, fill=lightgray!50] (1,3)--(3,3)--(3,5)--(2,5)--(2,4)--(1,4)--(1,3);
\draw[thick, fill=lightgray!50] (2,5)--(4,5)--(4,7)--(3,7)--(3,6)--(2,6)--(2,5);
\draw[thick, fill=green!20] (0,5)--(1,5)--(1,7)--(2,7)--(2,8)--(0,8)--(0,5);
\draw[thick, fill=purple!35] (3,4)--(5,4)--(5,7)--(4,7)--(4,5)--(3,5)--(3,4);
\draw[thick, fill=yellow!40!orange!50] (0,0)--(1,0)--(1,1)--(0,1)--(0,0);
\draw[thick,  fill=yellow!40!orange!50] (5,7)--(6,7)--(6,8)--(5,8)--(5,7);
\draw[thick,  fill=yellow!40!orange!50] (3,3)--(4,3)--(4,4)--(3,4)--(3,3);
\draw[thick, fill=blue!30] (4,3)--(5,3)--(5,4)--(4,4)--(4,3);
\draw[thick]  (0,0)--(1,0)--(1,1)--(2,1)--(2,3)--(5,3)--(5,7)--(6,7)--(6,8)--(0,8)--(0,0);
%
\draw[thick, dotted] (0,1)--(1,1);
\draw[thick, dotted] (0,2)--(2,2);
\draw[thick, dotted] (0,3)--(2,3);
\draw[thick, dotted] (0,4)--(5,4);
\draw[thick, dotted] (0,5)--(5,5);
\draw[thick, dotted] (0,6)--(5,6);
\draw[thick, dotted] (0,7)--(5,7);
\draw[thick, dotted] (1,0)--(1,8);
\draw[thick, dotted] (2,1)--(2,8);
\draw[thick, dotted] (3,3)--(3,8);
\draw[thick, dotted] (4,3)--(4,8);
\draw[thick, dotted] (5,3)--(5,8);
\node at (0.5,7.5){$\scriptstyle 2$};
\node at (1.5,7.5){$\scriptstyle 0$};
\node at (2.5,7.5){$\scriptstyle 1$};
\node at (3.5,7.5){$\scriptstyle 2$};
\node at (4.5,7.5){$\scriptstyle 0$};
\node at (5.5,7.5){$\scriptstyle 1$};
\node at (0.5,6.5){$\scriptstyle 1$};
\node at (1.5,6.5){$\scriptstyle 2$};
\node at (2.5,6.5){$\scriptstyle 0$};
\node at (3.5,6.5){$\scriptstyle 1$};
\node at (4.5,6.5){$\scriptstyle 2$};
\node at (0.5,5.5){$\scriptstyle 0$};
\node at (1.5,5.5){$\scriptstyle 1$};
\node at (2.5,5.5){$\scriptstyle 2$};
\node at (3.5,5.5){$\scriptstyle 0$};
\node at (4.5,5.5){$\scriptstyle 1$};
\node at (0.5,4.5){$\scriptstyle 2$};
\node at (1.5,4.5){$\scriptstyle 0$};
\node at (2.5,4.5){$\scriptstyle 1$};
\node at (3.5,4.5){$\scriptstyle 2$};
\node at (4.5,4.5){$\scriptstyle 0$};
\node at (0.5,3.5){$\scriptstyle 1$};
\node at (1.5,3.5){$\scriptstyle 2$};
\node at (2.5,3.5){$\scriptstyle 0$};
\node at (3.5,3.5){$\scriptstyle 1$};
\node at (4.5,3.5){$\scriptstyle 2$};
\node at (0.5,2.5){$\scriptstyle 0$};
\node at (1.5,2.5){$\scriptstyle 1$};
\node at (0.5,1.5){$\scriptstyle 2$};
\node at (1.5,1.5){$\scriptstyle 0$};
\node at (0.5,0.5){$\scriptstyle 1$};
\end{tikzpicture}
}
\end{align*}
\caption{Cuspidal Kostant tilings for \(\tau_0, \tau_1, \tau_2\)}
\label{fig:Young3ways}       
\end{figure}
The Kostant partitions associated with these tilings are:
\begin{align*}
\bkap^{\tau_0} &= (\alpha_0 \mid 2\delta+\alpha_0 \mid 2\delta + \alpha_0 + \alpha_1 \mid \delta^2 \mid \delta + \alpha_1 + \alpha_2 \mid \alpha_1 + \alpha_2 \mid \alpha_2^2);\\
\bkap^{\tau_1} &= (\alpha_0 + \alpha_1 \mid \alpha_2 + \alpha_0 \mid \delta + \alpha_0 + \alpha_1 \mid \delta + \alpha_2 + \alpha_0 \mid 2\delta + \alpha_0 + \alpha_1 \mid  \delta^2 \mid \alpha_1 + \alpha_2 \mid \alpha_2);\\
\bkap^{\tau_2} &= (\delta + \alpha_0 \mid 3\delta + \alpha_0 \mid  \delta^3 \mid \delta + \alpha_2 \mid \alpha_1^3 \mid \alpha_2).
\end{align*}
By Theorem~\ref{mainthmcusp}, these Kostant partitions are bilexicographically maximal amongst all Kostant partitions induced by Kostant tilings for \(\tau_0, \tau_1, \tau_2\), respectively. 
\end{example}

\section{Ribbons and skew shapes}\label{skewshapesec}
We introduce the main combinatorial objects in this section. Many results in this section are well known, but we include full proofs for clarity, as some of our terminology and setup is new. 
We denote the set of integers by \(\Z\), the natural numbers \(\NN = \{1,2,\ldots,\}\), and use the shorthand \([a,b] = \{a, a+1, \ldots, b\}\) for all \(a\leq b \in \Z\).

\subsection{Nodes}
Let \(\N = \Z \times \Z\). We refer to elements of \(\N\) as {\em nodes}. By convention, we visually represent nodes as boxes in a \(\Z \times \Z\) `matrix', so that the node \((i,j)\) is a box in the \(i\)th row and \(j\)th column of the array. In this orientation, positive increase in the first component corresponds to a southward move in the array, and positive increase in the second component corresponds to an eastward move in the array. For \(u \in \N\), we will write \(u_1, u_2\) for the components of \(u\), so that \(u = (u_1, u_2)\). We treat \(\N\) as a \(\ZZ\)-module, so that
\begin{align*}
n(u_1, u_2) = (nu_1, nu_2), \qquad \textup{ and } \qquad (u_1,u_2) +(v_1, v_2) = (u_1 + v_1, u_2 + v_2).
\end{align*}

\subsection{Translation}
For \(c \in \N\), we define the {\em translation by \(c\)} map \(\T_c:\N \to \N\) by \(\T_cu = 
u+ c\) for all \(u \in \N\). If \(\tau \subset \N\), we define \(\T_c \tau = \{\T_c u \mid u \in \tau\}\).
We define the special single-unit north, east, south and west translations \(\no, \ea, \so, \we\) by setting 
\begin{align*}
\no := \T_{(-1,0)},\qquad \ea := \T_{(0,1)},\qquad \so := \T_{(1,0)},\qquad \we  := \T_{(0,-1)}.
\end{align*}

\subsection{Relations between nodes}
We define transitive relations \(\searrow, \SEarrow, \nearrow, \NEarrow\) on \(\N\) by
\begin{align*}
u \searrow v &\iff v = \so^k \ea^\ell u \textup{ for some }k,\ell \in \ZZ_{\geq 0}.\\
u \SEarrow v &\iff v = \so^k \ea^\ell u \textup{ for some }k,\ell \in \ZZ_{> 0}.\\
u \nearrow v &\iff v = \no^k \ea^\ell u \textup{ for some }k,\ell \in \ZZ_{\geq 0}.\\
u \NEarrow v &\iff v = \no^k \ea^\ell u \textup{ for some }k,\ell \in \ZZ_{> 0}.
\end{align*}
We note that \(\searrow, \nearrow\) are in fact partial orders on \(\N\), with \(\searrow\) being the product partial order induced by \(\ZZ\) on \(\N\). 
We extend \(\SEarrow, \NEarrow\) to subsets \(\tau, \nu \subset \N\), writing \(\tau \NEarrow \nu\) provided \(u \NEarrow v\) for all \(u \in \tau, v \in \nu\). 
For \(u \in \tau\subset \N\), we say \(u \in \tau\) is {\em maximally northeast} (resp. {\em maximally southwest}) {\em in} \(\tau\) provided that does not exist any \(u \neq v \in \tau\) such that \(u \nearrow v\) (resp. \(v \nearrow u\)).

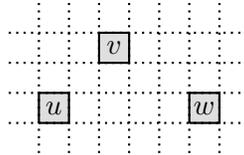
\begin{figure}[h]
\begin{align*}
{}
\hackcenter{
\begin{tikzpicture}[scale=0.4]
%
\draw[thick, fill = lightgray!50] (0,0)--(1,0)--(1,1)--(0,1)--(0,0);
\draw[thick, fill = lightgray!50] (2,2)--(3,2)--(3,3)--(2,3)--(2,2);
\draw[thick, fill = lightgray!50] (5,0)--(6,0)--(6,1)--(5,1)--(5,0);
%
\draw[thick, dotted] (-1,0)--(7,0);
\draw[thick, dotted] (-1,1)--(7,1);
\draw[thick, dotted] (-1,2)--(7,2);
\draw[thick, dotted] (-1,3)--(7,3);
\draw[thick, dotted] (0,-1)--(0,4);
\draw[thick, dotted] (1,-1)--(1,4);
\draw[thick, dotted] (2,-1)--(2,4);
\draw[thick, dotted] (3,-1)--(3,4);
\draw[thick, dotted] (4,-1)--(4,4);
\draw[thick, dotted] (5,-1)--(5,4);
\draw[thick, dotted] (6,-1)--(6,4);
\node at (0.5,0.5){$u$};
\node at (2.5,2.5){$v$};
\node at (5.5,0.5){$w$};
\end{tikzpicture}
}
\end{align*}
\caption{Nodes \(u,v,w\), satisfying \(u  \nearrow v\), \(v \searrow w\), \(u \nearrow w\), \(u \searrow w\)}
\label{fig:nodes}       
\end{figure}

For \(u,v \in \N\), a {\em path from \(u\) to \(v\)} is a sequence of nodes \((z^i)_{i=1}^k\) such that \(z^1 = u\), \(z^k = v\), and such that \(z^{i+1} \in \{\no(z^i), \ea(z^i), \so(z^i), \we(z^i)\}\) for \(i \in [1,k-1]\). We say it is moreover a \(\no/\ea\) (resp. \(\so/\ea\)) path if \(z^{i+1} \in \{\no(z^i), \ea(z^i)\}\) (resp. \(z^{i+1} \in \{\so(z^i), \ea(z^i)\}\)) for all \(i \in [1,k-1]\). Defining \(\dist(u,v) = |v_1 + v_2 - u_1 -u_2|\), we have that \(\dist(u,v)+1\) is the length of the shortest path from \(u\) to \(v\), and is thus the length of any \(\no/\ea\) or \(\so/\ea\) path that connects \(u\) and \(v\).

For \(u \in \N\), set \(\diag(u) = u_2 - u_1 \in \ZZ\), and define the {\em \(n\)th diagonal} in \(\N\) to be the set
\begin{align*}
\D_n &= \{u \in \N \mid \diag(u) = n\}\\
&= \{(\so \ea)^k(0,n) \mid k \in \ZZ_{\geq 0}\} \cup \{(\no \we)^k(0,n) \mid k \in \ZZ_{>0}\}.
\end{align*}

\subsection{Skew shapes}
Now we define the main combinatorial object of study in this paper.

\begin{definition}\label{skewdef}
We say a finite subset \(\tau\) of \(\N\) is:
\begin{enumerate}
\item a {\em skew shape} provided that for all \(u,w \in \tau\), \(v \in \N\), \(u \searrow v \searrow w\) implies \(v \in \tau\).
\item {\em thin} if \(|\D_n \cap \tau| \leq 1\) for all \(n \in \N\).
\item {\em connected} if for all \(u,v \in \tau\), there exists a path from \(u\) to \(v\) contained in \(\tau\).
\item a {\em ribbon} provided that \(\tau\) is a nonempty thin connected skew shape.
\item {\em cornered} if there exists at most one \(u \in \tau\) such that \(\so u, \we u \notin \tau\) and at most one \(v \in \tau\) such that \(\no v, \ea v \notin \tau\).
\item {\em diagonal-convex} provided that for all \(n \in \NN\), \(u,w \in \D_n \cap \tau\), and \(v \in \D_n\), \(u \searrow v \searrow w\) implies \(v \in \tau\).
\item a {\em Young diagram} provided that \(\tau = \varnothing\) or \(\tau\) is a skew shape containing a node \(\tau_{\nnww}\) such that \(\tau_{\nnww} \searrow u\) for all \(u \in \tau\).
\end{enumerate}
We write \(\SSSS\) (resp. \(\SSSS_{\tt c}\)) for the set of all skew shapes (resp. nonempty connected skew shapes), and \(\SSSS(n)\) (resp. \(\SSSS_{\tt c}(n)\)) for the subset of those of cardinality \(n\).
\end{definition}

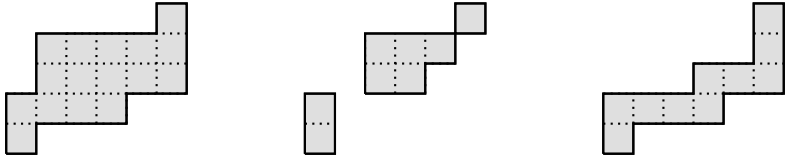
\begin{figure}[h]
\begin{align*}
{}
\hackcenter{
\begin{tikzpicture}[scale=0.4]
%
\draw[thick, fill=lightgray!50]  (0,-1)--(1,-1)--(1,0)--(4,0)--(4,1)--(6,1)--(6,4)--(5,4)--(5,3)--(1,3)--(1,2)--(1,1)--(0,1)--(0,-1);
%
\draw[thick, dotted] (1,0)--(1,1);
\draw[thick, dotted] (2,0)--(2,3);
\draw[thick, dotted] (3,0)--(3,3);
\draw[thick, dotted] (4,0)--(4,3);
\draw[thick, dotted] (5,1)--(5,3);
\draw[thick, dotted] (0,0)--(4,0);
\draw[thick, dotted] (0,1)--(6,1);
\draw[thick, dotted] (1,2)--(6,2);
\draw[thick, dotted] (2,3)--(6,3);
\draw[thick]  (0,-1)--(1,-1)--(1,0)--(4,0)--(4,1)--(6,1)--(6,4)--(5,4)--(5,3)--(1,3)--(1,2)--(1,1)--(0,1)--(0,-1);
%
\end{tikzpicture}
}
\qquad
\qquad
\hackcenter{
\begin{tikzpicture}[scale=0.4]
%
\draw[thick, fill=lightgray!50]  (0,-1)--(1,-1)--(1,1)--(0,1)--(0,-1);
\draw[thick, fill=lightgray!50] (2,1)--(4,1)--(4,2)--(5,2)--(5,3)--(2,3)--(2,1);
\draw[thick, fill=lightgray!50] (5,3)--(6,3)--(6,4)--(5,4)--(5,3);
%
\draw[thick, dotted] (2,1)--(2,3);
\draw[thick, dotted] (3,1)--(3,3);
\draw[thick, dotted] (4,1)--(4,3);
\draw[thick, dotted] (0,0)--(1,0);
\draw[thick, dotted] (2,2)--(5,2);
\draw[thick]  (0,-1)--(1,-1)--(1,1)--(0,1)--(0,-1);
\draw[thick] (2,1)--(4,1)--(4,2)--(5,2)--(5,3)--(2,3)--(2,1);
\draw[thick] (5,3)--(6,3)--(6,4)--(5,4)--(5,3);
%
\end{tikzpicture}
}
\qquad
\qquad
\hackcenter{
\begin{tikzpicture}[scale=0.4]
%
\draw[thick, fill=lightgray!50]  (0,-1)--(1,-1)--(1,0)--(4,0)--(4,1)--(6,1)--(6,4)--(5,4)--(5,2)--(3,2)--(3,1)--(1,1)--(0,1)--(0,-1);
%
\draw[thick, dotted] (1,0)--(1,1);
\draw[thick, dotted] (2,0)--(2,1);
\draw[thick, dotted] (3,0)--(3,1);
\draw[thick, dotted] (4,1)--(4,2);
\draw[thick, dotted] (5,1)--(5,2);
\draw[thick, dotted] (0,0)--(4,0);
\draw[thick, dotted] (0,1)--(6,1);
\draw[thick, dotted] (3,2)--(6,2);
\draw[thick, dotted] (5,3)--(6,3);
\draw[thick]  (0,-1)--(1,-1)--(1,0)--(4,0)--(4,1)--(6,1)--(6,4)--(5,4)--(5,2)--(3,2)--(3,1)--(1,1)--(0,1)--(0,-1);
%
\end{tikzpicture}
}
\end{align*}
\caption{Connected skew shape; disconnected skew shape; ribbon}
\label{fig:examples}       
\end{figure}

\begin{remark}\label{YDrem1}
Ribbons are also variously called {\em skew hooks}, {\em rim hooks} or {\em edge hooks} in the literature. Every skew shape can be realized as a set difference \(\lambda/ \mu := \lambda \backslash \mu\) of Young diagrams, and, in such context, are often called {\em skew Young diagrams}. We discuss this connection further in \S\ref{otherformsec}.
\end{remark}

\subsection{Results on skew shapes}
Now we establish some preliminary lemmas relating the terms in Definition~\ref{skewdef}.

\begin{lemma}\label{skew:cornered=connected}
Let \(\tau\) be a skew shape. Then \(\tau\) is cornered if and only if \(\tau\) is connected.
\end{lemma}
\begin{proof}
The claim is trivial if \(\tau = \varnothing\), so assume \(\tau\) is nonempty.

\((\implies)\)
Let \(u,v\) be nodes in \(\tau\). By induction on \(\dist(u,v)\) we show that there exists a path from \(u\) to \(v\) in \(\tau\). The base case \(\dist(u,v) = 0\) is trivial, so assume \(\dist(u,v) > 0\), and that the claim holds for all \(u',v' \in \tau\) with \(\dist(u',v') < \dist(u,v)\). Without loss of generality, we have either \(u \searrow v\) or \(u \NEarrow v\). If \(u \searrow v\) then since \(\tau\) is a skew shape any \(\so/\ea\) path from \(u\) to \(v\) is in \(\tau\). So assume \(u \NEarrow v\). Since \(\tau\) is finite, there exists a maximally southwest node \(m \in \tau\). Then \(\so m, \we m \notin \tau\), and \(m \neq v\). Then since \(\tau\) is cornered, there exists some \(z \in \{\so v, \we v\} \cap \tau\). Then \(\dist(u,z) < \dist(u,v)\), so by induction there exists a path from \(u\) to \(z\) in \(\tau\). As \(z\) is adjacent to \(u\), this may be extended to a path from \(u\) to \(v\) in \(\tau\). This completes the induction step, and the proof that \(\tau\) is connected.

\((\impliedby)\)
By way of contradiction, assume \(u,v\) are distinct nodes with \(\ea u , \no u, \ea v, \no v \notin \tau\). Without loss of generality, we have either \(u \searrow v\) or \(u \NEarrow v\). Say \(u \searrow v\). If \(v_2>u_2\), then \(u \searrow \ea u \searrow v\), so \(\ea u \in \tau\) since \(\tau\) is a skew shape, a contradiction. If \(v_1>u_1\), then \(u \searrow \no v \searrow v\), so \(\no v  \in \tau\), again a contradiction. Thus it cannot be that \(u \searrow v\), so assume \(u \NEarrow v\). Let \(d= \diag(u) +1\), so that \(\D_d\) is the diagonal directly above the diagonal containing \(u\) (and is therefore below the diagonal containing \(v\)). Since \(\tau\) is connected, there is a path from \(u\) to \(v\) in \(\tau\). This path must intersect \(\D_d\) at some point \(z \in \tau\). Since \(\ea u, \no u \notin \tau\), we have then that either \(u \SEarrow z\) or \(z \SEarrow u\). In the former case, we have that \(u \searrow \ea u \searrow z\), and in the latter case we have that \(z \searrow \no u \searrow u\). But then, as \(\tau\) is a skew shape, this would imply that either \(\ea u \in \tau\) or \(\no u \in \tau\), a contradiction. Thus, in any case we derive a contradiction, so no such nodes \(u,v\) may exist. We may show that there cannot exist distinct nodes \(u,v\) with \(\so u, \we u, \so v, \we v \notin \tau\) in a similar fashion. Thus \(\tau\) is cornered.
\end{proof}

\begin{proposition}\label{CSS=PDC}
Let \(\tau\) be a finite subset of \(\N\). Then \(\tau\) is a connected skew shape if and only if \(\tau\) is cornered and diagonal-convex.
\end{proposition}
\begin{proof}
The claim is trivial if \(\tau = \varnothing\), so assume \(\tau\) is nonempty.

\((\implies)\) That \(\tau\) is cornered follows from Lemma~\ref{skew:cornered=connected}, and diagonal-convexity is implied by the fact that \(\tau\) is a skew shape.

\((\impliedby)\) For \(n \in \Z\), let \(\D_{n,\tau} = \D_n \cap \tau\). 
Since \(\tau\) is diagonal-convex, we have either \(\D_{n,\tau} =\varnothing\) or 
\(
D_{n,\tau} = \{\D_{n,\tau}^0, \D_{n,\tau}^1, \ldots, \D_{n,\tau}^{r_n}\}
\)
for some \(r_n \in \NN\), where \(\D_{n,\tau}^i = (\ssee)^i \D_{n,\tau}^0\) for \(i = 0, \ldots, r_n\). Let \(a\) be minimal such that \(\D_{a,\tau} \neq \varnothing\) and \(b\) be maximal such that \(\D_{b,\tau} \neq \varnothing\). 

Note that \(\D_{i,\tau} \neq \varnothing\) for all \(i \in [a,b]\). Indeed, if there were some \(i \in [a,b]\) such that \(\D_{i,\tau}= \varnothing\) and \(\D_{i-1,\tau} \neq \varnothing\), then every \(u \in \D_{i-1,\tau}, \D_{b,\tau}\) would have \(\no u, \ea u \notin \tau\), a contradiction of the cornered-ness of \(\tau\). Therefore we may write \(\tau\) as the union of nonempty sets
\(
\tau = \bigsqcup_{i \in [a,b]} \D_{i,\tau}.
\)

For \(i \in [a,b-1]\), we claim that \(\D_{i+1,\tau}^0 \in \{\no \D_{i,\tau}^0, \ea \D_{i,\tau}^0\}\). Indeed, assume this is not the case. Then we have \(\ea \D_{i,\tau}^0 \SEarrow \D_{i+1,\tau}^0 \) or \(\D_{i+1,\tau}^0 \SEarrow \no\D_{i,\tau}^0\). In the former case we would have \(\no \D^0_{i,\tau}, \ea \D^0_{i,\tau} \notin \tau\) and \(\no \D_{b,\tau}^0, \ea \D_{b,\tau}^0 \notin \tau\). In the latter case we would have \(\so \D^0_{i+1,\tau}, \we \D^0_{i+1,\tau} \notin \tau\) and \(\so \D_{a,\tau}^0, \we \D_{a,\tau}^0 \notin \tau\). Both are contradictions of the cornered-ness of \(\tau\), proving the claim. Therefore we have \(\D_{i,\tau}^0 \nearrow \D_{j,\tau}^0\) for all \(a \leq i \leq j\leq b\). A similar argument shows \(\D_{i,\tau}^{r_i} \nearrow \D_{j,\tau}^{r_j}\) for all \(a \leq i \leq j\leq b\).

Assume \(u \searrow v \searrow w\) for some \(u,w \in \tau\). We claim that \(v \in \tau\). We have that \(u \in \D_{i,\tau}\), \(w \in \D_{j,\tau}\) for some \(i,j \in [a,b]\), so it follows that 
\begin{align*}
\D_{i,\tau}^0 \searrow u \searrow v \searrow  w \searrow \D_{j,\tau}^{r_j}.
\end{align*}
Let \(n = \diag(v)\). If we show that \(n \in [a,b]\) and \(v \in \D_{n,\tau}\), then we are done, since \(\D_{n,\tau} \subseteq \tau\).

We first show that \(n \in [a,b]\). 
If \(n > b\), then it is not the case that \(v \nearrow \D_{b,\tau}^{r_b}\). But we have \(v \searrow \D_{j,\tau}^{r_j} \nearrow \D_{b,\tau}^{r_b}\), so it follows that \((\D_{b,\tau}^{r_b})_1>v_1\). 
Then, since \(\D_{i,\tau}^0 \searrow v\), we have that \((\D_{b,\tau}^{r_b})_1>(\D_{i,\tau}^0)_1\). But this contradicts the fact that \(\D_{i,\tau}^0 \nearrow \D_{b,\tau}^0 = \D_{b,\tau}^{r_b}\). Therefore \(n \leq b\). 
If \(n < a\), then it is not the case that \(\D_{a,\tau}^{0} \nearrow v\). But we have \(\D_{a,\tau}^0 \nearrow \D_{i,\tau}^0 \searrow v\), so it follows that \(v_1>(\D_{a,\tau}^0)_1\). Then, since \(v \searrow \D_{j,\tau}^{r_j}\), we have that \((\D_{j,\tau}^{r_j})_1>(\D_{a,\tau}^0)_1\). But this contradicts the fact that \(\D_{a,\tau}^0 = \D_{a,\tau}^{r_a} \nearrow \D_{j,\tau}^{r_j}\). Thus \(a \leq n\). Therefore \(n \in [a,b]\).

Now we show that \(v \in \D_{n,\tau}\), in four separate cases:

{\em Assume \(i,j\leq n\).} Then \(\D_{i,\tau}^0 \nearrow \D_{n,\tau}^0\) and \(\D_{j,\tau}^{r_j} \nearrow \D_{n,\tau}^{r_n}\). Since \(\D_{i,\tau}^0 \searrow v\), we have that \(v_1\geq (\D_{n,\tau}^0)_1\). Since \(v \searrow \D_{j,\tau}^{r_j}\), we have that \((\D_{n,\tau}^{r_n})_2\geq v_2\). Therefore \(v \in \D_{n,\tau}\).

{\em Assume \(n \leq i,j\).} Then \(\D_{n,\tau}^0 \nearrow \D_{i,\tau}^0\), and \(\D_{n,\tau}^{r_n} \nearrow \D_{j,\tau}^{r_j}\). Since \(\D_{i,\tau}^0 \searrow v\), we have that \(v_2 \geq (\D_{n,\tau}^0)_2\). Since \(v \searrow \D_{j,\tau}^{r_j}\), we have that \((\D_{n,\tau}^{r_n})_1 \geq v_1\). Therefore \(v \in \D_{n,\tau}\). 

{\em Assume \(i \leq n \leq j\).} Then \(\D_{i,\tau}^0 \nearrow \D_{n,\tau}^0\) and \(\D_{n,\tau}^{r_n}\nearrow \D_{j,\tau}^{r_j}\). Then since \(\D_{i,\tau}^0 \searrow v\), we have that \(v_1 \geq (\D_{n,\tau}^0)_1\). Since \(v \searrow \D_{j,\tau}^{r_j}\), we have that \((\D_{n,\tau}^{r_n})_1 \geq v_1\). Therefore \(v \in \D_{n,\tau}\).

{\em Assume \(j \leq n \leq i\).} Then \(\D_{j,\tau}^{r_j} \nearrow \D_{n,\tau}^{r_n}\) and \(\D_{n,\tau}^{0} \nearrow \D_{i,\tau}^{0}\). Then since \(\D_{i,\tau}^0 \searrow v\), we have that \(v_2 \geq (\D_{n,\tau}^0)_2\). Since \(v \searrow \D_{j,\tau}^{r_j}\), we have that \((\D_{n,\tau}^{r_n})_2 \geq v_2\). Therefore \(v \in \D_{n,\tau}\).

Thus, in any case, we have \(v \in \D_n \subset \tau\), so the claim holds, and \(\tau\) is a skew shape. Then, since \(\tau\) is a cornered skew shape, it is connected by Lemma~\ref{skew:cornered=connected}, completing the proof.
\end{proof}

\subsection{A path criterion for connected skew shapes}

\begin{proposition}\label{pathcrit}
Let \(\tau\) be a finite subset of \(\N\). Then \(\tau\) is a connected skew shape if and only if every \(u,v \in \tau\) satisfies:
\begin{enumerate}
\item If \(u \searrow v\), then every \(\so/\ea\) path from \(u\) to \(v\) is contained in \(\tau\);
\item If \(u \nearrow v\), then there exists a \(\no/\ea\) path from \(u\) to \(v\) contained in \(\tau\).
\end{enumerate}
\end{proposition}
\begin{proof}
\((\implies)\) Let \(u,v \in \tau\) be such that \(u \searrow v\). If \((z^i)_{i=1}^k\) is any \(\so/\ea\) path from \(u\) to \(v\), then we clearly have \(u \searrow z^i \searrow v\) for all \(i \in [1,k]\), so the path is contained in \(\tau\) since \(\tau\) is a skew shape, proving (i).

Let \(u,v \in \tau\) be such that \(u \nearrow v\). We show by induction on \(\dist(u,v)\) that there exists a \(\no/\ea\) path from \(u\) to \(v\) in \(\tau\). The base case \(\dist(u,v) = 0\) is clear, so assume \(\dist(u,v)>0\) and make the induction assumption on all nodes \(u' \nearrow v'\) in \(\tau\) with \(\dist(u',v')<\dist(u,v)\). We consider three possible cases:

{\em Assume \(u_1 = v_1\).} Then \(v_2 = u_2 + c\) for some \(c \in \NN\). Then \(((u_1, u_2+i-1))_{i=1}^{c+1}\) is a \(\no/\ea\) path from \(u\) to \(v\), which is contained in \(\tau\) by the skew shape criterion. 

{\em Assume \(u_2 = v_2\).} Then \(v_1 = u_1 - c\) for some \(c \in \NN\). Then \(((u_1 - i + 1, u_2))_{i=1}^{c+1}\) is a \(\no/\ea\) path from \(u\) to \(v\), which is contained in \(\tau\) by the skew shape criterion. 

{\em Assume \(u \NEarrow v\).} 
If \(w\) is a maximally southwest node in \(\tau\), then \(\so w, \we w \notin \tau\), and \(w \neq v\). By Proposition~\ref{CSS=PDC}, \(\tau\) is cornered, so at least one of \(\so v, \we v\) is in \(\tau\); call this node \(v'\). Then \(u \nearrow v'\), so by the induction assumption there exists a \(\no/\ea\) path \((z^i)_{i=1}^{k}\) from \(u\) to \(v'\) contained in \(\tau\). Then, setting \(z^{k+1} = v\), we have that \((z^i)_{i=1}^{k+1}\) is a \(\no/\ea\) path from \(u\) to \(v\) contained in \(\tau\). This completes the induction step and the proof of (ii).

\((\impliedby)\) If \(u \searrow v \searrow w\) for nodes \(u,w \in \tau\), then there clearly exists a \(\so/\ea\) path from \(u\) to \(w\) which contains \(v\). Then by (i), \(v \in \tau\), so \(\tau\) is a skew shape. If \(u,v\) are nodes in \(\tau\), then one of \(u \nearrow v\), \(u \searrow v\), \(v \nearrow u\), \(v \searrow u\) must be true. Then (i), (ii) guarantee that a path exists in \(\tau\) from \(u\) to \(v\) (inverting the path from \(v\) to \(u\) if necessary). Thus \(\tau\) is connected.
\end{proof}

\subsection{Connected components}
Any nonempty finite set \(\tau \subseteq \N\) may be decomposed into {\em connected components}, i.e., nonempty connected sets \(\tau_1, \ldots, \tau_k\) for some \(k \in \NN\), such that there is no path from \(u\) to \(v\) contained in \(\tau\) whenever \(u \in \tau_i, v \in \tau_j\) for \(i \neq j\).

\begin{lemma}\label{seconnect}
Let \(\tau \in \SSSS\), and \(u,v \in \tau\) be such that \(u \searrow v\). Then the nodes \(\{w \in \N \mid u \searrow w \searrow v\}\) belong to the same connected component of \(\tau\).
\end{lemma}
\begin{proof}
Let \(w \in \N\) be such that \(u \searrow w \searrow v\). Choose any \(\so/\ea\)-path \((z^i)_{i=1}^k\) from \(u\) to \(v\) which passes through \(w\). We have \(z^i \in \tau\) for all \(i \in [1,k]\) by Proposition~\ref{pathcrit}. This path lies within \(\tau\) by Proposition~\ref{pathcrit}, so \(u,w,v\) are in the same connected component of \(\tau\).
\end{proof}

\begin{lemma}\label{skewcomps}
Connected components of skew shapes are connected skew shapes.
\end{lemma}
\begin{proof}
Let \(\tau\) be a skew shape, and let \(\tau'\) be a connected component of \(\tau\). Let \(u,v \in \tau'\), \(w \in \N\) be such that \(u \searrow w \searrow v\). By Lemma~\ref{seconnect}, \(w \in \tau'\), so \(\tau'\) is a skew shape.
\end{proof}

\begin{proposition}\label{skewshapecondecomp}
Let \(\tau\) be a nonempty finite subset of \(\N\). Then \(\tau \in \SSSS\) if and only if there exists some \(k \in \NN\) and \(\tau_k \NEarrow \cdots \NEarrow \tau_1 \in \SSSS_{\tt c}\) such that \(\tau = \tau_1 \sqcup \cdots \sqcup \tau_k\).
\end{proposition}
\begin{proof}
\((\impliedby)\) Let \(u,v \in \tau\), with \(u \searrow v\). Then by the assumption, we must have that \(u,v \in \tau_i\) for some \(i \in [1,k]\). Then, for any \(w \in \N\) with \(u \searrow w \searrow v\), we have \(w \in \tau_i \subseteq \tau\) since \(\tau_i\) is a skew shape. Thus \(\tau\) is a skew shape.

\((\implies)\) 
Let \(\tau_1, \ldots, \tau_k\) be the connected components of \(\tau\), ordered such that \(i \leq j\) implies that \(\diag(u) \geq \diag(v)\) for some \(u \in \tau_i, v \in \tau_j\). By Lemma~\ref{skewcomps}, these are all skew shapes. We must show that \(\tau_j \NEarrow \tau_i\) when \(j>i\).

Set \(D_i = \{n \in \N \mid \D_n \cap \tau_i \neq \varnothing\}\). Note that \(D_i \cap D_j \neq \varnothing\) only if \(i = j\), as nodes on the same diagonal in \(\tau\) must be in the same connected component by Lemma~\ref{seconnect}. Moreover, \(D_i\) is an interval for all \(i\), since \(\tau_i\) is connected. Thus we have that \(i<j\) implies \(m> n\) for all \(m \in D_i\), \(n \in D_j\), and \(\diag(u) > \diag(v)\) for all \(u \in \tau_i, v \in \tau_j\).
Let \(i \in [1,k-1]\), and assume \(u \in \tau_i\), \(v \in \tau_{i+1}\). It cannot be that \(u \searrow v\) or \(v \searrow u\), else \(u,v\) would be in the same connected component by Lemma~\ref{seconnect}. Thus \(u \NEarrow v\) or \(v \NEarrow u\). But \(\diag(u) > \diag(v)\) by the previous paragraph, so we have \(v \NEarrow u\), and thus \(\tau_{i+1} \NEarrow \tau_i\), as desired.
\end{proof}

\begin{lemma}\label{uniquemax}
Let \(\tau\) be a nonempty skew shape. Then there exists a unique maximally southwest node \(u \in \tau\), a unique maximally northeast node \(v \in \tau\), and \(u \nearrow w \nearrow v\) for all \(w \in \tau\).
\end{lemma}
\begin{proof}
We first prove the claim in the event \(\tau\) is connected. Any maximally southwest node \(u \in \tau\) will have \(\{\so u, \we u\} \cap \tau = \varnothing\), but by Proposition~\ref{CSS=PDC}, \(\tau\) is cornered, so \(u\) is the unique such node. Assume \(u \neq w \in \tau\). We cannot have \(w \nearrow u\) as \(u\) is maximally southwest. If \(u \SEarrow w\), then \(\so u \in \tau\) since \(\tau\) is a skew shape, a contradiction. If \(w \SEarrow u\), then \(\we u \in \tau\) since \(\tau\) is a skew shape, another contradiction. Thus we have \(u \nearrow w\). The proof for the unique maximally northeast node \(v \in \tau\) is similar.

Removing the connectedness assumption, we decompose \(\tau = \tau_1 \sqcup \cdots \sqcup \tau_k\) into connected skew shape components with \(\tau_k \NEarrow \cdots \NEarrow \tau_1\) as in Proposition~\ref{skewshapecondecomp}. Then, by the previous paragraph, there is a unique maximally southwest node \(u \in \tau_k\), and a unique maximally northeast node \(v \in \tau_1\), and \(u \nearrow w \nearrow v\) for all \(w \in \tau\).
\end{proof}

In view of Lemma~\ref{uniquemax}, for a nonempty skew shape \(\tau\), we may give the unique maximally southwest (resp. northeast) element the label \(\tau_{\ssww}\) (resp. \(\tau_{\nnee}\)). 

\subsection{Results on ribbons}
The following lemma is clear from definitions:
\begin{lemma}\label{dumbdiag}
Let \((z^i)_{i = 1}^k\) be a \(\no/\ea\) path. Then \(\diag(z^i) = \diag(z^1) + i -1\) for \(i \in [1,k]\).
\end{lemma}

\begin{lemma}\label{hookispath}
Let \(\xi\) be a finite subset of \(\N\). Then \(\xi\) is a ribbon if and only if there exists a \(\no/\ea\) path  \((z^i)_{i = 1}^{|\xi|}\) such that \(\xi = \{z^i\}_{i = 1}^{|\xi|}\).
\end{lemma}
\begin{proof}
\((\implies)\) Let \(\xi\) be a ribbon. Then by Proposition~\ref{pathcrit} and Lemma~\ref{uniquemax} there exists a \(\no/\ea\) path \((z^i)_{i=1}^k\) from \(\xi_{\ssww}\) to \(\xi_{\nnee}\) contained in \(\xi\). Let \(w \in \tau\). By Lemma~\ref{uniquemax}, we have \(\diag(\xi_{\ssww}) \leq \diag(w) \leq \diag(\xi_{\nnee})\), so by Lemma~\ref{dumbdiag} there exists some \(i \in [1,k]\) such that \(\diag(w) = \diag(z^i)\). As \(\xi\) is thin, this implies \(w = z^i\). Thus \(k = |\xi|\) and \(\xi = \{z^i\}_{i=1}^{|\xi|}\).

\((\impliedby)\) By Lemma~\ref{dumbdiag}, we have that \(\xi\) is thin, and thus diagonal-convex. It is clear that \(z^1\) is the unique node \(u \in \xi\) such that \(\so u, \we u \notin \xi\), and \(z^{|\xi|}\) is the unique node \(u \in \xi\) such that \(\no u, \ea u \notin \xi\). Thus \(\xi\) is cornered. Thus by Proposition~\ref{CSS=PDC}, \(\xi\) is a connected skew shape, so it is a ribbon.
\end{proof}

This gives the immediate corollary

\begin{corollary}\label{sizehook}
For a ribbon \(\xi\), we have \(|\xi| = \dist(\xi_{\ssww}, \xi_{\nnee})+1\).
\end{corollary}

\subsection{Tilings and tableaux}\label{tilingsec}
Let \(\tau\) be a nonempty skew shape. A {\em tiling} of \(\tau\) is a set \(\Lambda\) of pairwise disjoint nonempty skew shapes such that \(\tau = \bigsqcup_{\lambda \in \Lambda} \lambda\). We call the members of \(\Lambda\) {\em tiles}.
A {\em \(\Lambda\)-tableau} \(\ttt\) is a bijection \(\ttt: [1, |\Lambda|] \to \Lambda\) such that \(u \in \ttt(i), v \in \ttt(j)\) and \(u \searrow v\) imply \(i\leq j\). We also refer to the pair \((\Lambda, \ttt)\), or the tile sequence \((\ttt(1), \ldots, \ttt(|\Lambda|))\), as a {\em tableau for \(\tau\)}.
Roughly speaking, the ordering condition means the tiles \(\ttt(1), \ldots, \ttt(|\Lambda|)\) can be sequentially `slid into place' from the southeast without their constituent boxes colliding. For instance, in Figure~\ref{fig:tilex}, a tiling \(\Lambda\) for a skew shape \(\tau\) is shown, along with the (only) two possible \(\Lambda\)-tableaux.

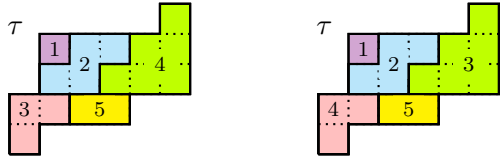
\begin{figure}[h]
\begin{align*}
{}
\hackcenter{
\begin{tikzpicture}[scale=0.4]
%
\draw[thick, fill=violet!35]   (0,-1)--(1,-1)--(1,0)--(4,0)--(4,1)--(6,1)--(6,4)--(5,4)--(5,3)--(1,3)--(1,2)--(1,1)--(0,1)--(0,-1);
\draw[thick, fill=pink]  (0,-1)--(1,-1)--(1,0)--(2,0)--(2,1)--(1,1)--(0,1)--(0,-1);
\draw[thick, fill=yellow] (2,0)--(4,0)--(4,1)--(2,1)--(2,0);
\draw[thick,fill=lime] (3,1)--(6,1)--(6,4)--(5,4)--(5,3)--(4,3)--(4,2)--(3,2)--(3,1);
\draw[thick, fill=cyan!25] (1,1)--(3,1)--(3,2)--(4,2)--(4,3)--(2,3)--(2,2)--(1,2)--(1,1);
%
\draw[thick, dotted] (1,0)--(1,1);
\draw[thick, dotted] (2,0)--(2,3);
\draw[thick, dotted] (3,0)--(3,3);
\draw[thick, dotted] (4,0)--(4,3);
\draw[thick, dotted] (5,1)--(5,3);
\draw[thick, dotted] (0,0)--(4,0);
\draw[thick, dotted] (0,1)--(6,1);
\draw[thick, dotted] (1,2)--(6,2);
\draw[thick, dotted] (2,3)--(6,3);
\node at (1.5,2.5){ $\scriptstyle 1$};
\node at (0.5,0.5){ $\scriptstyle3$};
\fill[fill=cyan!25] (2,1.5)--(3,1.5)--(3,2.5)--(2,2.5)--(2,0);
\node at (2.5,2){ $\scriptstyle2$};
\fill[fill=lime] (4.5,1.5)--(5.5,1.5)--(5.5,2.5)--(4.5,2.5)--(4.5,1.5);
\node at (5,2){ $\scriptstyle4$};
\fill[fill=yellow] (2.5,0)--(3.5,0)--(3.5,1)--(2.5,1)--(2.5,0);
\node at (3,0.5){ $\scriptstyle5$};
\node at (0.2,3.2){ $\tau$};
\draw[thick]   (0,-1)--(1,-1)--(1,0)--(4,0)--(4,1)--(6,1)--(6,4)--(5,4)--(5,3)--(1,3)--(1,2)--(1,1)--(0,1)--(0,-1);
\draw[thick]  (0,-1)--(1,-1)--(1,0)--(2,0)--(2,1)--(1,1)--(0,1)--(0,-1);
\draw[thick] (2,0)--(4,0)--(4,1)--(2,1)--(2,0);
\draw[thick] (3,1)--(6,1)--(6,4)--(5,4)--(5,3)--(4,3)--(4,2)--(3,2)--(3,1);
\draw[thick] (1,1)--(3,1)--(3,2)--(4,2)--(4,3)--(2,3)--(2,2)--(1,2)--(1,1);
\end{tikzpicture}
}
\qquad
\qquad
\hackcenter{
\begin{tikzpicture}[scale=0.4]
%
\draw[thick, fill=violet!35]   (0,-1)--(1,-1)--(1,0)--(4,0)--(4,1)--(6,1)--(6,4)--(5,4)--(5,3)--(1,3)--(1,2)--(1,1)--(0,1)--(0,-1);
\draw[thick, fill=pink]  (0,-1)--(1,-1)--(1,0)--(2,0)--(2,1)--(1,1)--(0,1)--(0,-1);
\draw[thick, fill=yellow] (2,0)--(4,0)--(4,1)--(2,1)--(2,0);
\draw[thick,fill=lime] (3,1)--(6,1)--(6,4)--(5,4)--(5,3)--(4,3)--(4,2)--(3,2)--(3,1);
\draw[thick, fill=cyan!25] (1,1)--(3,1)--(3,2)--(4,2)--(4,3)--(2,3)--(2,2)--(1,2)--(1,1);
%
\draw[thick, dotted] (1,0)--(1,1);
\draw[thick, dotted] (2,0)--(2,3);
\draw[thick, dotted] (3,0)--(3,3);
\draw[thick, dotted] (4,0)--(4,3);
\draw[thick, dotted] (5,1)--(5,3);
\draw[thick, dotted] (0,0)--(4,0);
\draw[thick, dotted] (0,1)--(6,1);
\draw[thick, dotted] (1,2)--(6,2);
\draw[thick, dotted] (2,3)--(6,3);
\node at (1.5,2.5){ $\scriptstyle 1$};
\node at (0.5,0.5){ $\scriptstyle4$};
\fill[fill=cyan!25] (2,1.5)--(3,1.5)--(3,2.5)--(2,2.5)--(2,0);
\node at (2.5,2){ $\scriptstyle2$};
\fill[fill=lime] (4.5,1.5)--(5.5,1.5)--(5.5,2.5)--(4.5,2.5)--(4.5,1.5);
\node at (5,2){ $\scriptstyle3$};
\fill[fill=yellow] (2.5,0)--(3.5,0)--(3.5,1)--(2.5,1)--(2.5,0);
\node at (3,0.5){ $\scriptstyle5$};
\node at (0.2,3.2){ $\tau$};
\draw[thick]   (0,-1)--(1,-1)--(1,0)--(4,0)--(4,1)--(6,1)--(6,4)--(5,4)--(5,3)--(1,3)--(1,2)--(1,1)--(0,1)--(0,-1);
\draw[thick]  (0,-1)--(1,-1)--(1,0)--(2,0)--(2,1)--(1,1)--(0,1)--(0,-1);
\draw[thick] (2,0)--(4,0)--(4,1)--(2,1)--(2,0);
\draw[thick] (3,1)--(6,1)--(6,4)--(5,4)--(5,3)--(4,3)--(4,2)--(3,2)--(3,1);
\draw[thick] (1,1)--(3,1)--(3,2)--(4,2)--(4,3)--(2,3)--(2,2)--(1,2)--(1,1);
\end{tikzpicture}
}
\end{align*}
\caption{A tiling \(\Lambda\) for a skew shape \(\tau\), and two \(\Lambda\)-tableaux---label \(i\) indicates the tile \(\ttt(i)\)}
\label{fig:tilex}       
\end{figure}

\begin{remark}
If \((\Lambda, \ttt)\) is a tableau for \(\tau\) with \(|\lambda| =1\) for all \(\lambda \in \Lambda\), then \((\Lambda, \ttt)\) is a Young tableau in the traditional sense. Young tableaux are important objects in combinatorics, representation theory and algebraic geometry, and the definition above can be viewed as a generalization of Young tableaux to larger tiles.
\end{remark}

\begin{lemma}\label{tabisskew}
Let \((\Lambda, \ttt)\) be a tableau for a skew shape \(\tau\). For \(1 \leq h \leq \ell \leq |\Lambda|\), define \(\tau_{h,\ell} = \bigsqcup_{i = h}^\ell \ttt(i)\). Then \(\tau_{h,\ell}\) is a skew shape.
\end{lemma}
\begin{proof}
Let \(u, w \in \tau_{h,\ell}\), \(v \in \N\), with \(u \searrow v\searrow w\). Then \(u \in \ttt(i)\), \(w \in \ttt(k\)) for some \(i,k \in [h, \ell]\). As \(\tau\) is a skew shape, we have \(v \in \tau\). Then \(v \in \ttt(j)\) for some \(j \in [1, |\Lambda|]\). Since \(u \searrow v\) we have \(h \leq i \leq j\), and since \(v \searrow w\) we have \(j \leq k\leq \ell\). But then \(j \in [h, \ell]\), so \(v \in \tau_{h, \ell}\). Thus \(\tau_{h, \ell}\) is a skew shape.
\end{proof}

\subsection{Removable ribbons}
We say that a skew shape \(\mu \subseteq \tau\) is {\em \(\ssee\)-removable in \(\tau\)} if \(\mu = \tau\) or \((\tau\backslash\mu, \mu)\) is a tableau for \(\tau\). We likewise say that \(\mu\) is {\em \(\nnww\)-removable in \(\tau\)} if \(\mu = \tau\) or \((\mu,\tau\backslash\mu)\) is a tableau for \(\tau\).
Roughly speaking, this means that \(\mu\) is \(\ssee\)-removable if it can be `slid away' to the southeast without colliding with \(\tau\backslash\mu\).
The next lemma follows directly from definitions.
\begin{lemma}\label{remcond}
Let \(\mu \subseteq \tau\) be skew shapes. Then \(\mu\) is \(\ssee\)-removable in \(\tau\) if and only if there does not exist \(u \in \mu\), \(v \in \tau\backslash \mu\) such that \(u \searrow v\). 
\end{lemma}

Let \(\tau\) be a skew shape, and define \(\textup{Rem}_\tau\) to be the set of all pairs \((u,v) \in \tau\) such that
\(u,v\) are in the same connected component of \(\tau\), \(u \nearrow v\), and \(\so u, \ea v \notin \tau\).
Then for \((u,v) \in \textup{Rem}_\tau\), define 
\begin{align*}
\xi^\tau_{u,v} := \{w \in \tau \mid u \nearrow w \nearrow v, \ssee w \notin \tau\} \subseteq \tau.
\end{align*}
An example of the set \(\xi^\tau_{u,v}\) for \((u,v) \in \textup{Rem}_\tau\) is shown in Figure~\ref{fig:xiuvex}.

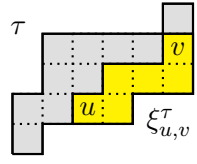
\begin{figure}[h]
\begin{align*}
{}
\hackcenter{
\begin{tikzpicture}[scale=0.4]
%
\draw[thick, fill=lightgray!50]   (0,-1)--(1,-1)--(1,0)--(4,0)--(4,1)--(6,1)--(6,4)--(5,4)--(5,3)--(1,3)--(1,2)--(1,1)--(0,1)--(0,-1);
\draw[thick, fill=yellow] (2,0)--(4,0)--(4,1)--(6,1)--(6,3)--(5,3)--(5,2)--(3,2)--(3,1)--(2,1)--(2,0);
%
\draw[thick, dotted] (1,0)--(1,1);
\draw[thick, dotted] (2,0)--(2,3);
\draw[thick, dotted] (3,0)--(3,3);
\draw[thick, dotted] (4,0)--(4,3);
\draw[thick, dotted] (5,1)--(5,3);
\draw[thick, dotted] (0,0)--(4,0);
\draw[thick, dotted] (0,1)--(6,1);
\draw[thick, dotted] (1,2)--(6,2);
\draw[thick, dotted] (2,3)--(6,3);
\draw[thick]   (0,-1)--(1,-1)--(1,0)--(4,0)--(4,1)--(6,1)--(6,4)--(5,4)--(5,3)--(1,3)--(1,2)--(1,1)--(0,1)--(0,-1);
\draw[thick] (2,0)--(4,0)--(4,1)--(6,1)--(6,3)--(5,3)--(5,2)--(3,2)--(3,1)--(2,1)--(2,0);
\node at (0.2,3.2){ $\tau$};
\node at (5.2,0){ $\xi_{u,v}^\tau$};
\node at (2.5,0.5){ $u$};
\node at (5.5,2.5){ $v$};
%
\end{tikzpicture}
}
\end{align*}
\caption{\(\ssee\)-removable ribbon \(\xi^\tau_{u,v}\) in skew shape \(\tau\)}
\label{fig:xiuvex}       
\end{figure}

\begin{lemma}\label{xiuv}
Let \(\tau\) be a skew shape. Then \(\{ \xi^\tau_{u,v} \mid (u,v) \in \textup{Rem}_\tau\}\) is a complete set of \(\ssee\)-removable ribbons in \(\tau\).
\end{lemma}
\begin{proof}
Let \((u,v) \in \textup{Rem}_\tau\). 
Since \(\tau\) is a skew shape, it is cornered and diagonal-convex by Proposition~\ref{CSS=PDC}. Since \(\tau\) is diagonal-convex, for every \(n \in \NN\) there is at most one \(w \in \D_n \cap \tau\) such that \(\ssee w \notin \tau\). Thus \(\xi^\tau_{u,v}\) is thin, and therefore also diagonal-convex. 

We now show that \(\xi^\tau_{u,v}\) is cornered.
Note that since \(\so u \notin \tau\), we have that \(\ssee u \notin \tau\) because \(\tau\) is a skew shape. Therefore \(u\) is a maximally southwest node in \(\xi^\tau_{u,v}\), and similarly, \(v\) is a maximally northeast node in \(\xi^\tau_{u,v}\). 
Let \(w \neq u \in \xi^\tau_{u,v}\). 
Then by Proposition~\ref{pathcrit} we have that \(\{\so w, \we w\} \cap \tau \neq \varnothing\). First assume that \(\so w \notin \tau\), so that \(\we w \in \tau\). It must be that \(w_2> u_2\) since \(\tau\) is a skew shape. Thus \(u \nearrow \we w \nearrow w \nearrow v\), and since \(\ssee \we w = \so w \notin \tau\), we have that \(\we w \in \xi^\tau_{u,v}\) by the definition of \(\xi^\tau_{u,v}\). On the other hand, assume that \(\so w \in\tau\). Since \(\so u \notin \tau\), we must have that \(u_1 > w_1\) since \(\tau\) is a skew shape. Thus \(u \nearrow \so w \nearrow w \nearrow v\). As \(w \in \xi^\tau_{u,v}\), we have that \(\ssee w \notin \tau\). Therefore \(\ssee \so w = \so \ssee w \notin \tau\) since \(\tau\) is a skew shape. Thus we have that \(\so w \in \xi^\tau_{u,v}\) by the definition of \(\xi^\tau_{u,v}\). So, in any case, we have that \(\{\so w, \we w \} \cap \xi^\tau_{u,v} \neq \varnothing\). We may show in a similar fashion that \(\{\no w, \ea w \} \cap \xi^\tau_{u,v} \neq \varnothing\) as well, so we have that \(\xi^\tau_{u,v}\) is cornered. 

As \(\xi^\tau_{u,v}\) is thin, cornered, and diagonal-convex, we have that it is a ribbon by Proposition~\ref{CSS=PDC}. Now we show that \(\xi^\tau_{u,v}\) is \(\ssee\)-removable. Let \(w \in \xi^\tau_{u,v}\), \(w \neq z \in \tau\), and assume that \(w \searrow z\). 
Since \(\ssee w \notin \tau\), we cannot have that \(w \SEarrow z\) because \(\tau\) is a skew shape, and we similarly have that \(\ssee z \notin \tau\). Therefore we have either that \(w_1 = z_1, w_2< z_2\), or that \(w_2=z_2, w_1<z_1\). Assume \(w_1 = z_1, w_2<z_2\). Then \(u \nearrow z\). If \(z_2>v_2\), then we would have \(v \searrow \ea v \searrow z\), a contradiction since \(\tau\) is a skew shape and \(\ea v \notin \tau\) by assumption. Thus \(v_2 \leq z_2\), so \(u \nearrow z \nearrow v\). On the other hand, assume that \(w_2=z_2, w_1<z_1\). Then \(z \nearrow v\). If \(u_1 < z_1\), then we would have \(u \nearrow \so u \nearrow z\), a contradiction since \(\tau\) is a skew shape and \(\so u \notin \tau\) by assumption. Thus \(u_1 \geq z_1\). Thus \(u \nearrow z \nearrow v\). In any case then, we have that \(\ssee z \notin \tau\) and \(u \nearrow z \nearrow v\), so \(z \in \xi^\tau_{u,v}\). Therefore we have by Lemma~\ref{remcond} that \(\xi^\tau_{u,v}\) is an \(\ssee\)-removable ribbon.

Now assume that \(\xi\) is some \(\ssee\)-removable ribbon in \(\tau\). Since \(\xi\) is connected, \(\xi\) must be contained entirely within one connected component, and so \(\xi_{\ssww}\) and \(\xi_{\nnee}\) are in the same connected component of \(\tau\). Since \(\xi\) is a connected skew shape, we have that \(\xi_{\ssww} \nearrow w \nearrow \xi_{\nnee}\) for all \(w \in \xi\) by Lemma~\ref{uniquemax}. 
Note that \(\so \xi_{\ssww} \notin \xi\) by the definition of \(\xi_{\ssww}\), and by Lemma~\ref{remcond} we have \(\xi_{\ssww} \notin \tau \backslash \xi\) because \(\xi\) is \(\ssee\)-removable in \(\tau\). Thus \(\so \xi_{\ssww} \notin \tau\). It may similarly be shown that \(\we \xi_{\nnee} \notin \tau\). Therefore \((\xi_{\ssww}, \xi_{\nnee}) \in \textup{Rem}_\tau\).

We now show that \(\xi = \xi^\tau_{\xi_{\ssww}, \xi_{\nnee}}\). Let \(w \in \xi\). Then \(w \in \tau\) and \(\xi_{\ssww} \nearrow w \nearrow \xi_{\nnee}\) as previously noted. We have \(\ssee w \notin \xi\) because \(\xi\) is thin, and we have \(\ssee w \notin \tau \backslash \xi\) by Lemma~\ref{remcond} because \(\xi\) is \(\ssee\)-removable in \(\tau\). Therefore \(\ssee w \notin \tau\), and thus \(w \in  \xi^\tau_{\xi_{\ssww}, \xi_{\nnee}}\), so \(\xi \subseteq  \xi^\tau_{\xi_{\ssww}, \xi_{\nnee}}\). By construction, \((\xi^\tau_{\xi_{\ssww}, \xi_{\nnee}})_{\ssww} = \xi_{\ssww}\), and \((\xi^\tau_{\xi_{\ssww}, \xi_{\nnee}})_{\nnee} = \xi_{\nnee}\), so by Corollary~\ref{sizehook} we have
\(
|\xi|  = \dist(\xi_{\ssww}, \xi_{\nnee}) +1 = |\xi^\tau_{\xi_{\ssww}, \xi_{\nnee}}|,
\)
which gives the equality \(\xi = \xi^\tau_{\xi_{\ssww}, \xi_{\nnee}}\), completing the proof.
\end{proof}

 \begin{lemma}\label{remhookswap}
Let \(\tau' \subsetneq \tau\) be skew shapes, let \(u,v,z,w\) be nodes in the same connected component of \(\tau\), and assume \((u,v) \in \textup{Rem}_\tau\), \(\tau' = \tau \backslash \xi^{\tau}_{u,v}\), and \((z,w) \in \textup{Rem}_{\tau'}\). Then we have the following implications:
\begin{enumerate}
\item \(u \nearrow v \nearrow z \nearrow w \implies (u,w) \in \textup{Rem}_{\tau}\);
\item \(z \nearrow w \nearrow u \nearrow v \implies (z,v) \in \textup{Rem}_\tau\);
\item \(u \nearrow z \nearrow v \nearrow w \implies (\ssee z, w) \in \textup{Rem}_{\tau}\);
\item \(z \nearrow u \nearrow w \nearrow v \implies (z, \ssee w) \in \textup{Rem}_{\tau}\);
\item \(u \nearrow z \nearrow w \nearrow v \implies (\ssee z, \ssee w) \in \textup{Rem}_\tau\);
\item \(z \nearrow u \nearrow v \nearrow w \implies (z,w) \in \textup{Rem}_\tau\);
\end{enumerate}
 \end{lemma}
 \begin{proof}
For ease of visualizing of these relationships, examples of cases (i)--(vi) are shown in Figure~\ref{fig:overlaps}.
We first prove a number of claims:
\begin{align}
v \nearrow w & \implies \ea w \notin \tau \label{vw}\\
z \nearrow u & \implies \so z \notin \tau \label{zu}\\
u \nearrow z \nearrow v &\implies \ssee z \in \tau \textup{ and } \so \ssee z \notin \tau \label{uzv}\\
u \nearrow w \nearrow v & \implies \ssee w  \in \tau \textup{ and } \ea \ssee w  \notin \tau. \label{uwv}
\end{align} 
Since \((z,w) \in \textup{Rem}_{\tau'}\), we must have \(\ea w, \so z \notin \tau'\). Thus, if \(v \nearrow w\), we have \(\ea w \notin \tau\) since \(x \nearrow v \nearrow w\) for all \(x \in \xi_{u,v}^\tau\), and if \(z \nearrow u\), we have \(\so z \notin \tau\) since \(z \nearrow u \nearrow x\) for all \(x \in \xi_1\), proving (\ref{vw}, \ref{zu}).

For (\ref{uzv}), assume that \(u \nearrow z \nearrow v\). We have \(z \in \tau' = \tau \backslash \xi^\tau_{u,v}\), so we must have \(\ssee z \in \tau\), else \(z\) would by definition belong to \(\xi^\tau_{u,v}\). Then we must have \(\so z \in \tau\) as well. But since \((z,w) \in \textup{Rem}_{\tau'}\), we must have that \(\so z \notin \tau'\), so it follows that \(\so z \in \xi^{\tau}_{u,v}\). Then \(\so \ssee z = \ssee \so z \notin \tau\). 

For (\ref{uwv}), assume that \(u \nearrow w \nearrow v\). We have \(w \in \tau' = \tau \backslash \xi^\tau_{u,v}\), so we must have \(\ssee w \in \tau\), else \(w\) would by definition belong to \(\xi^\tau_{u,v}\). Then we must have \(\ea w \in \tau\) as well. But since \(
(z,w) \in \textup{Rem}_{\tau'}\), we must have that \(\ea w \notin \tau'\), so it follows that \(\ea w \in \xi^{\tau}_{u,v}\). Then \(\ea \ssee w = \ssee \ea w \notin \tau\). 

Now for each pair of nodes \((a,b)\) on the right side of any of the implication statements in (i)--(vi), it is straightforward to verify via (\ref{vw}--\ref{uwv}) that \(a,b \in \tau\), \(a \nearrow b\), and \(\so a, \ea b \notin \tau\), so that \((a,b) \in \textup{Rem}_\tau\). 
\end{proof}

\begin{figure}[h]
\begin{align*}
{}
\hackcenter{
\begin{tikzpicture}[scale=0.4]
%
\draw[thick]   (1,0)--(2,0);
\draw[thick]   (7,5)--(7,6);
\draw[thick, fill=cyan!25]   (2,0)--(4,0)--(4,1)--(6,1)--(6,3)--(7,3)--(7,5)--(6,5)--(6,4)--(5,4)--(5,2)--(3,2)--(3,1)--(2,1)--(2,0);
\draw[thick, fill=yellow] (2,0)--(4,0)--(4,1)--(6,1)--(6,2)--(3,2)--(3,1)--(2,1)--(2,0);
%
\draw[thick, dotted] (3,0)--(3,5);
\draw[thick, dotted] (4,0)--(4,5.5);
\draw[thick, dotted] (5,1)--(5,5.75);
\draw[thick, dotted] (6,1)--(6,6);
\draw[thick, dotted] (1,1)--(6,1);
\draw[thick, dotted] (1.25,2)--(6,2);
\draw[thick, dotted] (1.5,3)--(6,3);
\draw[thick, dotted] (2,4)--(7,4);
\draw[thick]   (2,0)--(4,0)--(4,1)--(6,1)--(6,3)--(7,3)--(7,5)--(6,5)--(6,4)--(5,4)--(5,2)--(3,2)--(3,1)--(2,1)--(2,0);
\draw[thick] (2,0)--(4,0)--(4,1)--(6,1)--(6,2)--(3,2)--(3,1)--(2,1)--(2,0);
\node at (2,5){ $\tau$};
\node at (2.5,0.5){ $u$};
\node at (5.5,1.5){ $v$};
\node at (5.5,2.5){ $z$};
\node at (6.5,4.5){ $w$};
%
\end{tikzpicture}
}
\qquad
\hackcenter{
\begin{tikzpicture}[scale=0.4]
%
\draw[thick]   (1,0)--(2,0);
\draw[thick]   (7,5)--(7,6);
\draw[thick, fill=yellow]   (2,0)--(4,0)--(4,1)--(6,1)--(6,3)--(7,3)--(7,5)--(6,5)--(6,4)--(5,4)--(5,2)--(3,2)--(3,1)--(2,1)--(2,0);
\draw[thick, fill=cyan!25] (2,0)--(4,0)--(4,1)--(5,1)--(5,2)--(3,2)--(3,1)--(2,1)--(2,0);
%
\draw[thick, dotted] (3,0)--(3,5);
\draw[thick, dotted] (4,0)--(4,5.5);
\draw[thick, dotted] (5,1)--(5,5.75);
\draw[thick, dotted] (6,1)--(6,6);
\draw[thick, dotted] (1,1)--(6,1);
\draw[thick, dotted] (1.25,2)--(6,2);
\draw[thick, dotted] (1.5,3)--(6,3);
\draw[thick, dotted] (2,4)--(7,4);
\draw[thick]   (2,0)--(4,0)--(4,1)--(6,1)--(6,3)--(7,3)--(7,5)--(6,5)--(6,4)--(5,4)--(5,2)--(3,2)--(3,1)--(2,1)--(2,0);
\draw[thick] (2,0)--(4,0)--(4,1)--(5,1)--(5,2)--(3,2)--(3,1)--(2,1)--(2,0);
\node at (2,5){ $\tau$};
\node at (2.5,0.5){ $z$};
\node at (4.5,1.5){ $w$};
\node at (5.5,1.5){ $u$};
\node at (6.5,4.5){ $v$};
%
\end{tikzpicture}
}
\end{align*}
\begin{align*}
{}
\hackcenter{
\begin{tikzpicture}[scale=0.4]
%
\draw[thick]   (1,0)--(2,0);
\draw[thick]   (7,5)--(7,6);
\draw[thick, fill=cyan!25]   (2,0)--(4,0)--(4,1)--(6,1)--(6,3)--(7,3)--(7,5)--(6,5)--(6,4)--(4,4)--(4,3)--(3,3)--(3,1)--(2,1)--(2,0);
\draw[thick, fill=yellow] (2,0)--(4,0)--(4,1)--(6,1)--(6,3)--(5,3)--(5,2)--(3,2)--(3,1)--(2,1)--(2,0);
%
\draw[thick, dotted] (3,0)--(3,5);
\draw[thick, dotted] (4,0)--(4,5.5);
\draw[thick, dotted] (5,1)--(5,5.75);
\draw[thick, dotted] (6,1)--(6,6);
\draw[thick, dotted] (1,1)--(6,1);
\draw[thick, dotted] (1.25,2)--(6,2);
\draw[thick, dotted] (1.5,3)--(6,3);
\draw[thick, dotted] (2,4)--(7,4);
\draw[thick]   (2,0)--(4,0)--(4,1)--(6,1)--(6,3)--(7,3)--(7,5)--(6,5)--(6,4)--(4,4)--(4,3)--(3,3)--(3,1)--(2,1)--(2,0);
\draw[thick] (2,0)--(4,0)--(4,1)--(6,1)--(6,3)--(5,3)--(5,2)--(3,2)--(3,1)--(2,1)--(2,0);
\node at (2,5){ $\tau$};
\node at (2.5,0.5){ $u$};
\node at (5.5,2.5){ $v$};
\node at (3.5,2.5){ $z$};
\node at (6.5,4.5){ $w$};
%
\end{tikzpicture}
}
\qquad
\hackcenter{
\begin{tikzpicture}[scale=0.4]
%
\draw[thick]   (1,0)--(2,0);
\draw[thick]   (7,5)--(7,6);
\draw[thick, fill=cyan!25]   (2,0)--(4,0)--(4,1)--(6,1)--(6,3)--(7,3)--(7,5)--(6,5)--(6,4)--(4,4)--(4,3)--(3,3)--(3,1)--(2,1)--(2,0);
\draw[thick, fill=yellow] (4,1)--(6,1)--(6,3)--(7,3)--(7,5)--(6,5)--(6,4)--(5,4)--(5,2)--(4,2)--(4,1);
%
\draw[thick, dotted] (3,0)--(3,5);
\draw[thick, dotted] (4,0)--(4,5.5); 
\draw[thick, dotted] (5,1)--(5,5.75);
\draw[thick, dotted] (6,1)--(6,6);
\draw[thick, dotted] (1,1)--(6,1);
\draw[thick, dotted] (1.25,2)--(6,2);
\draw[thick, dotted] (1.5,3)--(6,3);
\draw[thick, dotted] (2,4)--(7,4);
\draw[thick] (2,0)--(4,0)--(4,1)--(6,1)--(6,3)--(7,3)--(7,5)--(6,5)--(6,4)--(4,4)--(4,3)--(3,3)--(3,1)--(2,1)--(2,0);
\draw[thick] (4,1)--(6,1)--(6,3)--(7,3)--(7,5)--(6,5)--(6,4)--(5,4)--(5,2)--(4,2)--(4,1);
\node at (2,5){ $\tau$};
\node at (2.5,0.5){ $z$};
\node at (4.5,3.5){ $w$};
\node at (4.5,1.5){ $u$};
\node at (6.5,4.5){ $v$};
%
\end{tikzpicture}
}
\qquad
\hackcenter{
\begin{tikzpicture}[scale=0.4]
%
\draw[thick]   (1,0)--(2,0);
\draw[thick]   (7,5)--(7,6);
\draw[thick, fill=cyan!25]   (2,0)--(4,0)--(4,1)--(6,1)--(6,3)--(7,3)--(7,5)--(6,5)--(6,4)--(4,4)--(4,3)--(3,3)--(3,1)--(2,1)--(2,0);
\draw[thick, fill=yellow] (3,2)--(3,1)--(2,1)--(2,0)--(4,0)--(4,1)--(6,1)--(6,3)--(7,3)--(7,5)--(6,5)--(6,4)--(5,4)--(5,2)--(4,2)--(3,2);
%
\draw[thick, dotted] (3,0)--(3,5);
\draw[thick, dotted] (4,0)--(4,5.5); 
\draw[thick, dotted] (5,1)--(5,5.75);
\draw[thick, dotted] (6,1)--(6,6);
\draw[thick, dotted] (1,1)--(6,1);
\draw[thick, dotted] (1.25,2)--(6,2);
\draw[thick, dotted] (1.5,3)--(6,3);
\draw[thick, dotted] (2,4)--(7,4);
\draw[thick] (2,0)--(4,0)--(4,1)--(6,1)--(6,3)--(7,3)--(7,5)--(6,5)--(6,4)--(4,4)--(4,3)--(3,3)--(3,1)--(2,1)--(2,0);
\draw[thick] (3,2)--(3,1)--(2,1)--(2,0)--(4,0)--(4,1)--(6,1)--(6,3)--(7,3)--(7,5)--(6,5)--(6,4)--(5,4)--(5,2)--(4,2)--(3,2);
\node at (2,5){ $\tau$};
\node at (2.5,0.5){ $u$};
\node at (4.5,3.5){ $w$};
\node at (3.5,2.5){ $z$};
\node at (6.5,4.5){ $v$};
%
\end{tikzpicture}
}
\qquad
\hackcenter{
\begin{tikzpicture}[scale=0.4]
%
\draw[thick]   (1,0)--(2,0);
\draw[thick]   (7,5)--(7,6);
\draw[thick, fill=yellow]   (2,0)--(4,0)--(4,1)--(6,1)--(6,3)--(7,3)--(7,5)--(6,5)--(6,4)--(4,4)--(4,3)--(3,3)--(3,1)--(2,1)--(2,0);
\draw[thick, fill=cyan!25] (2,0)--(4,0)--(4,2)--(5,2)--(5,3)--(7,3)--(7,5)--(6,5)--(6,4)--(4,4)--(4,3)--(3,3)--(3,1)--(2,1)--(2,0);
%
\draw[thick, dotted] (3,0)--(3,5);
\draw[thick, dotted] (4,0)--(4,5.5); 
\draw[thick, dotted] (5,1)--(5,5.75);
\draw[thick, dotted] (6,1)--(6,6);
\draw[thick, dotted] (1,1)--(6,1);
\draw[thick, dotted] (1.25,2)--(6,2);
\draw[thick, dotted] (1.5,3)--(6,3);
\draw[thick, dotted] (2,4)--(7,4);
\draw[thick]   (2,0)--(4,0)--(4,1)--(6,1)--(6,3)--(7,3)--(7,5)--(6,5)--(6,4)--(4,4)--(4,3)--(3,3)--(3,1)--(2,1)--(2,0);
\draw[thick] (2,0)--(4,0)--(4,2)--(5,2)--(5,3)--(7,3)--(7,5)--(6,5)--(6,4)--(4,4)--(4,3)--(3,3)--(3,1)--(2,1)--(2,0);
\node at (2,5){ $\tau$};
\node at (2.5,0.5){ $z$};
\node at (5.5,2.5){ $v$};
\node at (4.5,1.5){ $u$};
\node at (6.5,4.5){ $w$};
%
\end{tikzpicture}
}
\end{align*}
\caption{Cases (i)--(vi) in Lemma~\ref{remhookswap}}
\label{fig:overlaps}       
\end{figure}
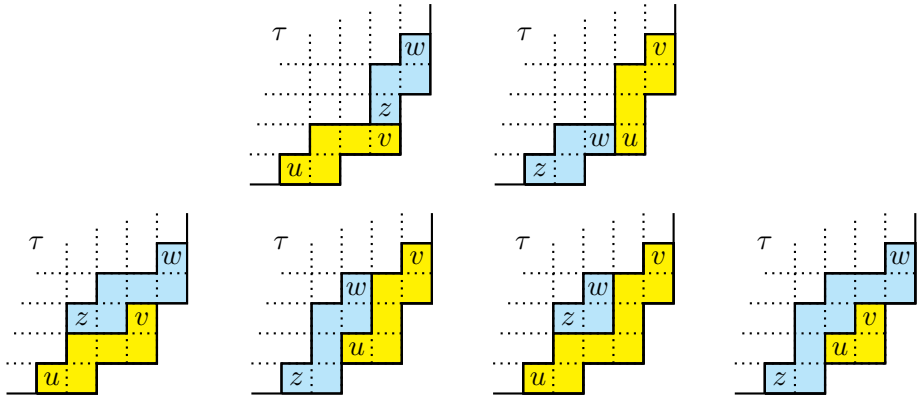

\section{Positive roots and convex preorders}\label{posrootsec}
We fix now and throughout the paper some choice of \(e \in \ZZ_{>1}\). Associated to \(e\) is an affine root system of type \({\tt A}_{e-1}^{(1)}\), corresponding to the affine Dynkin diagram shown in Figure~\ref{fig:dynkin}
(see \cite[\S 4, Table Aff 1]{Kac}). This root system plays a crucial role in the representation theory of the Kac-Moody Lie algebra \(\widehat{\mathfrak{sl}}_e\) and its associated quantum group, as well as modular representation theory of the symmetric group.
We describe the root system directly in the sequel.

\begin{figure}
\begin{align*}
{}
\hackcenter{
\begin{tikzpicture}[scale=0.4]
%
\coordinate (0) at (5,1);
\coordinate (1) at (0,-1);
\coordinate (2) at (2,-1);
\coordinate (3) at (4,-1);
\coordinate (4) at (6,-1);
\coordinate (5) at (8,-1);
\coordinate (6) at (10,-1);
\coordinate (0a) at (5,0.9);
\coordinate (1a) at (0,-1.1);
\coordinate (2a) at (2,-1.1);
\coordinate (3a) at (4,-1.1);
\coordinate (4a) at (6,-1.1);
\coordinate (5a) at (8,-1.1);
\coordinate (6a) at (10,-1.1);
\draw [thin, black,shorten <= 0.1cm, shorten >= 0.1cm]   (0) to (1);
\draw [thin, black,shorten <= 0.1cm, shorten >= 0.1cm]   (1) to (2);
\draw [thin, black,shorten <= 0.1cm, shorten >= 0.1cm]   (2) to (3);
\draw [thin, black,shorten <= 0.1cm, shorten >= 0.4cm]   (3) to (4);
\draw [thin, black,shorten <= 0.3cm, shorten >= 0.1cm]   (4) to (5);
\draw [thin, black,shorten <= 0.1cm, shorten >= 0.1cm]   (5) to (6);
\draw [thin, black,shorten <= 0.1cm, shorten >= 0.1cm]   (0) to (6);
\draw (0a) node[below]{$\scriptstyle{0}$};
\draw (1a) node[below]{$\scriptstyle{1}$};
\draw (2a) node[below]{$ \scriptstyle{2}$};
\draw (3a) node[below]{$\scriptstyle{3}$};
\draw (5a) node[below]{$ \scriptstyle{e-2}$};
\draw (6a) node[below]{$\scriptstyle{e-1}$};
\blackdot(5,1);
\blackdot(0,-1);
\blackdot(2,-1);
\blackdot(4,-1);
\blackdot(8,-1);
\blackdot(10,-1);
\draw(6,-1) node{$\cdots$};
\end{tikzpicture}
}
\end{align*}
\caption{Dynkin diagram of type \({\tt A}^{(1)}_{e-1}\)}
\label{fig:dynkin}       
\end{figure}
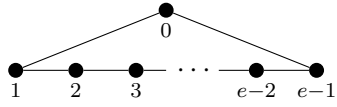

We write \(\ZZ_e = \ZZ/e\ZZ\), and will indicate elements of \(\ZZ_e\) with barred integers, i.e. \(\overline t = t + e\ZZ\) for \(t \in \ZZ\), freely omitting bars for \(t \in [0,e-1]\) when the context is clear.
Let \(\ZZ I\) be the free \(\ZZ\)-module of rank \(e\), with basis \(I = \{\alpha_i \mid i \in \ZZ_e\}\), and set \(Q_+ = \ZZ_{\geq 0}I\). For \(\beta = \sum_{i \in \Z_e} c_i\alpha_i \in Q_+\), we write \(\height(\beta) := \sum_{i \in \Z_e} c_i \in \ZZ_{\geq 0}\) for the {\em height} of \(\beta\).
For any \(t \in \ZZ_e\), \(h \in \NN\), we define \(\alpha(t,h) \in Q_+\) via
\begin{align*}
\alpha(t,h):= \alpha_t + \alpha_{t+\overline{1}} + \cdots + \alpha_{t+ \overline{h-1}}.
\end{align*}
We have then that \(\height(\alpha(t,h)) = h\). Of particular importance is the {\em null root} of height \(e\):
\begin{align*}
\delta := \alpha_{ 0} + \cdots + \alpha_{e-1} = \alpha(t,e)\qquad (t \in \ZZ_e).
\end{align*}
It follows from definitions that 
\begin{align}\label{splithookcont}
\alpha(t,h_1 + h_2) = \alpha(t, h_1) + \alpha(t + \overline{h}_1, h_2) 
\end{align}
for all \(t \in \ZZ_e, h_1,h_2 \in \NN\). The element \(\alpha(t,h)\) corresponds to the sum of simple roots in a counterclockwise path of length \(h\) in the Dynkin diagram in Figure~\ref{fig:dynkin} which begins at the vertex labeled \(t\).

\begin{definition} \(\)
\begin{enumerate}
\item
We say \(\beta \in Q_+\) is a {\em positive root} if \(\beta = \alpha(t,h)\) for some \(t \in \ZZ, h \in \NN\). We write \(\Phi_+\) for the set of all positive roots, so we have
\begin{align}
\Phi_+ = \{ \alpha(t,h) \mid t \in \ZZ_e, h \in \NN\} \subset \ZZ I.
\end{align}
\item We say \(\beta \in \Phi_+\) is {\em real} if \(\overline{\height(\beta)} \neq \overline 0\). Writing \(\Phi_+^\re\) for the set of all real positive roots, we have
\begin{align}
\Phi_+^\re =  \{ \alpha(t,h) \mid t \in \ZZ_e, h \in \NN, \overline h \neq \overline 0\} \subset \Phi_+.
\end{align}
\item We say \(\beta \in \Phi_+\) is {\em imaginary} if \(\overline{\height(\beta)} = \overline 0\).  Writing \(\Phi_+^\im\) for the set of imaginary positive roots, we have
\begin{align}\label{allimag}
\Phi_+^\im =  \{ \alpha(t,h) \mid t \in \ZZ_e, h \in \NN, \overline h = \overline 0\} = \{m \delta \mid m \in\NN \} \subset \Phi_+.
\end{align}
\item We say a positive root \(\beta\) is {\em divisible} if there exists \(\beta' \in \Phi_+\), \(m \in \ZZ_{>1}\) such that \(\beta= m \beta'\), and {\em indivisible} if not. Writing \(\Psi\) for the set of indivisible roots, we have
\begin{align}\label{PsiDef}
\Psi = \Phi_+^\re \sqcup \{ \delta\} .
\end{align}
\item We define \(\Phi_+'\) to be the set of all positive integer multiples of positive roots, so we have
\begin{align}\label{multphi}
\Phi_+' = \{m \beta \mid m \in \NN, \beta \in \Phi_+\} = \{m \beta \mid m \in \NN, \beta \in \Psi\} \subset Q_+.
\end{align}
\end{enumerate}
\end{definition}

\subsection{Convex preorders}
A {\em convex preorder} on \(\Phi_+\) is a binary relation \(\succeq\) on \(\Phi_+\) which, for all \(\beta, \gamma, \nu \in \Phi_+\) satisfies the following:
\begin{enumerate}
\item \(\beta \succeq \beta\) (reflexivity);
\item \(\beta \succeq \gamma\) and \(\gamma \succeq \nu\) imply \(\beta \succeq \nu\) (transitivity);
\item \(\beta \succeq \gamma\) or \(\gamma \succeq \beta\) (totality);
\item \(\beta \succeq \gamma\) and \(\beta + \gamma \in \Phi_+\) imply \(\beta \succeq \beta  + \gamma \succeq \gamma\) (convexity);
\item \(\beta \succeq \gamma\) and \(\gamma \succeq \beta\) if and only if \(\beta = \gamma\) or \(\beta, \gamma \in \Phi_+^\im\) (imaginary equivalency).
\end{enumerate}
We write \(\beta \succ \gamma\) if \(\beta \succeq \gamma\) and \(\gamma \not \succeq \beta\). Then (iii) and (v) together imply that \(\succ\) restricts to a total order on \(\Psi\). We also write \(\beta \approx \gamma\) if \(\beta \succeq \gamma\) and \(\gamma \succeq \beta\), so (v) and (\ref{allimag}) imply that \(\beta \approx \gamma\), \(\beta \neq \gamma\) if and only if \(\beta = m \delta\), \(\gamma = m'\delta\), for some \(m \neq m' \in \NN\).

Convex preorders are known to exist for any choice of \(e\), see \cite[Example 3.5, 3.6]{McN} for explicit constructions. One example of a convex preorder in the case \(e=3\) is shown in Example~\ref{BigEx}. A more straightforward example is the following:

\begin{example}\label{e2order}
Take \(e=2\). Below we show one of two possible convex preorders on \(\Phi_+\). The other is the reverse preorder (see \S\ref{revsec}).
\begin{align*}
\alpha_1 \succ
\delta+\alpha_1 
\succ
2\delta+\alpha_1
\succ
\cdots
\succ
m\delta
\succ
\cdots
\succ
2\delta+\alpha_0
\succ
\delta+\alpha_0
\succ
\alpha_0.
\end{align*}
\end{example}

\subsection{Implications of convexity} 
The next lemma follows immediately from definitions:
\begin{lemma}\label{notapprox}
Let \(\beta, \beta' \in \Phi_+\), Then we have:
\begin{enumerate}
\item If \(\beta \approx \beta'\) and \(\height(\beta) > \height(\beta')\), then \(\beta, \beta' \in \Phi_+^\im\);
\item If \(\beta \in \Psi\) and \(\height(\beta) > \height(\beta')\), then \(\beta \not \approx \beta'\);
\item If \(\beta, \beta' \in \Psi\), then \(\beta \approx \beta'\) if and only if \(\beta = \beta'\).
\end{enumerate}
\end{lemma}

We have the following useful generalization of the convexity property:

\begin{lemma} \label{GenConvexLem} \cite{KCusp, BKT}
Let \(\gamma, \beta_1, \ldots, \beta_k \in \Phi_+\), \(m \in \NN\) be such that \(\beta_1 + \cdots + \beta_k = m \gamma\). Then we have the following:
\begin{enumerate}
\item If \(\gamma \in \Phi_+^\re\), and \(\beta_i \succeq \gamma\) for all \(i \in [1,k]\) or \(\gamma \succeq \beta_i\) for all \(i \in [1,k]\), then \(\gamma = \beta_i\) for all \(i \in [1,k]\).
\item If \(\gamma \in \Phi_+^\im\), and \(\beta_i \succeq \gamma\) for all \(i \in [1,k]\) or \(\gamma \succeq \beta_i\) for all \(i \in [1,k]\), then \(\beta_i \in \Phi_+^\im\) for all \(i \in [1,k]\).
\item If \(\beta_1 \succeq \cdots \succeq \beta_k\), then \(\beta_1 \succeq \gamma \succeq \beta_k\).
\item If \(\beta_1 \succeq \cdots \succeq \beta_k\) and \(\beta_1 \succ \beta_k\), then \(\beta_1 \succ \gamma \succ \beta_k\).
\end{enumerate}
\end{lemma}
\begin{proof}
Claims (i) and (ii) appear as the properties \cite[(Con1), (Con3)]{KCusp}, which follow from \cite[Lemma 3.1]{KCusp}, which utilizes \cite[Lemma 2.9]{BKT}. Claims (iii) and (iv) are contrapositive reformulations implied by (i) and (ii).
\end{proof}

\begin{lemma} \label{psidef} 
The function \(\NN \times \Psi \to \Phi_+'\), \((n, \beta) \mapsto n \beta\) is a bijection.
\end{lemma}
\begin{proof}
Surjectivity is clear by (\ref{multphi}). For injectivity, assume \(n \beta = n' \beta'\) for some \(n,n' \in \NN\), \(\beta,\beta' \in \Psi\). Then by Lemma~\ref{GenConvexLem} we have \(\beta \approx \beta'\). As \(\beta, \beta' \in \Psi\), this implies that \(\beta = \beta'\) and thus \(n = n'\).
\end{proof}

By Lemma~\ref{psidef} we have well-defined functions \(m:\Phi_+' \to \NN\) and \(\psi: \Phi_+' \to \Psi\) such that \(\gamma = m(\gamma) \psi(\gamma)\) for all \(\gamma \in \Phi_+'\).

\subsection{Kostant partitions}
For \(\theta \in Q_+\), a {\em Kostant partition} of \(\gamma\) is a tuple of non-negative integers \(\bkap = (\kappa_\beta)_{\beta \in \Psi}\) such that \(\sum_{\beta \in \Psi} \kappa_\beta \beta = \theta\). If \(\beta_1 \succ \cdots \succ \beta_k\) are the members of \(\Psi\) such that \(\kappa_{\beta_1} \neq 0\), then it is convenient to write \(\bkap\) in the form
\begin{align*}
\bkap = 
(\beta_1^{\kappa_{\beta_1}} \mid \ldots \mid \beta_k^{\kappa_{\beta_k}}).
\end{align*}

We write \(\Xi(\theta)\) for the set of all Kostant partitions of \(\theta\).
The convex preorder \(\succeq\) on \(\Phi_+\) restricts to a total order \(\succ\) on \(\Psi\), which induces a right lexicographic total order \(\triangleright_R\) on \(\Xi(\beta)\), where
\begin{align*}
\bkap \triangleright_R \bkap'
\iff
\textup{there exists \(\beta \in \Psi\) such that }
\kappa_\beta > \kappa'_\beta \textup{ and } \kappa_\nu = \kappa'_\nu \textup{ for all } \beta \succ \nu.
\qquad
\end{align*}
We also have a 
left lexicographic total order \(\triangleright_L\) on \(\Xi(\beta)\), where
\begin{align*}
\bkap \triangleright_L \bkap'
\iff
\textup{there exists \(\beta \in \Psi\) such that }
\kappa_\beta > \kappa'_\beta \textup{ and } \kappa_\nu = \kappa'_\nu \textup{ for all } \beta \prec \nu.
\qquad
\end{align*}
Then the orders \(\triangleright_R\), \(\triangleright_L\) induce a bilexicographic partial order \(\trianglerighteq\) on \(\Xi(\beta)\), given by
\begin{align*}
\bkap \trianglerighteq \bkap'
\qquad
\iff
\qquad
\bkap \trianglerighteq_R \bkap'
\textup{ and }
\bkap \trianglerighteq_L \bkap'.
\end{align*}

For a sequence \(\bbeta = (\beta_i)_{i=1}^k\) of elements of \(\Phi_+'\), we say \(\bbeta\) is a {\em Kostant sequence} if \(\psi(\beta_i) \succeq \psi(\beta_j)\) whenever \(i \leq j\).
To any Kostant sequence \(\bbeta = (\beta_i)_{i=1}^k\) we may associate a Kostant partition \(\bkap^{\bbeta}\) by setting
\begin{align*}
\kappa^{\bbeta}_\nu = \sum_{ i\in [1,k], \psi(\beta_i) = \nu} m(\beta_i) \qquad \qquad(\nu \in \Psi).
\end{align*}

\section{Content and cuspidal skew shapes}

\subsection{Content}\label{contentsec}
We define the {\em residue} of a node \(u \in \N\) to be 
\(\res(u) =\overline{u_2 - u_1} = \overline{\diag(u)}\).
The map \(\res: \N \to \Z_e\) is a \(\Z\)-module homomorphism, so we have
\begin{align*}
\res(n u) = n \res(u)
\qquad
\textup{and}
\qquad
\res(u+v) = \res(u) + \res(v),
\end{align*}
for all \(n \in \ZZ\), \(u,v \in \N\). For \(t \in \Z_e\), we write
\begin{align*}
\N_t := \{ u \in \N \mid \res(u) = t\}.
\end{align*}

The {\em content} of a skew shape \(\tau\) is defined as
\begin{align*}
\cont(\tau) = \sum_{u \in \tau} \alpha_{\res(u)} \in Q_+.
\end{align*}
For \(\theta \in Q_+\), we write \(\SSSS(\theta)\) (resp. \(\SSSS_{\tt c}(\theta)\)) for the set of skew shapes (resp. nonempty connected skew shapes) of content \(\theta\). It is clear from definitions then that if \(\Lambda\) is any tiling of \(\tau\), we have
\(
\cont(\tau) = \sum_{\lambda \in \Lambda} \cont(\lambda).
\)
In our visual representation, we label each node with its associated residue. See Figure~\ref{fig:e4con} for examples.

\begin{figure}[h]
\begin{align*}
{}
\hackcenter{
\begin{tikzpicture}[scale=0.4]
%
\draw[thick, fill=lightgray!50]  (0,-1)--(1,-1)--(1,0)--(4,0)--(4,1)--(6,1)--(6,4)--(5,4)--(5,3)--(1,3)--(1,2)--(1,1)--(0,1)--(0,-1);
%
\draw[thick, dotted] (1,0)--(1,1);
\draw[thick, dotted] (2,0)--(2,3);
\draw[thick, dotted] (3,0)--(3,3);
\draw[thick, dotted] (4,0)--(4,3);
\draw[thick, dotted] (5,1)--(5,3);
\draw[thick, dotted] (0,0)--(4,0);
\draw[thick, dotted] (0,1)--(6,1);
\draw[thick, dotted] (1,2)--(6,2);
\draw[thick, dotted] (2,3)--(6,3);
\draw[thick]   (0,-1)--(1,-1)--(1,0)--(4,0)--(4,1)--(6,1)--(6,4)--(5,4)--(5,3)--(1,3)--(1,2)--(1,1)--(0,1)--(0,-1);
\node at (0.5,-0.5){$\scriptstyle0 $};
\node at (0.5,0.5){$\scriptstyle1 $};
\node at (1.5,0.5){$\scriptstyle2 $};
\node at (2.5,0.5){$\scriptstyle0 $};
\node at (3.5,0.5){$\scriptstyle1 $};
\node at (1.5,1.5){$\scriptstyle0 $};
\node at (2.5,1.5){$\scriptstyle1 $};
\node at (3.5,1.5){$\scriptstyle2 $};
\node at (4.5,1.5){$\scriptstyle0 $};
\node at (5.5,1.5){$\scriptstyle1 $};
\node at (1.5,2.5){$\scriptstyle1 $};
\node at (2.5,2.5){$\scriptstyle2 $};
\node at (3.5,2.5){$\scriptstyle0 $};
\node at (4.5,2.5){$\scriptstyle1 $};
\node at (5.5,2.5){$\scriptstyle2 $};
\node at (5.5,3.5){$\scriptstyle0 $};
\end{tikzpicture}
}
\qquad
\qquad
\hackcenter{
\begin{tikzpicture}[scale=0.4]
%
\draw[thick, fill=lightgray!50]  (0,-1)--(1,-1)--(1,0)--(4,0)--(4,1)--(6,1)--(6,4)--(5,4)--(5,2)--(3,2)--(3,1)--(1,1)--(0,1)--(0,-1);
%
\draw[thick, dotted] (1,0)--(1,1);
\draw[thick, dotted] (2,0)--(2,1);
\draw[thick, dotted] (3,0)--(3,1);
\draw[thick, dotted] (4,1)--(4,2);
\draw[thick, dotted] (5,1)--(5,2);
\draw[thick, dotted] (0,0)--(4,0);
\draw[thick, dotted] (0,1)--(6,1);
\draw[thick, dotted] (3,2)--(6,2);
\draw[thick, dotted] (5,3)--(6,3);
\draw[thick]  (0,-1)--(1,-1)--(1,0)--(4,0)--(4,1)--(6,1)--(6,4)--(5,4)--(5,2)--(3,2)--(3,1)--(1,1)--(0,1)--(0,-1);
\node at (0.5,-0.5){$\scriptstyle0 $};
\node at (0.5,0.5){$\scriptstyle1 $};
\node at (1.5,0.5){$\scriptstyle2 $};
\node at (2.5,0.5){$\scriptstyle0 $};
\node at (3.5,0.5){$\scriptstyle1 $};
\node at (3.5,1.5){$\scriptstyle2 $};
\node at (4.5,1.5){$\scriptstyle0 $};
\node at (5.5,1.5){$\scriptstyle1 $};
\node at (5.5,2.5){$\scriptstyle2 $};
\node at (5.5,3.5){$\scriptstyle0 $};
\end{tikzpicture}
}
\end{align*}
\caption{Residues for \(e=3\). Skew shape of content \(6 \alpha_0 + 6 \alpha_1 + 4 \alpha_2\); ribbon of content \(4 \alpha_0 + 3 \alpha_1 + 3 \alpha_2 = \alpha(\overline 0, 10)\)}
\label{fig:e4con}       
\end{figure}
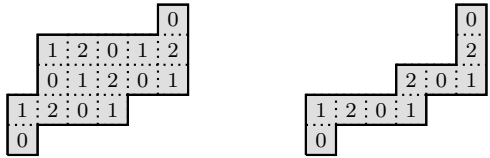

\subsection{Similarity}\label{similaritysec}
If \(\tau, \nu \in \SSSS_{\tt c}\) are such that \(\T_d \tau = \nu\) for some \(d \in \N\), then we will write \(\tau \sim \nu\), and say that \(\tau, \nu\) are {\em similar}.

We now consider a refinement of this similarity condition which takes into account the \(e\)-modular residues.
If \(c \in \N_{\overline 0}\), then we have \(\res(\T_c u) = \res(u)\) for all \(u \in \N\). If \(\nu, \omega\) are skew shapes such that \(\T_c \nu= \omega\) for some \(c \in \N_{\overline{0}}\), then we say \(\omega\) is a {\em residue-preserving translation} of \(\nu\).

Let \(\bmu = (\mu_1, \ldots, \mu_\ell) \in \SSSS_{\tt c}^\ell\). We then define the {\em \(e\)-similarity class} \([\bmu]_e = [\mu_1, \ldots, \mu_\ell]_e \subset \SSSS\) as follows. Let \(\tau\) be a nonempty skew shape, with connected component decomposition \(\tau = \tau_1 \sqcup \cdots \sqcup \tau_k\) with \(\tau_k \NEarrow \cdots \NEarrow \tau_1\) as in Proposition~\ref{skewshapecondecomp}. We say \(\tau \in [\bmu]_e\) provided that \(k = \ell\) and \(\tau_i\) is a residue-preserving translation of \(\mu_i\) for all \(i \in [1,\ell]\). We write \(\nu \sim_e \omega\) and say that \(\nu, \omega\) are {\em \(e\)-similar} if \(\nu,\omega \in [\bmu]_e\) for some \(\bmu\). Roughly speaking, 
\(e\)-similar skew shapes \(\nu\sim_e\omega\) consist of the same connected `shapes', in the same top-to-bottom order, with corresponding nodes in each having the same residue; only the absolute position and spacing between connected components in each may differ. Clearly \(\nu \sim_e \omega\) implies \(\cont(\nu) = \cont(\omega)\). As \(\N\) is infinite and residues are cyclic, it is easy to see that \([\bmu]_e\) is nonempty for every sequence of nonempty connected skew shapes \(\bmu\).

For classification purposes, we treat skew shapes in the same \(e\)-similarity class as identical (importantly, they have isomorphic associated Specht modules up to grading shift, see \S\ref{Spechtmodsec}).
When we display a skew shape \(\tau\) with residues, as in Figure~\ref{fig:e4con}, we are generally displaying the \(e\)-similarity class of the skew shape, disregarding the absolute position of \(\tau\) in \(\N\).

\subsection{Other combinatorial formulations}\label{otherformsec}
The combinatorial setup for skew shapes, residues, and contents outlined in this paper is convenient for present purposes, but bears slight differences with other approaches to the subject. Our formulation fixes a residue function, and allows the position of diagrams to vary, while in some other formulations, the position of diagrams is fixed and the residue function varies. We explain how we translate between formulations for the benefit of readers familiar with the subject.

It is common to define Young diagrams \( \lambda\) as subsets of \( \NN \times \NN\), where \(\lambda_{\nnww} = (1,1)\); see Remark~\ref{YDrem1}. In this setting, a {\em charge} \(c \in \ZZ_e\) is chosen, residues of nodes in \(\lambda\) are given by \(\res(u) = \overline{u_2 - u_1} + c\), and a skew shape (better called a skew Young diagram here) is defined via set difference \(\lambda / \mu := \lambda \backslash \mu\) for some Young diagrams \(\lambda, \mu\). One may transport the skew Young diagram \(\lambda/\mu\) defined thusly into a skew shape in this paper's combinatorial setting by placing the nodes of \(\lambda\backslash\mu\) in \(\N\) and suitably translating this shape as needed in order to match residues (e.g., shifting \(c\) units to the east to account for the charge \(c\)). All such choices of placement are equivalent up to \(e\)-similarity.

More generally, we may consider a {\em level \(\ell\)} skew Young diagram \(\blam/\bmu = (\lambda^{(1)}/\mu^{(1)}, \ldots, \lambda^{(\ell)}/\mu^{(\ell)})\), with {\em multicharge} \(\bc = (c_1, \ldots, c_\ell) \in \Z_e^\ell\), as in \cite{KMR, HM, Muth}. We associate this with a sequence \(\btau = (\tau_1, \ldots, \tau_\ell)\) of skew shapes in \(\N\) as in the previous paragraph, such that \(\tau_\ell \NEarrow \cdots \NEarrow \tau_1\), thereby associating \(\blam/\bmu\) with \(\tau:= \tau_1 \sqcup \cdots \sqcup \tau_\ell\), see Figure~\ref{fig:Muth1}. Again, any such choice of \(\tau\) is equivalent up to \(e\)-similarity.

In the other direction, we may associate a skew shape \(\tau\) in our combinatorial setting to a skew Young diagram and charge in a straightforward manner. There exists some Young diagrams \(\lambda^\tau, \mu^\tau\) such that \((\mu^\tau, \tau)\) is a tableau for \(\lambda^\tau\). For instance, one may take
\begin{align*}
\lambda^\tau = \{ u \in \N \mid ((\tau_{\nnee})_1, (\tau_{\ssww})_2) \searrow u \searrow v, \textup{ for some }v \in \tau\},
\;\;
\textup{and}
\;\;
\mu^\tau = \lambda^\tau \backslash \tau.
\end{align*}
Then we may associate \(\tau\) with the skew Young diagram \(\lambda^\tau / \mu^\tau\) with charge \(c = \res(\lambda^\tau_{\nnww})\); see Figure~\ref{fig:Muth1}.

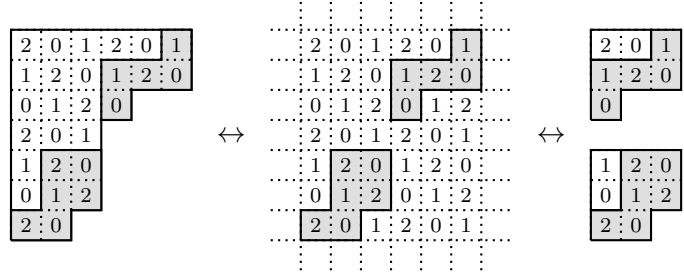
\begin{figure}[h]
\begin{align*}
{}
\hackcenter{
\begin{tikzpicture}[scale=0.4]
%
\draw[thick]  (0,0)--(2,0)--(2,1)--(3,1)--(3,4)--(4,4)--(4,5)--(6,5)--(6,7)--(0,7)--(0,0);
\draw[thick, fill=lightgray!50]  (0,0)--(2,0)--(2,1)--(3,1)--(3,3)--(1,3)--(1,1)--(0,1)--(0,0);
\draw[thick, fill=lightgray!50]  (0+3,4)--(1+3,4)--(1+3,5)--(3+3,5)--(3+3,7)--(2+3,7)--(2+3,6)--(0+3,6)--(0+3,4);
%
\draw[thick, dotted] (0,1)--(2,1);
\draw[thick, dotted] (0,2)--(3,2);
\draw[thick, dotted] (0,3)--(3,3);
\draw[thick, dotted] (0,4)--(3,4);
\draw[thick, dotted] (0,5)--(4,5);
\draw[thick, dotted] (0,6)--(6,6);
\draw[thick, dotted] (0,7)--(6,7);
\draw[thick, dotted] (0,0)--(0,7);
\draw[thick, dotted] (1,0)--(1,7);
\draw[thick, dotted] (2,1)--(2,7);
\draw[thick, dotted] (3,4)--(3,7);
\draw[thick, dotted] (4,5)--(4,7);
\draw[thick, dotted] (5,5)--(5,7);
\draw[thick, dotted] (6,6)--(6,7);
\node at (0.5,0.5){$\scriptstyle2 $};
\node at (0.5,1.5){$\scriptstyle0 $};
\node at (0.5,2.5){$\scriptstyle1 $};
\node at (0.5,3.5){$\scriptstyle 2 $};
\node at (0.5,4.5){$\scriptstyle0 $};
\node at (0.5,5.5){$\scriptstyle1 $};
\node at (0.5,6.5){$\scriptstyle2 $};
\node at (1.5,0.5){$\scriptstyle0 $};
\node at (1.5,1.5){$\scriptstyle1 $};
\node at (1.5,2.5){$\scriptstyle2 $};
\node at (1.5,3.5){$\scriptstyle 0 $};
\node at (1.5,4.5){$\scriptstyle1$};
\node at (1.5,5.5){$\scriptstyle2 $};
\node at (1.5,6.5){$\scriptstyle0 $};
\node at (2.5,1.5){$\scriptstyle2 $};
\node at (2.5,2.5){$\scriptstyle0 $};
\node at (2.5,3.5){$\scriptstyle 1 $};
\node at (2.5,4.5){$\scriptstyle2$};
\node at (2.5,5.5){$\scriptstyle0 $};
\node at (2.5,6.5){$\scriptstyle1 $};
\node at (3.5,4.5){$\scriptstyle0$};
\node at (3.5,5.5){$\scriptstyle1 $};
\node at (3.5,6.5){$\scriptstyle2 $};
\node at (4.5,5.5){$\scriptstyle2 $};
\node at (4.5,6.5){$\scriptstyle0 $};
\node at (5.5,5.5){$\scriptstyle0 $};
\node at (5.5,6.5){$\scriptstyle1 $};
%
\end{tikzpicture}
}
\;\;
\leftrightarrow
\;\;
\hackcenter{
\begin{tikzpicture}[scale=0.4]
%
\draw[thick, fill=lightgray!50]  (0,0)--(2,0)--(2,1)--(3,1)--(3,3)--(1,3)--(1,1)--(0,1)--(0,0);
\draw[thick, fill=lightgray!50]  (0+3,4)--(1+3,4)--(1+3,5)--(3+3,5)--(3+3,7)--(2+3,7)--(2+3,6)--(0+3,6)--(0+3,4);
%
\draw[thick, dotted] (-1,0)--(7,0);
\draw[thick, dotted] (-1,1)--(7,1);
\draw[thick, dotted] (-1,2)--(7,2);
\draw[thick, dotted] (-1,3)--(7,3);
\draw[thick, dotted] (-1,4)--(7,4);
\draw[thick, dotted] (-1,5)--(7,5);
\draw[thick, dotted] (-1,6)--(7,6);
\draw[thick, dotted] (-1,7)--(7,7);
\draw[thick, dotted] (0,-1)--(0,8);
\draw[thick, dotted] (1,-1)--(1,8);
\draw[thick, dotted] (2,-1)--(2,8);
\draw[thick, dotted] (3,-1)--(3,8);
\draw[thick, dotted] (4,-1)--(4,8);
\draw[thick, dotted] (5,-1)--(5,8);
\draw[thick, dotted] (6,-1)--(6,8);
\node at (0.5,0.5){$\scriptstyle2 $};
\node at (0.5,1.5){$\scriptstyle0 $};
\node at (0.5,2.5){$\scriptstyle1 $};
\node at (0.5,3.5){$\scriptstyle 2 $};
\node at (0.5,4.5){$\scriptstyle0 $};
\node at (0.5,5.5){$\scriptstyle1 $};
\node at (0.5,6.5){$\scriptstyle2 $};
\node at (1.5,0.5){$\scriptstyle0 $};
\node at (1.5,1.5){$\scriptstyle1 $};
\node at (1.5,2.5){$\scriptstyle2 $};
\node at (1.5,3.5){$\scriptstyle 0 $};
\node at (1.5,4.5){$\scriptstyle1$};
\node at (1.5,5.5){$\scriptstyle2 $};
\node at (1.5,6.5){$\scriptstyle0 $};
\node at (2.5,0.5){$\scriptstyle1 $};
\node at (2.5,1.5){$\scriptstyle2 $};
\node at (2.5,2.5){$\scriptstyle0 $};
\node at (2.5,3.5){$\scriptstyle 1 $};
\node at (2.5,4.5){$\scriptstyle2$};
\node at (2.5,5.5){$\scriptstyle0 $};
\node at (2.5,6.5){$\scriptstyle1 $};
\node at (3.5,0.5){$\scriptstyle2 $};
\node at (3.5,1.5){$\scriptstyle0 $};
\node at (3.5,2.5){$\scriptstyle1 $};
\node at (3.5,3.5){$\scriptstyle 2 $};
\node at (3.5,4.5){$\scriptstyle0$};
\node at (3.5,5.5){$\scriptstyle1 $};
\node at (3.5,6.5){$\scriptstyle2 $};
\node at (4.5,0.5){$\scriptstyle0 $};
\node at (4.5,1.5){$\scriptstyle1 $};
\node at (4.5,2.5){$\scriptstyle2 $};
\node at (4.5,3.5){$\scriptstyle 0 $};
\node at (4.5,4.5){$\scriptstyle1$};
\node at (4.5,5.5){$\scriptstyle2 $};
\node at (4.5,6.5){$\scriptstyle0 $};
\node at (5.5,0.5){$\scriptstyle1 $};
\node at (5.5,1.5){$\scriptstyle2 $};
\node at (5.5,2.5){$\scriptstyle0 $};
\node at (5.5,3.5){$\scriptstyle 1 $};
\node at (5.5,4.5){$\scriptstyle2$};
\node at (5.5,5.5){$\scriptstyle0 $};
\node at (5.5,6.5){$\scriptstyle1 $};
\end{tikzpicture}
}
\;\;
\leftrightarrow
\;\;
\hackcenter{
\begin{tikzpicture}[scale=0.4]
%
\draw[thick, fill=lightgray!50]  (0,0)--(2,0)--(2,1)--(3,1)--(3,3)--(1,3)--(1,1)--(0,1)--(0,0);
\draw[thick, fill=lightgray!50]  (0,4)--(1,4)--(1,5)--(3,5)--(3,7)--(2,7)--(2,6)--(0,6)--(0,4);
\draw[thick]  (0,1)--(1,1)--(1,3)--(0,3)--(0,1);
\draw[thick]  (0,6)--(2,6)--(2,7)--(0,7)--(0,6);
%
\draw[thick, dotted] (1,1)--(2,1);
\draw[thick, dotted] (0,2)--(3,2);
\draw[thick, dotted] (0,3)--(3,3);
\draw[thick, dotted] (0,0)--(0,3);
\draw[thick, dotted] (1,0)--(1,1);
\draw[thick, dotted] (2,1)--(2,3);
\draw[thick, dotted] (0,5)--(1,5);
\draw[thick, dotted] (0,6)--(3,6);
\draw[thick, dotted] (0,7)--(3,7);
\draw[thick, dotted] (0,5)--(0,7);
\draw[thick, dotted] (1,5)--(1,7);
\draw[thick, dotted] (2,5)--(2,7);
\node at (0.5,0.5){$\scriptstyle2 $};
\node at (0.5,1.5){$\scriptstyle0 $};
\node at (0.5,2.5){$\scriptstyle1 $};
\node at (1.5,0.5){$\scriptstyle0 $};
\node at (1.5,1.5){$\scriptstyle1 $};
\node at (1.5,2.5){$\scriptstyle2 $};
\node at (2.5,1.5){$\scriptstyle2 $};
\node at (2.5,2.5){$\scriptstyle0 $};
\node at (0.5,4.5){$\scriptstyle0 $};
\node at (0.5,5.5){$\scriptstyle1 $};
\node at (0.5,6.5){$\scriptstyle2 $};
\node at (1.5,5.5){$\scriptstyle2 $};
\node at (1.5,6.5){$\scriptstyle0 $};
\node at (2.5,5.5){$\scriptstyle0 $};
\node at (2.5,6.5){$\scriptstyle1 $};
\end{tikzpicture}
}
\end{align*}
\caption{Skew Young diagram \(\lambda^\tau/\mu^\tau\) with charge \(2\); Skew shape \(\tau\); Level two skew Young diagram \(\blam/\bmu\) with multicharge \((2,1)\); all are associated as in \S\ref{otherformsec}}
\label{fig:Muth1}       
\end{figure}

\subsection{Ribbons and roots}
The following lemma establishes that the content of a ribbon depends only on the positions of the southwest/northeast end nodes, and not on the nodes in between.
\begin{lemma}\label{hookcont}
If \(\xi\) is a ribbon, then \(\cont(\xi) = \alpha(\res(\xi_{\ssww}), \dist(\xi_{\ssww},\xi_{\nnee})+1)\).
\end{lemma}
\begin{proof}
There is a \(\no/\ea\) path \((z^i)_{i=1}^{|\xi|}\) such that \(\xi = \{z^i\}_{i=1}^{|\xi|}\) by Lemma~\ref{hookispath}. We have then that \(z^1 = \xi_{\ssww}\) and \(z^{|\xi|} = \xi_{\nnee}\). Then by Corollary~\ref{sizehook} we have \(|\xi| = \dist(\xi_{\ssww}, \xi_{\nnee})+1\), and so
\begin{align*}
\cont(\xi) &= \sum_{i=1}^{|\xi|} \alpha_{\res(z^i)} = \sum_{i=1}^{|\xi|} \alpha_{\overline{ \diag(z^i)  }}
=
\sum_{i=1}^{|\xi|} \alpha_{\overline{ \diag(z^1) + i-1  }}
=
\sum_{i=0}^{|\xi|-1} \alpha_{\overline{ \diag(z^1)} + \overline{i}  }\\
&=
\sum_{i=0}^{\dist(\xi_{\ssww}, \xi_{\nnee})} \alpha_{ \res(\xi_{\ssww}) + \overline{i}}
=
 \alpha(\res(\xi_{\ssww}), \dist(\xi_{\ssww},\xi_{\nnee})+1),
\end{align*}
 applying Lemma~\ref{dumbdiag} for the third equality.
\end{proof}

\begin{corollary}\label{skewpos}
Let \(\beta \in Q_+\). Then \(\beta \in \Phi_+\) if and only if \(\beta = \cont(\xi)\) for some ribbon \(\xi\).
\end{corollary}
\begin{proof}
We have by Lemma~\ref{hookcont} that \(\cont(\xi) \in \Phi_+\) for all ribbons \(\xi\). Now let \(\beta = \alpha(t,h) \in \Phi_+\). Let \(z \in \N_t\), and define nodes \(z^i = (z_1, z_2 + i-1)\) for \(i \in [1,h]\). Then \((z^i)_{i=1}^h\) is a \(\no/\ea\) path, and so \(\xi = \{z^i\}_{i=1}^h\) is a ribbon by Lemma~\ref{hookispath}. By construction, \(\xi_{\ssww} = z\), \(\xi_{\nnee} = (z_1, z_2+h-1)\), so \(\dist(\xi_{\ssww}, \xi_{\nnee}) = h-1\). Then by Lemma~\ref{hookcont} we have
\(
\cont(\xi) = \alpha(\res(z), h) = \alpha(t,h) = \beta,
\)
as desired.
\end{proof}

\section{Cuspidality}
Recall that we have fixed a convex preorder \(\succeq\) on \(\Phi_+\).
\begin{definition}\label{cuspdef}
Let \(\tau \in \SSSS(\beta)\) for some \(\beta \in \Phi_+\). We say that \(\tau\) is {\em cuspidal} provided that, for every tableau \((\lambda_1, \lambda_2)\) of \(\tau\), the following conditions hold:
\begin{enumerate}
\item There exist \(\gamma_1, \ldots, \gamma_k \in \Phi_+\) such that \(\cont(\lambda_1) = \gamma_1 + \cdots + \gamma_k\) and \(\beta \succ \gamma_i\) for all \(i \in [1,k]\), and;
\item There exist \(\nu_1, \ldots, \nu_\ell \in \Phi_+\) such that \(\cont(\lambda_2) = \nu_1 + \cdots + \nu_\ell\) and \(\nu_i \succ \beta\) for all \(i \in [1,\ell]\).
\end{enumerate}
\end{definition}

\begin{definition}\label{cuspdef}
Let \( \tau \in \SSSS(m \beta)\) for some \(m \in \NN\), \(\beta \in \Phi_+\). We say that \(\tau\) is {\em semicuspidal} provided that, for every tableau \((\lambda_1, \lambda_2)\) of \(\tau\), the following conditions hold:
\begin{enumerate}
\item There exist \(\gamma_1, \ldots, \gamma_k \in \Phi_+\) such that \(\cont(\lambda_1) = \gamma_1 + \cdots + \gamma_k\) and \(\beta \succeq \gamma_i\) for all \(i \in [1,k]\), and;
\item There exist \(\nu_1, \ldots, \nu_\ell \in \Phi_+\) such that \(\cont(\lambda_2) = \nu_1 + \cdots + \nu_\ell\) and \(\nu_i \succeq \beta\) for all \(i \in [1,\ell]\).
\end{enumerate}
\end{definition}

It follows from definitions that cuspidality is invariant under \(e\)-similarity: 
\begin{lemma}
Let \(\tau, \nu\) be skew shapes, with \(\tau \sim_e \nu\). Then \(\tau\) is cuspidal (resp. semicuspidal) if and only if \(\nu\) is cuspidal (resp. semicuspidal). 
\end{lemma}

\begin{lemma}\label{shapehooktest}
If \(\tau\) is a cuspidal (resp. semicuspidal) skew shape, then for every \(\ssee\)-removable ribbon \(\xi \subsetneq \tau\) we have \(\cont(\xi) \succ \cont(\tau)\) (resp. \(\cont(\xi) \succeq \cont(\tau)\)).
\end{lemma}
\begin{proof}
 Let \(\tau\) be cuspidal. If \(\xi\) is an \(\ssee\)-removable ribbon in \(\tau\), then \((\tau \backslash \xi, \xi)\) is a skew decomposition for \(\tau\). 
 Then by the cuspidality property, we have \(\cont(\xi) = \gamma_1 + \cdots + \gamma_k\), for some positive roots \(\gamma_1, \ldots, \gamma_k \succ \cont(\tau)\). Then it follows by Corollary~\ref{skewpos} and Lemma~\ref{GenConvexLem}(iii) that \(\cont(\xi) \succ \cont(\tau)\). The proof in the semicuspidal case is similar.
\end{proof}

 \begin{lemma}\label{CuspForSkewHook}
Let \(\xi\) be a ribbon. Then \(\xi\) is cuspidal if and only if for every \(\ssee\)-removable ribbon \(\nu \subsetneq \xi\), we have \(\cont(\nu) \succ \cont(\xi)\).
 \end{lemma}
 \begin{proof}
The `only if' direction is provided by Lemma~\ref{shapehooktest}, so we focus on the `if' direction. 
 Assume that \(\xi\) is a ribbon with
 the property that \(\cont(\nu) \succ \cont(\xi)\) for every \(\ssee\)-removable ribbon \(\nu\) in \( \xi\). We have that \(\cont(\xi) \in \Phi_+\) by Corollary~\ref{skewpos}. We will show that \(\xi\) is cuspidal.

 Let \((\varepsilon, \mu)\) be any tableau of \(\xi\). Let \(\mu_1, \ldots, \mu_t\) be the connected components of \(\mu\). Since each \(\mu_i\) is connected and a subset of a ribbon, and \((\varepsilon, \mu)\) is a tableau, we have that each \(\mu_i\) is an \(\ssee\)-removable ribbon in \(\xi\). Thus, by assumption, \(\cont(\mu_i) \succ \cont(\xi)\) for all \(i\). But then \(\cont(\mu) = \sum_{i=1}^t \cont(\mu_i)\), and so \(\xi\) satisfies condition (ii) in Definition~\ref{cuspdef}.
 
 Let \(\varepsilon_1, \ldots, \varepsilon_u\) be the connected components of \(\varepsilon\). Again, each \(\varepsilon_i\) is a ribbon. Assume by way of contradiction there is some \(i\) such that \(\cont(\varepsilon_i) \succeq \cont(\xi)\). Since \((\varepsilon, \mu)\) is a tableau, we have that \((\varepsilon_i, \xi\backslash \varepsilon_i)\) is a tableau. By the previous paragraph, we have that \(\cont(\xi \backslash \varepsilon_i)\) can be written as a sum of positive roots \(\gamma_1 + \cdots + \gamma_s\), where \(\gamma_j \succ \cont(\xi)\) for all \(j \in [1,s]\). 
 If \(\cont(\varepsilon_i) \approx \cont(\xi)\), then we have that \(\cont(\varepsilon_i),\cont(\xi) \in \Phi_+^\im\) by Lemma~\ref{notapprox}. But this implies that \(\cont(\xi \backslash \varepsilon_i) = \cont(\xi) - \cont(\varepsilon_i) \in \Phi_+^\im\), and so \(\cont(\xi\backslash \varepsilon_i) \approx \cont(\xi)\). But since \(\cont(\xi \backslash \varepsilon_i) \in \Phi_+\), we have by Lemma~\ref{GenConvexLem}(iii) that \(\cont(\xi \backslash \varepsilon_i) =  \gamma_1 + \cdots + \gamma_s \succ \cont(\xi)\), a contradiction. Therefore \(\cont(\varepsilon_i) \succ \cont(\xi)\). But then we have by Lemma~\ref{GenConvexLem}(iii) that 
 \begin{align*}
 \cont(\xi) = \cont(\xi \backslash \varepsilon_i) + \cont(\varepsilon_i) = \gamma_1 + \cdots + \gamma_s + \cont(\varepsilon_i) \succ \cont(\xi),
 \end{align*}
 another contradiction. Therefore \(\cont(\varepsilon_i) \prec \cont(\xi)\) for all \(i \in [1,u]\). Since \(\cont(\varepsilon) = \sum_{i=1}^u \cont(\varepsilon_i)\), we have that \(\xi\) satisfies condition (i) in Definition~\ref{cuspdef}. Thus \(\xi\) is cuspidal.
 \end{proof}

 \subsection{Constructing cuspidal ribbons}\label{cuspcons}
For \(\beta \in \Psi\), we write \(\init(\beta) \subseteq \N\) for the set of nodes \(b \in \N\) such that \(\beta = \alpha(\res(b), \height(\beta))\). Note then that \(\init(\delta) = \N\), and if \(\beta \in \Phi_+^\re\) then \(\init(\beta) = \N_t\) for some \(t \in \Z_e\).

\begin{definition}\label{defxib}
Let \(\beta \in \Psi\) and \(b \in \init(\beta)\). Define a \(\no/\ea\) path \((z^i)_{i=1}^{\height(\beta)}\) by setting \(z^1 = b\), and 
\begin{align*}
z^i = \begin{cases}
\no z^{i-1} & \textup{if } \alpha(\res(b), i-1) \succ \beta;\\
\ea z^{i-1} & \textup{if } \beta \succ \alpha(\res(b), i-1),\\
\end{cases}
\end{align*}
for \(i = 2, \ldots, \height(\beta)\). Then define \(\zeta^{(\beta, b)} = \{z^i\}_{i=1}^{\height(\beta)}\).
\end{definition}

\begin{lemma}\label{zetaiscusp}
The set of nodes  
\(\zeta^{(\beta,b)}\) is a cuspidal ribbon of content \(\beta\), with \(\zeta^{(\beta,b)}_{\ssww} = b\).
\end{lemma}
\begin{proof}
First, note that the path \((z^i)_{i=1}^{\height(\beta)}\) is well-defined, which follows from Lemma~\ref{notapprox}(ii). We have that \(\zeta^{(\beta,b)}\) is a ribbon by Lemma~\ref{hookispath}, and \(\cont(\zeta^{(\beta,b)}) = \beta\) by Lemma~\ref{hookcont}.

It remains to show that \( \zeta = \zeta^{(\beta,b)}\) is cuspidal. Let \(\nu \subsetneq \zeta\) be an \(\ssee\)-removable ribbon in \(\zeta\). By Lemma~\ref{CuspForSkewHook}, it will be enough to show that \(\cont(\nu) \succ \beta\). By Lemma~\ref{xiuv}, we have 
\(\nu = \xi^\zeta_{z^i,z^j}\) for some \(1 \leq i \leq j \leq h\), where \(\so z^i, \ea z^j \notin \zeta\).

First assume \(i=1, j<h\). Since \(\ea z^j  \notin \zeta\), it follows from the construction of \(\zeta\) that \(z^{j+1} = \no z^j\), and thus \(\alpha(\res(b),j) \succ \beta\). But then by Lemma~\ref{hookcont} we have 
\begin{align*}
\cont(\nu) = \cont(\xi^\zeta_{z^1, z^j}) =  \alpha(\res(b),j) \succ \beta,
\end{align*}
as desired.

Now assume \(1<i\). Since \(\so z^i \notin \zeta\), we have \(z^i = \ea z^{i-1} \) and thus \(\beta \succ \alpha(\res(b), i-1)\). If \(j<h\), then since \(\ea z^j  \notin \zeta\), it follows from the construction of \(\zeta\) that \(z^{j+1} = \no z^j\), and thus \(\alpha(\res(b),j) \succ \beta\). On the other hand, if \(j=h\), then \(\alpha(\res(b),j) = \beta\), so in any case we have \(\alpha(\res(b),j) \succeq \beta\).
Therefore, applying Lemma~\ref{hookcont}, we have
\begin{align*}
\alpha(\res(b), j) &= 
\cont(\xi^\zeta_{z^1, z^j}) 
=\cont(\xi^\zeta_{z^1, z^{i-1}}\sqcup \xi^\zeta_{z^i, z^j})\\
&=
\cont(\xi^\zeta_{z^1, z^{i-1}}) + \cont(\xi^\zeta_{z^i, z^j})
= \alpha(\res(b), i-1) + \cont(\nu).
\end{align*}
If \(\beta \succeq \cont(\nu)\), we would have by Lemma~\ref{GenConvexLem}(iii) that \(\beta \succ \alpha(\res(b), j)\), a contradiction. Thus \(\cont(\nu) \succ \beta\), as desired.
\end{proof}

\begin{lemma}\label{onlycuspribbon}
Let \(\nu\) be a ribbon of content \(\beta \in \Psi\). Then \(\nu\) is cuspidal if and only if \(\nu = \zeta^{(\beta, \nu_{\ssww})}\).
\end{lemma}
\begin{proof}
The `if' direction is proved in Lemma~\ref{zetaiscusp}. Now we prove the `only if' direction. Let \(b =\nu_{\ssww}\), and \(h = \height(\beta)\). By Lemma~\ref{hookcont} we must have \(\beta = \alpha(\res(b),h)\). Consider the skew shape \(\zeta = \zeta^{(\beta,b)}\), and let \((z^i)_{i=1}^h\) be the \(\no/\ea\) path constructed in Definition~\ref{defxib} so that \(\zeta = \{z^i\}_{i=1}^h\). By Lemma~\ref{hookispath}, there is a \(\no/\ea\) path \((w^i)_{i=1}^h\), with \(w^1 = b\) and \(\nu = \{w^i\}_{i=1}^h\). Assume by way of contradiction that \(\nu \neq \zeta\). Then there exists \(2 \leq t \leq h\) such that \(w^i = z^i\) for \(i<t\), and \(w^t \neq z^t\). Note that by Lemma~\ref{notapprox}, we have \(\beta \not \approx \alpha(\res(b),t-1)\).

First assume \(\beta \succ \alpha(\res(b),t-1)\). Then \(z^t = \ea z^{t-1}\).
We have \(w^t \in \{\no w^{t-1}, \ea w^{t-1}\} = \{\no z^{t-1}, \ea z^{t-1}\}\), but \(w^t \neq z^t\), so \(w^t = \no z^{t-1}\). Then \(\ea w^{t-1} \notin \nu\), and \(\ea w^h \notin \nu\), so \(\xi^\nu_{w^1, w^{t-1}}\) is \(\ssee\)-removable in \(\nu\) by Lemma~\ref{xiuv}. By cuspidality of \(\nu\) and Lemma~\ref{CuspForSkewHook} we have that
\begin{align*}
\alpha(\res(\nu_1), t-1) \succ \cont(\xi^\nu_{\nu_t, \nu_h}) \succ \beta,
\end{align*}
a contradiction.

Now assume \(\alpha(\res(b),t-1) \succ \beta\). Then \(z^t = \no z^{t-1}\), and we have \(w^t \in \{\no w^{t-1}, \ea w^{t-1}\} = \{\no z^{t-1}, \ea z^{t-1}\}\), but \(w^t \neq z^t\), so \(w^t = \ea z^{t-1}\). Then \(\so w^t \notin \nu\), and \(\ea w^h \notin \nu\), so \(\xi^\nu_{w^t, w^h}\) is \(\ssee\)-removable in \(\nu\) by Lemma~\ref{xiuv}. By cuspidality of \(\nu\) and Lemma~\ref{CuspForSkewHook} we have that
\begin{align*}
\alpha(\res(w^t), h-t + 1) = \cont(\xi^\nu_{w^t, w^h}) \succ \beta.
\end{align*}
But then by convexity and (\ref{splithookcont}) we have
\begin{align*}
\beta = \alpha(\res(b), h) =  \alpha(\res(b),t-1) + \alpha(\res(w^t), h-t + 1) \succ \beta,
\end{align*}
a contradiction. Thus, in any case we derive a contradiction, so we must have \(\nu = \zeta\) as desired.
\end{proof}

\subsection{Distinguished cuspidal ribbons}
For each \(\beta \in \Phi_+^\re\), we make a distinguished choice of \(b^\beta \in \init(\beta)\). For each \(t \in \Z_e\), we make a distinguished choice of \(b^t \in \N_t\). Then we define the distinguished cuspidal ribbons, for all \(\beta \in \Phi_+^\re\) and \(t \in \N_t\):
\begin{align*}
\zeta^\beta := \zeta^{(\beta, b^{\beta})}, \qquad \zeta^t := \zeta^{(\delta, b^t)}
\end{align*}

By construction, the path associated with \(\zeta^{(\beta, b)}\) in Definition~\ref{defxib} relies only on \(\res(b)\) and \(\beta\), and not on the specific location of \(b\). Thus by Lemmas~\ref{zetaiscusp} and~\ref{onlycuspribbon} we have the following result:

\begin{proposition}\label{allcusprib}
For all \(\beta \in \Phi_+^\re\), \(t \in \ZZ_e\), we have
\begin{align*}
[\zeta^{\beta}] = \{\zeta^{(\beta,b)} \mid b \in \init(\beta)\},
\qquad
\textup{and}
\qquad
[\zeta^{t}] = \{\zeta^{(\delta,b)} \mid b \in \N_t\}.
\end{align*}
The set \(\{\zeta^{\beta} \mid \beta \in \Phi_+^\re\} \cup \{\zeta^t \mid t \in \Z_e\}\) represents a complete and irredundant set of cuspidal ribbons with content in \(\Psi\), up to \(e\)-similarity.
\end{proposition}

We will show in Theorem~\ref{allcusp} that this in fact describes all cuspidal {\em skew shapes} up to \(e\)-similarity.

\begin{example}\label{e2shapes}
Take the case \(e=2\) and the convex preorder as defined in Example~\ref{e2order}. In this case we have a highly regular stairstep pattern for all real cuspidal ribbons, and a pair of two-node dominoes for the imaginary cuspidal ribbons, see Figure~\ref{fig:e2cusps}.
\begin{figure}[h]
\begin{align*}
{}
\hackcenter{
\begin{tikzpicture}[scale=0.4]
%
\fill[fill=lightgray!50] (0,0)--(1,0)--(1,1)--(2,1)--(2,2)--(3,2)--(3,3)--(4,3)--(4,4)--(5,4)--(5,5)--(6,5)--(6,6)--(4,6)--(4,5)--(3,5)--(3,4)--(2,4)--(2,3)--(1,3)--(1,2)--(0,2)--(0,0);
\draw[thick] (4,3.5)--(4,4)--(5,4)--(5,5)--(6,5)--(6,6)--(4,6)--(4,5)--(3,5)--(3,4.5);
\draw[thick]  (1.5,3)--(1,3)--(1,2)--(0,2)--(0,0)--(1,0)--(1,1)--(2,1)--(2,2)--(2.5,2);
%
\draw[thick, dotted] (1,0)--(1,2);
\draw[thick, dotted] (2,1)--(2,3);
\draw[thick, dotted] (0,1)--(1,1);
\draw[thick, dotted] (1,2)--(2,2);
\draw[thick, dotted] (3,4)--(4,4)--(4,5)--(5,5)--(5,6);
\node at (0.5,0.5){$\scriptstyle1 $};
\node at (0.5,1.5){$\scriptstyle0 $};
\node at (1.5,1.5){$\scriptstyle1 $};
\node at (1.5,2.5){$\scriptstyle0 $};
\node at (3.5,4.5){$\scriptstyle0 $};
\node at (4.5,4.5){$\scriptstyle1 $};
\node at (4.5,5.5){$\scriptstyle0 $};
\node at (5.5,5.5){$\scriptstyle1 $};
\node at (2.75,3.5){$\scriptstyle \iddots $};
%
\end{tikzpicture}
}
\hackcenter{
\begin{tikzpicture}[scale=0.4]
%
\fill[fill=lightgray!50] (0,0)--(2,0)--(2,1)--(3,1)--(3,2)--(4,2)--(4,3)--(5,3)--(5,4)--(6,4)--(6,6)--(5,6)--(5,5)--(4,5)--(4,4)--(3,4)--(3,3)--(2,3)--(2,2)--(1,2)--(1,1)--(0,1)--(0,0);
\draw[thick] (4.5,3)--(5,3)--(5,4)--(6,4)--(6,6)--(5,6)--(5,5)--(4,5)--(4,4)--(3.5,4);
\draw[thick]  (2,2.5)--(2,2)--(1,2)--(1,1)--(0,1)--(0,0)--(2,0)--(2,1)--(3,1)--(3,1.5);
%
\draw[thick, dotted] (1,0)--(1,1)--(2,1)--(2,2)--(3,2);
\draw[thick, dotted] (4,3)--(4,4)--(5,4)--(5,5)--(6,5);
\node at (0.5,0.5){$\scriptstyle0 $};
\node at (1.5,0.5){$\scriptstyle1 $};
\node at (1.5,1.5){$\scriptstyle0 $};
\node at (2.5,1.5){$\scriptstyle1 $};
\node at (4.5,3.5){$\scriptstyle1 $};
\node at (4.5,4.5){$\scriptstyle0 $};
\node at (5.5,4.5){$\scriptstyle1 $};
\node at (5.5,5.5){$\scriptstyle0 $};
\node at (3.25,3){$\scriptstyle \iddots $};
%
\end{tikzpicture}
}
\qquad
\hackcenter{
\begin{tikzpicture}[scale=0.4]
%
\fill[fill=lightgray!50] (0,0)--(2,0)--(2,1)--(0,1)--(0,0);
\draw[thick]  (0,0)--(2,0)--(2,1)--(0,1)--(0,0);
%
\draw[thick, dotted] (1,0)--(1,1);
\node at (0.5,0.5){$\scriptstyle0 $};
\node at (1.5,0.5){$\scriptstyle1 $};
%
%
\end{tikzpicture}
}
\qquad
\hackcenter{
\begin{tikzpicture}[scale=0.4]
%
\fill[fill=lightgray!50] (0,0)--(1,0)--(1,2)--(0,2)--(0,0);
\draw[thick]  (0,0)--(1,0)--(1,2)--(0,2)--(0,0);
%
\draw[thick, dotted] (0,1)--(1,1);
\node at (0.5,0.5){$\scriptstyle1 $};
\node at (0.5,1.5){$\scriptstyle0 $};
%
%
\end{tikzpicture}
}
\end{align*}
\caption{Cuspidal ribbons, \(e=2\) case: \( \zeta^{n\delta + \alpha_1}; \zeta^{n \delta + \alpha_0}; \zeta^{ \overline{0}}; \zeta^{\overline{1}}\)}
\label{fig:e2cusps}       
\end{figure}
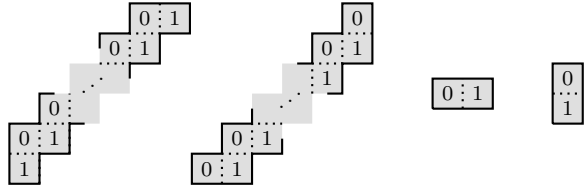
\end{example}

We refer the reader to Example~\ref{BigEx} for a preorder in the \(e=3\) case with more irregular cuspidal ribbon shapes.

\subsection{Minimal ribbons}
Let \(\tau\) be a skew shape, and \(\xi\) be an \(\ssee\)-removable ribbon in \(\tau\). We say \(\xi\) is a {\em minimal \(\ssee\)-removable ribbon} (resp. {\em maximal}) provided that for every \(\ssee\)-removable ribbon \(\nu\) in \(\tau\) we have \(\cont(\nu) \succeq \cont(\xi)\) (resp. \(\cont(\nu) \preceq \cont(\xi)\)). We extend these notions as well to \(\nnww\)-removable ribbons in the obvious way.

\begin{lemma}\label{minremiscusp}
Let \(\xi\) be a minimal \(\ssee\)-removable ribbon in a skew shape \(\tau\), with \(\cont(\xi) \in \Psi\). Then \(\xi\) is cuspidal.
\end{lemma}
\begin{proof}
Take any \(\ssee\)-removable ribbon \(\nu \subsetneq \xi\) in \(\xi\). Since \(\xi\) is \(\ssee\)-removable in \(\tau\), we have that \(\nu\) is \(\ssee\)-removable in \(\tau\). By the minimality of \(\xi\), we have \(\cont(\nu) \succeq \cont(\xi)\), and by Lemma~\ref{notapprox}(ii) we have \(\cont(\nu) \not \approx \cont(\xi)\), so \(\cont(\nu) \succ \cont(\xi)\). Thus \(\xi\) is cuspidal by Lemma~\ref{CuspForSkewHook}.
\end{proof}

\begin{lemma}\label{mdeltaribbon}
If \(\nu\) is a ribbon of content \(m\delta\), then \(\nu\) has an \(\ssee\)-removable ribbon of content \(\delta\).
\end{lemma}
\begin{proof}
We go by induction on \(m\). The base case \(m=1\) is trivial, so assume \(m>1\) and make the induction assumption on \(m' < m\). By Lemma~\ref{hookispath} there exists a \(\no/\ea\) path \((z^i)_{i=1}^{me}\) such that \(\nu = \{z^i\}_{i=1}^{me}\). If \(z^{e+1} = \no z^e\), it follows by Lemma~\ref{xiuv} that \(\xi^\nu_{z^1, z^e}\) is \(\ssee\)-removable in \(\nu\), and we have \(\cont(\xi^\nu_{z^1,z^e}) = \delta\) by Lemma~\ref{hookcont}. On the other hand, if \(z^{e+1} = \ea z^e\), it follows by Lemma~\ref{xiuv} that \(\xi^\nu_{z^{e+1}, z^{me}}\) is \(\ssee\)-removable in \(\nu\), and we have \(\cont(\xi^\nu_{z^{e+1}, z^{me}}) = (m-1)\delta\) by Lemma~\ref{hookcont}. Then by induction there exists an \(\ssee\)-removable ribbon \(\omega\) in \(\xi^\nu_{z^{e+1}, z^{me}}\) of content \(\delta\). As \(\omega\) is \(\ssee\)-removable in \(\xi^\nu_{z^{e+1}, z^{me}}\), it is \(\ssee\)-removable in \(\nu\). This completes the induction step, and the proof.
\end{proof}

\begin{lemma}\label{minremexist}
If \(\tau\) is a skew shape, then \(\tau\) has a minimal \(\ssee\)-removable ribbon \(\xi\) such that \(\cont(\xi) \in \Psi\).
\end{lemma}
\begin{proof}
Let \(\nu\) be a minimal \(\ssee\)-removable ribbon in \(\tau\). Then \(\cont(\nu) \in \Phi_+\) by Lemma~\ref{hookcont}. If \(\cont(\nu) \in \Phi_+^\re\), we are done by (\ref{PsiDef}), so assume otherwise. Then 
by (\ref{allimag}), we have \(\cont(\nu) = m\delta\) for some \(m\in \NN\).
Then by Lemma~\ref{mdeltaribbon}, \(\nu\) has an \(\ssee\)-removable ribbon \(\xi\) of content \(\delta \in \Psi\). But then \(\xi\) is \(\ssee\)-removable in \(\tau\), and \(\delta \approx m\delta\), so \(\xi\) is minimal in \(\tau\) as well.
\end{proof}

The next key lemma establishes that consecutively \(\ssee\)-removed minimal ribbons are non-decreasing in the convex preorder. 


 \begin{lemma}\label{remsinorder}
Let \(\tau\) be a skew shape. Let \(\xi_1\) be a minimal \(\ssee\)-removable ribbon in \(\tau\), and let \(\xi_2\) be a minimal \(\ssee\)-removable ribbon in \(\tau'=\tau\backslash \xi_1\). Then \(\cont(\xi_2) \succeq \cont(\xi_1)\).
 \end{lemma}
 \begin{proof}
 
 First, note that if \(\xi_1 \sqcup \xi_2\) is disconnected, then \(\xi_2\) is an \(\ssee\)-removable ribbon in \(\tau\), so by minimality of \(\xi_1\), we would have \(\cont(\xi_2) \succeq \cont(\xi_1)\), as desired. Thus we may assume \(\xi_1 \sqcup \xi_2\) is a connected skew shape, and therefore \(\xi_1, \xi_2\) belong to the same connected component of \(\tau\). 
  We have \(\xi_1 = \xi^\tau_{u,v}\) for some \((u,v) \in \textup{Rem}_\tau\), and \(\xi_2 = \xi^{\tau'}_{z,w}\) for some \((z,w) \in \textup{Rem}_{\tau'}\), where \(u,v,z,w\) belong to the same connected component of \(\tau\).  
We consider the six possible cases of arrangements of the nodes \(u,v,z,w\) separately.
 
 {\em Case 1: Assume \(u \nearrow v \nearrow z \nearrow w\).} First, note that since \(\xi_1 \sqcup \xi_2\) is connected, we must have \(z \in \{\no v, \ea v\}\). But, since \((u,v) \in \textup{Rem}_\tau\), we have \(\ea v \notin \tau\), so \(z=\no v\). 
By Lemma~\ref{remhookswap}(i) we have that \((u,w) \in \textup{Rem}_\tau\) and \(\xi^\tau_{u,w}\) is an \(\ssee\)-removable ribbon in \(\tau\) by Lemma~\ref{xiuv}.
 
We have
\begin{align*}
\dist(u,w) = 
\dist(u,v) + \dist(v,z) + \dist(z,w) = 
 \dist(u,v) + \dist(z,w) + 1
\end{align*}
and
\begin{align*}
\res(z) = \res(\ea v) = \res(v) + \overline{1} = \res(u) + \overline{\dist(u,v) + 1}.
\end{align*}
It follows by (\ref{splithookcont}) and Lemma~\ref{hookcont} that 
\begin{align*}
\cont(\xi^\tau_{u,w}) &= \alpha(\res(u), \dist(u,w)+1) = \alpha(\res(u),  \dist(u,v) + \dist(z,w) + 2)\\
&= \alpha(\res(u), \dist(u,v)+1) + \alpha(\res(u) + \overline{\dist(u,v)+1}, \dist(z,w)+1)\\
&= \alpha(\res(u),  \dist(u,v)+1) + \alpha(\res(z), \dist(z,w)+1)\\
&= \cont(\xi_1) + \cont(\xi_2).
\end{align*}

Assume by way of contradiction that \(\cont(\xi_1) \succ \cont(\xi_2)\). Then we have \(\cont(\xi_1) \succ \cont(\xi^\tau_{u,w}) \succ \cont(\xi_2)\) by Lemma~\ref{GenConvexLem}. But, since \(\xi^\tau_{u,w}\) is \(\ssee\)-removable in \(\tau\), this contradicts \(\xi_1\) being a minimal \(\ssee\)-removable ribbon in \(\tau\). Therefore \(\cont(\xi_2) \succeq \cont(\xi_1)\), as desired.

 {\em Case 2: Assume \(w \nearrow u\).} Note that since \(\xi_1 \sqcup \xi_2\) is connected, we must have \(w \in \{\so u, \we u\}\). But, since \((u,v) \in \textup{Rem}_\tau\),
we have \(\so u \notin \tau\), so \(w=\we u\). We have by Lemma~\ref{remhookswap}(ii) that \((z,v) \in \textup{Rem}_\tau\). Therefore \(\xi^\tau_{z,v}\) is an \(\ssee\)-removable ribbon in \(\tau\) by Lemma~\ref{xiuv}. 
 It can be shown along the lines of Case 1 that \(\cont(\xi^\tau_{z,v}) = \cont(\xi_1) + \cont(\xi_2)\). Thus if \(\cont(\xi_1) \succ \cont(\xi_2)\), we derive a contradiction along the same lines as Case 1 as well, so we have \(\cont(\xi_2) \succeq \cont(\xi_1)\), as desired.
 
 {\em Case 3: Assume \(u \nearrow z \nearrow v \nearrow w\)}. It follows from Lemma~\ref{remhookswap}(iii) that \((\ssee z, w) \in \textup{Rem}_\tau\). 
Therefore \(\xi^\tau_{\ssee z,w}\) is an \(\ssee\)-removable ribbon in \(\tau\) by Lemma~\ref{xiuv}.
 We have
\(
 \dist(\ssee z ,w) = \dist(z,w),
\)
 and
\(
 \res(\ssee z) =\res(z),
\)
 so by Lemma~\ref{hookcont} it follows that 
 \begin{align*}
 \cont(\xi^\tau_{\ssee z,w}) &= \alpha(\res(\ssee z), \dist(\ssee z,w)+1)\\
 &= \alpha(\res(z),\dist(z,w)+1) = \cont(\xi_2).
 \end{align*}
As \(\xi_1\) is a minimal \(\ssee\)-removable ribbon in \(\tau\), we have then that
\(
\cont(\xi_2) = \cont(\xi^\tau_{\ssee z,w}) \succeq \cont(\xi_1),
\)
as desired.

{\em Case 4: Assume \(z \nearrow u \nearrow w \nearrow v\)}. 
It follows from Lemma~\ref{remhookswap}(iv) that \((z, \ssee w) \in \textup{Rem}_\tau\).
Therefore \(\xi^\tau_{z, \ssee w}\) is an \(\ssee\)-removable ribbon in \(\tau\) by Lemma~\ref{xiuv}.
As in Case 3, it is straightforward to show that \(\cont(\xi^\tau_{z, \ssee w}) = \cont(\xi_2)\), and thus by the minimality of \(\xi_1\) we have
\(
\cont(\xi_2) = \cont(\xi^\tau_{z, \ssee w}) \succeq \cont(\xi_1),
\)
as desired.

{\em Case 5: Assume \(u \nearrow z \nearrow w \nearrow v\)}. It follows from Lemma~\ref{remhookswap}(v) that \((\ssee z, \ssee w) \in \textup{Rem}_\tau\tau\). Then \(\xi^\tau_{\ssee z, \ssee w}\) is an \(\ssee\)-removable ribbon in \(\tau\). 

We have
\(
 \dist(\ssee z,\ssee w) = \dist (z,w),
\)
 and
\(
 \res(\ssee z) = \res(z)
\), so it follows by Lemma~\ref{hookcont} that 
 \begin{align*}
 \cont(\xi^\tau_{\ssee z,\ssee w}) &= \alpha(\res(\ssee z), \dist(\ssee z,\ssee w)+1)\\
 &= \alpha(\res(z),\dist(z,w)+1) = \cont(\xi_2).
 \end{align*}
As \(\xi_1\) is a minimal \(\ssee\)-removable ribbon in \(\tau\), it follows that
\(
\cont(\xi_2) = \cont(\xi^\tau_{\ssee(z),w}) \succeq \cont(\xi_1),
\)
as desired.

{\em Case 6: Assume \(z \nearrow u \nearrow v \nearrow w\)}.  It follows from Lemma~\ref{remhookswap}(vi) that \((z, w) \in \textup{Rem}_\tau\tau\). Thus \(\xi^\tau_{z,w}\) is an \(\ssee\)-removable ribbon in \(\tau\). 
But then by Lemma~\ref{xiuv} and the minimality of \(\xi_1\) we have
\(
\cont(\xi_2) =\cont(\xi^\tau_{z,w}) \succeq \cont(\xi_1),
\)
as desired, completing the proof. 
\end{proof}

\section{Kostant tilings}\label{Kostsec}

We say a tiling \(\Lambda\) of a skew shape \(\tau\) is a {\em ribbon} (resp. {\em cuspidal}) (resp. {\em semicuspidal}) tiling if \(\lambda\) is a ribbon (resp. cuspidal skew shape) (resp. semicuspidal skew shape) for all \(\lambda \in \Lambda\).
We are particularly interested in such tilings whose constituent skew shapes are arranged in a way that respects the convex preorder \(\succeq\), as follows.

\begin{definition}
Assume that \(\Lambda\) is a tiling of a skew shape \(\tau\) such that \(\cont(\lambda) \in \Phi_+'\) for all \(\lambda\in \Lambda\), and there exists a \(\Lambda\)-tableau \(\ttt\) such that \((\cont(\ttt(i)))_{i=1}^{|\Lambda|}\) is a Kostant sequence. Then we say \(\Lambda\) is a {\em Kostant} tiling, and that \((\Lambda,\ttt)\) is a {\em Kostant} tableau for \(\tau\).
\end{definition}

If \((\Lambda,\ttt)\) is a Kostant tableau for \(\tau\), then we have an associated Kostant sequence \(\bbeta^{\ttt}:= (\cont(\ttt(i)))_{i=1}^{|\Lambda|}\), and associated Kostant partition \(\bkap^\Lambda = \bkap^{\bbeta^{\ttt}}\). The Kostant partition \(\bkap^\Lambda\) depends only on \(\Lambda\) and not on the choice of \(\Lambda\)-tableau \(\ttt\).

\begin{definition}
Let \(\tau\) be a skew shape, and let \((\Lambda,\ttt)\) be a tableau for \(\tau\). 
We say \((\Lambda,\ttt)\) is a {\em minimal \(\ssee\)-removable ribbon tableau} of \(\tau\) provided that \(\cont(\lambda_j) \in \Psi\) and \(\lambda_j\) is a minimal \(\ssee\)-removable ribbon in \(\bigsqcup_{i=1}^j \lambda_i\) for all \(j \in [1, |\Lambda|]\). We say \(\Lambda\) is a {\em minimal \(\ssee\)-removable ribbon tiling} of \(\tau\) if there exists a \(\Lambda\)-tableau \(\ttt\) such that \((\Lambda, \ttt)\) is a minimal \(\ssee\)-removable ribbon tableau.
\end{definition}

We have the following dual notion as well:

\begin{definition}
Let \(\tau\) be a skew shape, and let \((\Lambda,\ttt)\) be a tableau for \(\tau\). 
We say \((\Lambda,\ttt)\) is a {\em maximal \(\nnww\)-removable ribbon tableau} of \(\tau\) provided that \(\cont(\lambda_j) \in \Psi\) and \(\lambda_j\) is a maximal \(\nnww\)-removable ribbon in \(\bigsqcup_{i=j}^{|\Lambda|} \lambda_i\) for all \(j \in [1, |\Lambda|]\). We say \(\Lambda\) is a {\em maximal \(\nnww\)-removable ribbon tiling} of \(\tau\) if there exists a \(\Lambda\)-tableau \(\ttt\) such that \((\Lambda, \ttt)\) is a maximal \(\nnww\)-removable ribbon tableau.
\end{definition}

\begin{example}\label{e2Kosex}
Take \(e=2\) and recall the convex preorder \(\succeq\) from Example~\ref{e2order} and the cuspidal ribbons from Example~\ref{e2shapes}. We show two tilings for a skew shape \(\tau\) in Figure~\ref{fig:Kosex1}. The tiling \(\Lambda_1\) on the left is Kostant for this preorder, with associated Kostant partition
\begin{align*}
\bkap^{\Lambda_1} = ( 3\delta + \alpha_1 \mid  \delta^3 \mid \delta + \alpha_0).
\end{align*}
The tiling \(\Lambda_2\) on the right in Figure~\ref{fig:Kosex1} is a cuspidal Kostant tiling for \(\tau\), which is (not coincidentally, per Theorem~\ref{mainthmcusp}) a minimal \(\ssee\)-removable ribbon tiling, and maximal \(\nnww\)-removable ribbon tiling as well. The associated Kostant partition is
\begin{align*}
\bkap^{\Lambda_2} = (\alpha_1^2 \mid \delta + \alpha_1 \mid  \delta^3 \mid \delta + \alpha_0 \mid \alpha_0^2).
\end{align*}
\end{example}

\begin{figure}[h]
\begin{align*}
{}
\hackcenter{
\begin{tikzpicture}[scale=0.4]
%
\draw[thick, fill=brown!35]  (0,-1)--(1,-1)--(1,0)--(4,0)--(4,1)--(6,1)--(6,4)--(5,4)--(5,3)--(1,3)--(1,2)--(1,1)--(0,1)--(0,-1);
\draw[thick, fill=pink]  (1,1)--(4,1)--(4,2)--(5,2)--(5,3)--(1,3)--(1,1);
\draw[thick, fill=cyan!25]  (4,1)--(6,1)--(6,4)--(5,4)--(5,2)--(4,2)--(4,1);
\draw[thick,  fill=blue!40!green!45]  (1,0)--(4,0)--(4,1)--(1,1)--(1,0);
%
\draw[thick, dotted] (1,0)--(1,1);
\draw[thick, dotted] (2,0)--(2,3);
\draw[thick, dotted] (3,0)--(3,3);
\draw[thick, dotted] (4,0)--(4,3);
\draw[thick, dotted] (5,1)--(5,3);
\draw[thick, dotted] (0,0)--(4,0);
\draw[thick, dotted] (0,1)--(6,1);
\draw[thick, dotted] (1,2)--(6,2);
\draw[thick, dotted] (2,3)--(6,3);
\draw[thick]   (0,-1)--(1,-1)--(1,0)--(4,0)--(4,1)--(6,1)--(6,4)--(5,4)--(5,3)--(1,3)--(1,2)--(1,1)--(0,1)--(0,-1);
\node at (0.5,-0.5){$\scriptstyle0 $};
\node at (0.5,0.5){$\scriptstyle1 $};
\node at (1.5,0.5){$\scriptstyle0 $};
\node at (2.5,0.5){$\scriptstyle1 $};
\node at (3.5,0.5){$\scriptstyle0 $};
\node at (1.5,1.5){$\scriptstyle1 $};
\node at (2.5,1.5){$\scriptstyle0 $};
\node at (3.5,1.5){$\scriptstyle1 $};
\node at (4.5,1.5){$\scriptstyle0 $};
\node at (5.5,1.5){$\scriptstyle1 $};
\node at (1.5,2.5){$\scriptstyle0 $};
\node at (2.5,2.5){$\scriptstyle1 $};
\node at (3.5,2.5){$\scriptstyle0 $};
\node at (4.5,2.5){$\scriptstyle1 $};
\node at (5.5,2.5){$\scriptstyle0 $};
\node at (5.5,3.5){$\scriptstyle1 $};
\end{tikzpicture}
}
\qquad
\qquad
\hackcenter{
\begin{tikzpicture}[scale=0.4]
%
\draw[thick,  fill=violet!35]  (0,-1)--(1,-1)--(1,0)--(4,0)--(4,1)--(6,1)--(6,4)--(5,4)--(5,3)--(1,3)--(1,2)--(1,1)--(0,1)--(0,-1);
\draw[thick, fill=yellow]  (0,-1)--(1,-1)--(1,0)--(0,0)--(0,-1);
\draw[thick, fill=orange!60]  (0,0)--(1,0)--(1,1)--(0,1)--(0,0);
\draw[thick, fill=yellow]  (3,0)--(4,0)--(4,1)--(3,1)--(3,0);
\draw[thick, fill=lightgray!50]  (1,0)--(3,0)--(3,1)--(4,1)--(4,2)--(5,2)--(5,3)--(3,3)--(3,2)--(2,2)--(2,1)--(1,1)--(1,0);
\draw[thick]  (2,1)--(3,1);
\draw[thick]  (3,2)--(4,2);
\draw[thick, fill=lime] (4,1)--(6,1)--(6,3)--(5,3)--(5,2)--(4,2)--(4,1);
\draw[thick, fill=orange!60]  (5,3)--(6,3)--(6,4)--(5,4)--(5,3);
%
\draw[thick, dotted] (1,0)--(1,1);
\draw[thick, dotted] (2,0)--(2,3);
\draw[thick, dotted] (3,0)--(3,3);
\draw[thick, dotted] (4,0)--(4,3);
\draw[thick, dotted] (5,1)--(5,3);
\draw[thick, dotted] (0,0)--(4,0);
\draw[thick, dotted] (0,1)--(6,1);
\draw[thick, dotted] (1,2)--(6,2);
\draw[thick, dotted] (2,3)--(6,3);
\draw[thick]   (0,-1)--(1,-1)--(1,0)--(4,0)--(4,1)--(6,1)--(6,4)--(5,4)--(5,3)--(1,3)--(1,2)--(1,1)--(0,1)--(0,-1);
\node at (0.5,-0.5){$\scriptstyle0 $};
\node at (0.5,0.5){$\scriptstyle1 $};
\node at (1.5,0.5){$\scriptstyle0 $};
\node at (2.5,0.5){$\scriptstyle1 $};
\node at (3.5,0.5){$\scriptstyle0 $};
\node at (1.5,1.5){$\scriptstyle1 $};
\node at (2.5,1.5){$\scriptstyle0 $};
\node at (3.5,1.5){$\scriptstyle1 $};
\node at (4.5,1.5){$\scriptstyle0 $};
\node at (5.5,1.5){$\scriptstyle1 $};
\node at (1.5,2.5){$\scriptstyle0 $};
\node at (2.5,2.5){$\scriptstyle1 $};
\node at (3.5,2.5){$\scriptstyle0 $};
\node at (4.5,2.5){$\scriptstyle1 $};
\node at (5.5,2.5){$\scriptstyle0 $};
\node at (5.5,3.5){$\scriptstyle1 $};
\end{tikzpicture}
}
\end{align*}
\caption{Kostant tiling for \(\tau\); cuspidal Kostant tiling for \(\tau\)}
\label{fig:Kosex1}       
\end{figure}
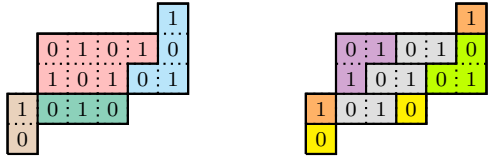

\subsection{Results on Kostant tilings}\label{resKos}

\begin{lemma}\label{minremtab}
Every nonempty skew shape \(\tau\) has a minimal \(\ssee\)-removable ribbon tableau \((\Lambda,\ttt)\).
\end{lemma}
\begin{proof}
We go by induction on \(|\tau|\). If \(|\tau| = 1\), then we trivially take \(\Lambda = \{\tau\}\) and \(\ttt(1) = \tau\). Now let \(|\tau|>1\) and make the induction assumption on skew shapes \(\tau'\) with \(|\tau'| <|\tau|\).
By Lemma~\ref{minremexist}, \(\tau\) has a minimal \(\ssee\)-removable ribbon \(\xi\) such that \(\cont(\xi) \in \Psi\). Let \(\tau' = \tau \backslash \xi\). If \(\tau' = \varnothing\), we are done, taking \(\Lambda = \{\xi\}\), \(\ttt(1) = \xi\). Assume \(\tau' \neq \varnothing\). Then \(\tau'\) is a skew shape by Lemma~\ref{tabisskew}, which by the induction assumption has a minimal \(\ssee\)-removable ribbon tableau \((\Lambda', \ttt')\). But then taking \(\Lambda = \Lambda' \cup \{\xi\}\), \(\ttt(i) = \ttt'(i)\), \(\ttt(|\Lambda|) = \xi\), for all \(i \in [1,|\Lambda'|]\), we have that \((\Lambda, \ttt)\) is a minimal \(\ssee\)-removable ribbon tableau for \(\tau\).
\end{proof}

\begin{lemma}\label{minremtabisKos}
If \((\Lambda,\ttt)\) is a minimal \(\ssee\)-removable ribbon tableau for a skew shape \(\tau\), then \((\Lambda,\ttt)\) is a cuspidal Kostant tableau for \(\tau\).
\end{lemma}
\begin{proof}
By Lemma~\ref{minremiscusp}, every \(\lambda \in \Lambda\) is cuspidal, so it remains to verify that \((\cont(\ttt(i)))_{i=1}^{|\Lambda|}\) is a Kostant sequence by showing that \(\cont(\ttt(j)) \succeq \cont(\ttt(j+1))\) for all \(j \in [1,|\Lambda|-1]\). Indeed, \(\ttt(j+1)\) is a minimal \(\ssee\)-removable ribbon tableau in \(\bigsqcup_{i=1}^{j+1}\ttt(i)\), and \(\ttt(i)\) is a minimal \(\ssee\)-removable ribbon tableau in \(\bigsqcup_{i=1}^{j} \ttt(i)= (\bigsqcup_{i=1}^{j+1}\ttt(i)) \backslash \{ \ttt(j+1)\}\), so the result follows by Lemma~\ref{remsinorder}.
\end{proof}

\begin{lemma}\label{cuspisribbon}
Every cuspidal skew shape is a ribbon.
\end{lemma}
\begin{proof}
Let \(\tau\) be a cuspidal skew shape, so that \(\cont(\tau) \in \Phi_+\), and assume by way of contradiction that \(\tau\) is not a ribbon. By Lemma~\ref{minremtab}, there exists a minimal \(\ssee\)-removable ribbon tableau for \(\tau\). As \(\tau\) is not itself a ribbon, it must be that \(|\Lambda|>1\). By Lemma~\ref{minremtabisKos} we have 
\begin{align*}
\cont(\ttt(1)) \succeq \cdots \succeq \cont(\ttt(|\Lambda|)).
\end{align*}
As
\(
\cont(\ttt(1)) + \cdots + \cont(\ttt(|\Lambda|)) = \cont(\tau), 
\)
we have by Lemma~\ref{GenConvexLem} that
\begin{align*}
\cont(\ttt(1)) \succeq \cont(\tau)\succeq \cont(\ttt(|\Lambda|)).
\end{align*}
But \(\ttt(|\Lambda|)\) is a minimal \(\ssee\)-removable ribbon in \(\tau\), so \(\cont(\ttt(|\Lambda|)) \succ \cont(\tau)\) by Lemma~\ref{shapehooktest}, giving the desired contradiction.
\end{proof}

\begin{lemma}\label{cuspisindiv}
If \(\xi\) is a cuspidal skew shape, then \(\cont(\xi) \in \Psi\).
\end{lemma}
\begin{proof}
By Lemma~\ref{cuspisribbon} we have that \(\xi\) is a ribbon. If \(\cont(\xi) \in \Phi_+ \backslash \Psi\), then \(\cont(\xi) = m\delta\) for some \(m >1\). But then \(\xi\) has an \(\ssee\)-removable skew hook \(\nu \subsetneq \xi\) of content \(\cont(\xi) \succeq \delta\), a contradiction of Lemma~\ref{CuspForSkewHook}, so \(\cont(\xi) \in \Psi\).
\end{proof}

\begin{lemma}\label{othersmaller}
Let \(\Lambda\) be a minimal \(\ssee\)-removable ribbon tiling for \(\tau\), and \(\Lambda'\) be a Kostant tiling for \(\tau\). Then \(\bkap^\Lambda \trianglerighteq_R \bkap^{\Lambda'}\), and \(\bkap^\Lambda = \bkap^{\Lambda'}\) if and only if every \(\lambda' \in \Lambda'\) is a union of tiles \(\lambda \in \Lambda\) with \(\psi(\cont(\lambda')) = \cont(\lambda)\).
\end{lemma}
\begin{proof}
For sake of space, if \(\Lambda, \Lambda'\) are such that every \(\lambda' \in \Lambda'\) is a union of tiles \(\lambda \in \Lambda\) with \(\psi(\cont(\lambda')) = \cont(\lambda)\), we will refer to this as the `union condition'. It is clear from the definition of \(\bkap^\Lambda\), \(\bkap^{\Lambda'}\) that \(\bkap^\Lambda = \bkap^{\Lambda'}\) when \(\Lambda, \Lambda'\) satisfy the union condition.

We prove the claim by induction on \(|\tau|\), the base case \(|\tau| = 1\) being trivial. Assume that \(|\tau| >1\), and the claim holds for all \(\tau'\) with \(|\tau'| < |\tau|\). There exists some \(\nu \in \Lambda'\) which is \(\ssee\)-removable in \(\tau\) and \(\psi(\cont(\nu')) \succeq \psi(\cont(\nu))\) for all \(\nu' \in \Lambda'\).

For all \(\lambda \in \Lambda\), set \(\nu^\lambda = \nu \cap \lambda\). Set \(\Lambda_0 = \{\lambda \in \Lambda \mid \nu^\lambda = \lambda\}\) and \(\Lambda_1 = \{\lambda \in \Lambda \mid \nu^\lambda \neq \varnothing, \lambda\}\).
By Lemma~\ref{minremtabisKos}, each \(\lambda \in \Lambda\) is cuspidal, and \(\nu^\lambda\) is \(\ssee\)-removable in each \(\lambda\) because \(\nu\) is \(\ssee\)-removable in \(\tau\). Thus, for all \(\lambda \in \Lambda_1\), we have
 \(
\cont(\nu^\lambda) = \beta^{\lambda, 1} + \cdots + \beta^{\lambda, r_\lambda}
\)
for some \(\cont(\lambda) \prec \beta^{\lambda, 1}, \cdots, \beta^{\lambda, r_\lambda} \in \Phi_+\). Then we have
\begin{align*}
\cont(\nu) = \sum_{\lambda \in \Lambda_0} \cont(\lambda) + \sum_{\Lambda \in \Lambda_1} \beta^{\lambda, 1} + \cdots + \beta^{\lambda, r_{\lambda}}.
\end{align*}

If there exists \(\lambda \in \Lambda_0\) such that \(\psi(\cont(\nu)) \succ \cont(\lambda)\), then \(\Lambda, \Lambda'\) do not satisfy the union condition, and we have \(\bkap^\Lambda \triangleright_R \bkap^{\Lambda'}\). If there exists \(\lambda \in \Lambda_1\), \(j \in [1, r_\lambda]\) such that \(\psi(\cont(\nu)) \succeq \beta^{\lambda, j}\), then \(\Lambda, \Lambda'\) do not satisfy the union condition, and we again have \(\psi(\cont(\nu)) \succ \cont(\lambda)\), so \(\bkap^\Lambda \triangleright_R \bkap^{\Lambda'}\). 

We may assume then that \(\cont(\lambda) \succeq \psi(\cont(\nu))\) for all \(\lambda \in \Lambda_0\) and \(\beta^{\lambda,j} \succ \psi(\cont(\nu))\) for all \(\lambda \in \Lambda_1, j \in [1,r_\lambda]\). By Lemma~\ref{GenConvexLem}, this implies that \(\Lambda_1 = \varnothing\), and \(\cont(\lambda) = \psi(\cont(\nu))\) for all \(\lambda \in \Lambda_0\). If \(\tau':=\tau \backslash \nu = \varnothing\), then we are done, as \(\Lambda, \Lambda'\) satisfy the union condition and \(\bkap^\Lambda = \bkap^{\Lambda'}\). Assume then that \(\tau' \neq \varnothing\). Then we have that \(\Lambda \backslash \Lambda_{0}\) is a minimal \(\ssee\)-removable ribbon tiling for \(\tau'\), and \(\Lambda'\backslash\{\nu\}\) is a Kostant tiling for \(\tau'\). By the induction assumption, we have \(\bkap^{\Lambda \backslash \Lambda_{0}} \trianglerighteq_R \bkap^{\Lambda' \backslash \{\nu\}}\), with equality if and only if \(\Lambda \backslash \Lambda_{0}\), \( \Lambda' \backslash \{\nu\}\) satisfy the union condition. But since \(\nu = \bigsqcup_{\lambda \in \Lambda_0} \lambda\) and \(\cont(\lambda) = \psi(\cont(\nu))\) for all \(\lambda \in \Lambda_0\), it follows that \(\bkap^\Lambda \trianglerighteq_R \bkap^{\Lambda'}\), with equality if and only if \(\Lambda \backslash \Lambda_{0}\), \( \Lambda' \backslash \{\nu\}\) satisfy the union condition, which occurs if and only if \(\Lambda, \Lambda'\) satisfy the union condition. This completes the induction step, and the proof.
\end{proof} 

\begin{corollary}\label{uniquerib}
Every skew shape has a unique minimal \(\ssee\)-removable ribbon tiling.
\end{corollary}
\begin{proof}
Existence is established in Lemma~\ref{minremtab}.
If \(\Lambda, \Lambda'\) are minimal \(\ssee\)-removable ribbon tilings of \(\tau\), then by Lemma~\ref{othersmaller} we have \(\bkap^{\Lambda} \trianglerighteq_R \bkap^{\Lambda'}\) and \(\bkap^{\Lambda'} \trianglerighteq_R \bkap^{\Lambda}\), so \(\bkap^{\Lambda} = \bkap^{\Lambda'}\) and thus every \(\lambda \in \Lambda\) is a union of tiles in \(\lambda' \in \Lambda'\), and vice versa. It follows that \(\Lambda = \Lambda'\).
\end{proof}

\begin{lemma}\label{cuspisrib}
Every cuspidal Kostant tiling is a minimal \(\ssee\)-removable ribbon tiling.
\end{lemma}
\begin{proof}
We prove the claim by induction on \(|\tau|\), the base case \(|\tau|=1\) being clear. Assume that \(|\tau| >1\), and the claim holds for all \(\tau'\) with \(|\tau'| < |\tau|\). 
Let \((\Lambda, \ttt)\) be a cuspidal Kostant tableau for \(\tau\). Then by Lemma~\ref{cuspisribbon}, each \(\lambda \in \Lambda\) is a ribbon, so \(\ttt(j)\) is an \(\ssee\)-removable ribbon in \(\bigsqcup_{i=1}^j \ttt(i)\) for all \(j \in [1, |\Lambda|]\). We also have that \(\cont(\ttt(j)) \in \Psi\) for all \(j \in [1,|\Lambda|]\) by Lemma~\ref{cuspisindiv}.
If \(|\Lambda| = 1\), the claim is clearly true, so assume \(|\Lambda|>1\). Note that \(\Lambda \backslash \ttt(|\Lambda|)\) is a cuspidal Kostant tiling for \(\tau':= \tau \backslash \ttt(|\Lambda|)\), so if \(\ttt(|\Lambda|)\) is a minimal for \(\tau\), the claim follows by induction. 

Assume by way of contradiction that \(\ttt(|\Lambda|)\) is not minimal for \(\tau\). There exists a minimal \(\ssee\)-removable skew hook \(\xi\) in \(\tau\), and \(\cont(\ttt(j)) \succeq \cont(\ttt(|\Lambda|)) \succ \cont(\xi)\) for all \(j \in [1,|\Lambda|]\). Let \(\xi^j = \xi \cap \ttt(j)\) for all \(j \in [1,|\Lambda|]\), and set \(J = \{  j \in [1,\Lambda] \mid \xi^j \neq \varnothing\}\). Since \(\xi\) is \(\ssee\)-removable in \(\tau\), \(\xi^j\) is \(\ssee\)-removable in \(\ttt(j)\) for all \(j \in J\). But then, since \(\ttt(j)\) is cuspidal, we have that \(\cont(\xi^j) = \beta^{j,1} + \cdots + \beta^{j,r_j}\) for some \(\cont(\ttt(j)) \preceq \beta^{j,1}, \ldots, \beta^{j,r_j} \in \Phi_+\). Then by Lemma~\ref{GenConvexLem} we have
\begin{align*}
\cont(\xi) = \sum_{j \in J} \cont(\xi^j) =  \sum_{j \in J} \beta^{j,1} + \cdots + \beta^{j,r_j} \succeq \cont(\ttt(|\Lambda|)) \succ \cont(\xi),
\end{align*}
the desired contradiction. This completes the induction step and the proof.
\end{proof}

\subsection{Reversal}\label{revsec}
There is an inherent symmetry to much of the combinatorial data considered herein. For \(u \in \N\), define the {\em reversal} \(u^\rev:= (-u_2,-u_1)\) and extend this to \(\tau^\rev:= \{u^\rev \mid u \in \tau\}\) for \(\tau \subset \N\). Reversal preserves residue and content, and sends skew shapes to skew shapes and ribbons to ribbons.

If \((\Lambda, \ttt)\) is a tableau for \(\tau\), then we may define a tableau \((\Lambda^\rev, \ttt^\rev)\) for \(\tau^\rev\) by setting \(\Lambda^\rev :=\{ \lambda^\rev \mid \lambda \in \Lambda\}\) and \(\ttt^\rev(i) = \ttt(|\Lambda| - i +1)^\rev\) for \(i \in [1,|\Lambda|]\).

For a convex preorder \(\succeq\), we may also define the {\em reversal} convex preorder \(\succeq^\rev\) by setting \(\beta \succeq^\rev \beta'\) if and only if \(\beta' \succeq \beta\). 

The following proposition is straightforward to verify from definitions. To avoid confusion, we label here the terms `cuspidal' and `Kostant', and the bilexicographic partial order \(\triangleleft\) which depend upon a chosen convex preorder with the symbol for that preorder.

\begin{proposition}\label{revlem}
Let \(m \in \N\), \(\beta \in \Phi_+\), \(\theta \in Q_+\), \(\xi \in \SSSS(\beta)\), \(\mu \in \SSSS(m \beta)\), \(\tau \in \SSSS(\theta)\).
\begin{enumerate}
\item \(\xi\) is \(\succeq\)-cuspidal if and only if \(\xi^\rev\) is \(\succeq^\rev\)-cuspidal.
\item \(\mu\) is \(\succeq\)-semicuspidal if and only if \(\mu^\rev\) is \(\succeq^\rev\)-semicuspidal.
\item \(\xi\) is a \(\succeq\)-minimal \(\ssee\)-removable ribbon in \(\tau\) if and only if \(\xi^\rev\) is a \(\succeq^\rev\)-maximal \(\nnww\)-removable ribbon in \(\tau^\rev\).
\item \((\Lambda, \ttt)\) is a \(\succeq\)-Kostant tableau for \(\tau\) if and only if \((\Lambda^\rev, \ttt^\rev)\) is a \(\succeq^\rev\)-Kostant tableau for \(\tau^\rev\).
\item For \(\bkap, \bnu \in \Xi(\beta)\), we have \(\bkap \triangleright^{\succeq}_R \bnu\) if and only if \(\bkap \triangleright^{\succeq^\rev}_L \bnu\).
\end{enumerate}
\end{proposition}

\subsection{Main theorems, cuspidal version}

In the next theorem we show that, up to \(e\)-similarity, there is a unique cuspidal skew shape associated to every real positive root, and there are \(e\) distinct cuspidal skew shapes associated to the null root \(\delta\).

\begin{theorem}\label{allcusp}
The set \(\{\zeta^{\beta} \mid \beta \in \Phi_+^\re\} \cup \{\zeta^t \mid t \in \Z_e\}\) represents a complete and irredundant set of cuspidal skew shapes, up to \(e\)-similarity.
\end{theorem}
\begin{proof}
Follows from Lemmas~\ref{cuspisribbon}, \ref{cuspisindiv}, and Proposition~\ref{allcusprib}.
\end{proof}

Parts (i), (iv) of the next theorem establish that every skew shape \(\tau\) possesses a unique cuspidal Kostant tiling, and the associated Kostant partition is bilexicographically maximal among all Kostant tilings for \(\tau\). Moreover, parts (ii),(iii) show that the unique cuspidal Kostant tiling can be directly constructed via progressive minimal (or maximal) ribbon removals. We refer the reader back to Examples~\ref{BigEx} and~\ref{e2Kosex} for demonstrative examples of cuspidal Kostant tilings.

\begin{theorem}\label{mainthmcusp}
Let \(\tau\) be a nonempty skew shape. Then:
\begin{enumerate}
\item There exists a unique cuspidal Kostant tiling \(\Gamma_\tau\) for \(\tau\). 
\item The tiling \(\Gamma_\tau\) is the unique minimal \(\ssee\)-removable ribbon tiling for \(\tau\).
\item The tiling \(\Gamma_\tau\) is the unique maximal \(\nnww\)-removable ribbon tiling for \(\tau\).
\item For any Kostant tiling \(\Lambda\) for \(\tau\), we have \(\bkap(\Gamma_\tau) \trianglerighteq \bkap(\Lambda) \), with \(\bkap(\Gamma_\tau)=\bkap(\Lambda) \) if and only if every \(\lambda \in \Lambda\) is a union of tiles \(\gamma \in \Gamma_\tau\) with \(\psi(\cont(\lambda)) = \cont(\gamma)\).
\end{enumerate}
\end{theorem}
\begin{proof}
Uniqueness in (ii) is provided by Corollary~\ref{uniquerib}, and the equality in (ii) is provided by Lemmas~\ref{minremtabisKos} and~\ref{cuspisrib}. Existence in (i) follows from Lemma~\ref{minremtab}, and the uniqueness in (i) from uniqueness in (ii). Then (iii) follows from (ii) and Proposition~\ref{revlem}. Finally, (iv) follows from (ii) and Lemma~\ref{othersmaller} and Proposition~\ref{revlem}.
\end{proof}

\section{Semicuspidal tableaux}\label{semicuspcons}
In this section we build on \S\ref{Kostsec} to investigate semicuspidal skew shapes and tableaux. Recall that every nonempty skew shape \(\tau\) has a unique cuspidal Kostant tiling \(\Gamma_\tau\), as in Theorem~\ref{mainthmcusp}.

\begin{lemma}\label{sciffcusptile}
Let \(\tau\) be a skew shape of content \(\theta \in \Phi_+'\). Then \(\tau\) is semicuspidal if and only if, for all \(\gamma \in \Gamma_\tau\) we have \(\cont(\gamma) = \psi(\theta)\).
\end{lemma}
\begin{proof}
Let \((\gamma_i)_{i=1}^k\) be a cuspidal Kostant tableau for \(\tau\). Then \(\Gamma_\tau = \{\gamma_i\}_{i=1}^k\), and \(\gamma_i\) is a ribbon for all \(i \in [1,k]\). 

\((\implies)\) Since \(\tau\) is semicuspidal, we must have \(\cont(\gamma_1) \succeq \cont(\gamma_k) \succeq \psi(\theta)\) by Lemma~\ref{shapehooktest}. Then by Lemma~\ref{GenConvexLem}(i),(ii), we have that \(\cont(\gamma_i) \approx \psi(\theta)\) for all \(i\). But \(\cont(\gamma_i) \in \Psi\) for all \(i\), so by Lemma~\ref{notapprox}, we have \(\cont(\gamma_i) = \psi(\theta)\) for all \(i \in [1,k]\). 

\((\impliedby)\) Let \((\lambda_1, \lambda_2)\) be any tableau for \(\tau\). Define:
\begin{align*}
\Gamma^0_\tau &= \{\gamma \in \Gamma_\tau \mid \gamma \cap \lambda_1 \neq \varnothing, \gamma\cap \lambda_2 \neq \varnothing\};\\
 \Gamma^1_\tau &= \{\gamma \in \Gamma_\tau \mid \gamma \subseteq \lambda_1\};\\
  \Gamma^2_\tau &= \{\gamma \in \Gamma_\tau \mid \gamma \subseteq \lambda_2\}.
  \end{align*}
Let \(\gamma \in \Gamma^0_\tau\). Then \((\gamma \cap \lambda_1, \gamma \cap \lambda_2)\) is a tableau for \(\gamma\). As \(\gamma\) is cuspidal, we have that \(\cont(\gamma \cap \lambda_1)\) is a sum of positive roots less than \(\psi(\theta)\), and \(\cont(\gamma \cap \lambda_2)\) is a sum of positive roots greater than \(\psi(\theta)\). By assumption, \(\cont(\gamma) = \psi(\theta)\) for all \(\gamma \in \Lambda_1, \Lambda_2\). Then we have
\begin{align*}
\cont(\lambda_1) &= \sum_{\gamma \in \Gamma_\tau^0} \cont(\gamma \cap \lambda_1) + \sum_{\gamma \in \Gamma^1_\tau} \psi(\theta);\\
\cont(\lambda_2) &= \sum_{\gamma \in \Gamma_\tau^0} \cont(\gamma \cap \lambda_2) + \sum_{\gamma \in \Gamma^2_\tau} \psi(\theta),
\end{align*}
so it follows that \(\cont(\lambda_1)\) can be written as a sum of positive roots greater than or equal to \(\cont(\psi(\theta))\), and \(\cont(\lambda_2)\) can be written as a sum of positive roots greater than or equal to \(\cont(\psi(\theta))\). Thus \(\tau\) is semicuspidal.
\end{proof}

As all cuspidal shapes are (connected) ribbons by Theorem~\ref{allcusp}, any cuspidal tiling of a skew shape is a union of cuspidal tilings of its connected components. Therefore Lemma~\ref{sciffcusptile} implies the following

\begin{corollary}\label{concompsc}
Let \(m \in \NN\) and \(\beta \in \Phi_+\). A skew shape \(\tau \in \SSSS(m \beta)\) is semicuspidal if and only if every connected component of \(\tau\) is a semicuspidal skew shape of content \(m'\beta\) for some \(m' \leq m\). 
\end{corollary}

\subsection{Real semicuspidal skew shapes} First we focus on semicuspidal skew shapes associated to real positive roots.

\begin{lemma}\label{conscreal}
Let \(\tau\) be a connected skew shape of content \(m \beta \in \Phi_+'\), where \(\beta \in \Phi_+^\re\). Then \(\tau\) is semicuspidal if and only if \(m=1\) and \(\tau \sim_e \zeta^\beta\).
\end{lemma}
\begin{proof}
The `if' direction is immediate by Lemma~\ref{allcusp}, as \(\zeta^\beta\) is cuspidal. Now assume \(\tau\) is semicuspidal. Then by Lemma~\ref{sciffcusptile} and Theorem~\ref{allcusp} there exists a cuspidal Kostant tableau \((\gamma_i)_{i=1}^m\) for \(\tau\) such that \(\gamma_i \sim_e \zeta^\beta\) for all \(i \in [1,m]\). If \(m=1\), we are done.

Assume by way of contradiction that \(m>1\).
By Theorem~\ref{mainthmcusp}, every minimal \(\ssee\)-removable ribbon in \(\tau\) is a tile in \(\Gamma_\tau\). Let \(\mu\) be the union of all minimal \(\ssee\)-removable ribbons in \(\tau\). Then \(\mu\) is a skew shape such that each of its connected components is \(e\)-similar to \(\zeta^\beta\). Let \(\tau' = \tau \backslash \mu\). If \(\tau'\) is empty, then \(\mu = \tau\) would consist of \(m\) disconnected skew shapes, a contradiction, since \(\tau\) is connected. Thus \(\tau'\) is nonempty. Let \(\gamma'\) be a minimal \(\ssee\)-removable ribbon in \(\tau'\). By Theorem~\ref{mainthmcusp}, \(\gamma' \sim_e \zeta^\beta\) is a tile in \(\Gamma_\tau\). As \(\gamma'\) was {\em not} removable in \(\tau\), it follows that there is some minimal \(\ssee\)-removable ribbon \(\gamma \sim_e \zeta^\beta\) in \(\tau\) such that \(\nu:=\gamma \sqcup \gamma'\) is a connected skew shape.

We now will derive a contradiction by focusing on this shape \(\nu\). 
Note that \((\gamma', \gamma)\) is a cuspidal Kostant tableau for \(\nu\).
For clarity, we write \(\nu' = \gamma' = \nu \backslash \gamma\). 
We have \(\gamma = \xi^\nu_{u,v}\) and \(\gamma' = \xi^{\nu'}_{z,w}\), for some \((u,v) \in \textup{Rem}_\tau{\nu}\), \((z,w) \in \textup{Rem}_\tau{\nu'}\). Since \(\gamma \sim_e \gamma'\), the nodes \(u,v,z,w\) must be arranged as in one of the following cases:

{\em Case 1:} \(u \nearrow v \nearrow z \nearrow w\). As \(\nu\) is connected, this implies that \(z \in \{\ea v, \no v\}\), so \(\nu\) is a ribbon. But then \(m \beta = \cont(\nu) \in \Phi_+\) by Corollary~\ref{skewpos}, a contradiction since \(m >1\).

{\em Case 2:} \(z \nearrow w \nearrow u \nearrow v\). As \(\nu\) is connected, this implies that \(u \in \{\ea w, \no w\}\), so \(\nu\) is a ribbon, and we get a contradiction as in Case 1.

{\em Case 3:} \(u \nearrow z \nearrow v \nearrow w\). Then by Lemma~\ref{remhookswap}(iii), we have that \((\ssee z,w) \in \textup{Rem}_\tau\nu\). Thus \(\xi^\nu_{(\ssee z, w)}\) is an \(\ssee\)-removable hook in \(\nu\). Note that \(\cont(\xi^\nu_{\ssee z, w}) = \beta = \cont(\gamma)\). Then \(\xi^\nu_{(\ssee z, w)}\) is minimal in \(\nu\) since \(\gamma\) is minimal by Theorem~\ref{mainthmcusp}. Thus \(\xi^\nu_{\ssee z, w}\) is cuspidal by Lemma~\ref{minremiscusp}. But \(\xi^\nu_{\ssee z, w} \not \sim_e \xi^{\nu'}_{z,w}  = \gamma' \sim_e \zeta^\beta\), a contradiction of Theorem~\ref{allcusp}.

{\em Case 4:} \(z \nearrow u \nearrow w \nearrow v\).  Then by Lemma~\ref{remhookswap}(iv), we have that \((z, \ssee w) \in \textup{Rem}_\tau\nu\), so \(\xi^\nu_{z, \ssee w}\) is an \(\ssee\)-removable hook in \(\nu\), which derives a contradiction along the same lines as Case 4. 

This exhausts the possibilities for arrangements of the nodes \(u,v,z,w\), so we get a contradiction in any case. Therefore \(m=1\), as desired.
\end{proof}

Recall the \(e\)-similarity classes defined in \S\ref{similaritysec}, and choose, for every \(m \in \NN\) and \(\beta \in \Phi_+^\re\), a distinguished skew shape \(\zeta^{m \beta} \in [(\zeta^\beta)^m]_e\).

\begin{corollary}\label{realsemicor}
Let \(m \in \NN\), \(\beta \in \Phi_+^\re\). Assume \(\tau \in \SSSS(m\beta)\). 
Then \(\tau\) is semicuspidal if and only if \(\tau \sim_e \zeta^{m \beta}\). 
\end{corollary}
\begin{proof}
For the `only if' direction, assume \(\tau\) is semicuspidal. Then by Corollary~\ref{concompsc}, all connected components of \(\tau\) are semicuspidal, and each of these is \(e\)-similar to \(\zeta^\beta\) by Lemma~\ref{conscreal}, implying the result. 
The `if' direction is granted by Lemma~\ref{sciffcusptile}, since any \(\tau \in [(\zeta^\beta)^m]_e\) is trivially tiled by tiles \(e\)-similar to \(\zeta^\beta\).
\end{proof}

\subsection{The dilation map}\label{infldef}
In this subsection, fix some \(t \in \Z_e\), and recall that \(
[\zeta^t]_e =  \{\zeta^{(\delta, b)} \mid b \in \N_t\}
\) by Proposition~\ref{allcusprib}. Recall that \(\zeta^t = \zeta^{(\delta, b^t)}\), where \(b^t \in \N_t\). Set:
\begin{align*}
x^t= \ea(\zeta^{t}_{\nnee} - b^t),
\qquad
\textup{and}
\qquad
 y^t= \no(\zeta^{t}_{\nnee} - b^t).
\end{align*}
By construction, \(\res(\zeta^{t}_{\nnee}) = \overline{t-1}\), so it follows that \(\res(x^t) = \res(y^t) = \overline 0\). 
We define a map \(\phi_t: \N \to \N_t\) by setting
\begin{align*}
\phi_t(u) = b^t+ u_1y^t + u_2x^t,
\end{align*}
for all \(u \in \N\). 
Then we define the {\em \(t\)-dilation} map \(\infl_t: \N \to [\zeta^t]_e\) by setting:
\begin{align*}
\infl_t(u) = \zeta^{(\delta, \phi_t(u))}.
\end{align*}

The following lemma establishes that that \([\zeta^t]_e\) is a tiling of \(\N\).

\begin{lemma}\label{tilelem}
For all \(u \in \N\), there exists a unique \(b \in \N_t\) such that \(u \in \zeta^{(\delta,b)}\).
\end{lemma}
\begin{proof}
Let \(b' \in \N_t\). Then, since \(\cont(\zeta^{(\delta, b')}) = \delta\), we have that \(\zeta^{(\delta, b')}\) contains a node \(v\) with \(\res(v) = \res(u)\). Then  \(u = v+ c\) some node \(c \in \N_0\). 
Take \(b = b' + c\). Then \(\res(b) = t\). We have \(u \in \T_c(\zeta^{(\delta,b')}) = \zeta^{(\delta, b' + c)} = \zeta^{(\delta,b)}\) by Proposition~\ref{allcusprib}., establishing existence.

Now assume \(b'' \in \N_t\), and \(u \in \zeta^{(\delta, b'')}\). Then \(b'' = b + d\) for some \(d \in \N_{\overline 0}\). Then \(\zeta^{(\delta,b'')} = \T_d(\zeta^{(\delta, b)})\), so \(u= u'+ d\) for some \(u' \in \zeta^{(\delta, b)}\) with \(\res(u') = \res(u)\). But, as \(\cont(\zeta^{(\delta,b)}) = \delta\), the only node \(u' \in \zeta^{(\delta,b)}\) with residue \(\res(u)\) is \(u\), so \(u' = u\), and thus \(d=(0,0)\). Thus \(b'' = b\), establishing uniqueness.
\end{proof}

\begin{lemma}\label{N0basis}
The set \(\N_{\overline 0}\) is a free \(\Z\)-module with basis \(\{x^t,y^t\}\).
\end{lemma}
\begin{proof}
We first show that \(\ZZ\{x^t,y^t\}\) spans \(\N_{\overline 0}\). Note that by definition we have \(x^t = \ssee y^t\), so \((1,1) = x^t-y^t \in \ZZ\{x^t,y^t\}\). Then \(x^t - x^t_1(1,1) = (0,x^t_2-x^t_1) \in \ZZ\{x^t,y^t\}\).
By Corollary~\ref{sizehook}, we have
\begin{align*}
e&= \dist(\zeta^{t}_{\ssww}, \zeta^{t}_{\nnee}) + 1 
=\dist(b^t, \zeta^{t}_{\nnee}) + 1
 =|(\zeta^{t}_{\nnee})_1 - b^t_1| + |(\zeta^{t}_{\nnee})_2 - b^t_2| + 1\\
 &=|x^t_1| + |x^t_2-1| + 1 = -x^t_1 + (x^t_2-1) + 1 = x^t_2 - x^t_1.
\end{align*}
Thus \((0,e) \in \ZZ\{x^t,y^t\}\). Let \(u \in \N_{\overline 0}\). Then we have \(u = (u_1,u_1) + (0,ke)\) for some \(k \in \ZZ\), so \(u = u_1(1,1) + k(0,e) \in \ZZ\{x^t,y^t\}\). Thus \(\ZZ\{x^t,y^t\} = \N_{\overline 0}\).

Now assume \(cx^t + dy^t = 0\) for some \(c,d \in \ZZ\). Then we have
\begin{align*}
0 &=cx^t + dy^t = c x^t+ d\nnww x^t = c(x^t_1,x^t_2) + d(x^t_1-1,x^t_2-1)\\
&= ((c+d)x^t_1 - d, (c+d)x^t_2 - d).
\end{align*}
It follows that \(0 = (c+d)(x^t_2-x^t_1) = (c+d)e\), so \(c+ d = 0\), which implies that \(c=d=0\). Thus \(\{x^t,y^t\}\) are linearly independent, and so constitute a basis for \(\N_{\overline 0}\).
\end{proof}

\begin{lemma}\label{phibij}
The map \(\phi_t: \N \to \N_t\) is a bijection.
\end{lemma}
\begin{proof}
First we show surjectivity. If \(b' \in \N_t\), then \(b' = b^t+c\) for some \(c \in \N_0\). Then by Lemma~\ref{N0basis}, \(c = ry^t + sx^t\) for some \(r,s \in \ZZ\). Then we have \(\phi_t((r,s)) = b^t + ry^t + sx^t = b^t+c = b'\), as desired.

Next we show injectivity. Assume \(\phi_t(u) = \phi_t(v)\). Then we have \(b^t+u_1y^t + u_2x^t = b^t+ v_1y^t + v_2x^t\), so \(u_1y^t + u_2x^t = v_1y^t + v_2x^t\). But then by Lemma~\ref{N0basis}, this implies that \(u = v\), as desired.
\end{proof}

\begin{corollary}\label{inflbij}
The \(t\)-dilation map \(\infl_t : \N \to [\zeta^t]\) is a bijection.
\end{corollary}

\begin{lemma}\label{inflmoves1}
Let \(z \in \N\). Then 
\begin{align*}
\infl_t(\ea z) =\T_{x^t}( \infl_t(z));
\;\;\;
\infl_t(\no z) =\T_{y^t}( \infl_t(z));
\;\;\;
\infl_t(\ssee z) = \ssee(\infl_t(z)).
\end{align*}
\end{lemma}
\begin{proof}
We have
\begin{align*}
\infl_t( \ea z) &= \infl_t( (z_1, z_2 + 1)) = \zeta^{(\delta, \phi_t(z_1,z_2+1))}\\
&= \zeta^{(\delta, b^t+ z_1y^t + z_2x^t + x^t)} = \T_{x^t}(\zeta^{(\delta, b^t+ z_1y^t + z_2x^t)} ) = \T_{x^t}(\infl_t(z)).
\end{align*}
The second equality is similar. The proof of the last equality follows from the first two and the fact that \(x^t-y^t = (1,1)\).
\end{proof}

\begin{lemma}\label{inflmoves2}
Let \(u,z \in \N\), with \(u \in \infl_t(z)\). Then we have:
\begin{enumerate}
\item \(u = (\infl_t(z))_\nnee \iff \ea u = (\infl_t(\ea z))_\ssww \iff \no u = (\infl_t(\no z))_\ssww\).
\item \(u \neq (\infl_t(z))_\nnee \implies \ea u \in \infl_t(z) \sqcup \infl_t(\ssee z)\)
\item \(u = (\infl_t(z))_\ssww \iff \so u = (\infl_t(\so z))_\nnee \iff \we u = (\infl_t(\we z))_\nnee\)
\item \(u \neq (\infl_t(z))_\ssww \implies \so u \in \infl_t(z) \sqcup \infl_t(\ssee z)\)
\end{enumerate}
\end{lemma}
\begin{proof}

We have \(u = (\infl_t(z))_{\nnee}\) if and only if \(u = \phi_t(z) + x^t - (0,1)\), which occurs if and only if 
\begin{align*}
\ea u = \phi_t(z) + x^t = \phi_t(\ea z) = \zeta^{(\delta, \phi_t(\ea z))}_\ssww = (\infl_t(\ea z))_\ssww,
\end{align*}
proving claim (i). Claim (iii) is similar. For claim (ii), assume \(u \neq (\infl_t(z))_\nnee\). 
If \(\ea u \in \infl_t(z)\), we are done. Assume not. Then we must have that \(\no u \in \infl_t(z)\) since \(\infl_t(z)\) is a ribbon. But then \(\ea u = \ssee \no u \in \ssee(\infl_t(z)) = \infl_t(\ssee z)\) by Lemma~\ref{inflmoves1}, as desired. Claim (iv) is similar.
\end{proof}

We extend the dilation map to subsets \(\tau\) of \(\N\) by setting
\(
\infl_t(\tau) = \bigsqcup_{u \in \tau} \infl_t(u).
\)

\begin{lemma}\label{skewtoskew}
Let \(\tau\) be a nonempty finite subset of \(\N\). Then:
\begin{enumerate}
\item \(\infl_t(\tau)\) is cornered if and only if \(\tau\) is cornered;
\item \(\infl_t(\tau)\) is diagonal convex if and only if \(\tau\) is diagonal convex;
\item \(\infl_t(\tau)\) is a connected skew shape if and only if \(\tau\) is a connected skew shape;
\item \(\infl_t(\tau)\) is a ribbon if and only if \(\tau\) is a ribbon.
\end{enumerate}
\end{lemma}
\begin{proof}
We begin with (i).
If \(v \in \infl_t(\tau)\), then \(v \in \infl_t(u)\) for some \(u \in \tau\). Since \(\infl_t(u)\) is a ribbon, we have \(\{\no v, \ea v\} \cap \infl_t(\tau) = \varnothing\) only if \(v = \infl_t(u)_\nnee\), and \(\{\so v, \we v\} \cap \infl_t(\tau) = \varnothing\) only if \(v = \infl_t(u)_\ssww\). 

Now let \(u \in \tau\). 
Then it follows that \((\infl_t(u))_\nnee \in \infl_t(\tau)\). By Lemma~\ref{inflmoves2}(i), we have \(\ea(\infl_t(u))_\nnee \in (\infl_t(\ea u))\) and \(\no(\infl_t(u))_\nnee \in (\infl_t(\no u))\). Then \(u \in \tau\) is a node such that \(\{\no u,  \ea u\} \cap \tau = \varnothing\) if and only if \(\infl_t(u)_\nnee\) is a node in \(\infl_t(\tau)\) such that \(\{\no(\infl_t(u)_\nnee), \ea(\infl_t(u)_\nnee)\} \cap \infl_t(\tau) = \varnothing\). By a similar argument, \(u \in \tau\) is a node such that \(\{\so u,  \we u\} \cap \tau = \varnothing\) if and only if \(\infl_t(u)_\ssww\) is a node in \(\infl_t(\tau)\) such that \(\{\so(\infl_t(u)_\ssww), \we(\infl_t(u)_\ssww)\} \cap \infl_t(\tau) = \varnothing\). From this, and the argument in the previous paragraph, it follows that \(\infl_t(\tau)\) is cornered if and only if \(\tau\) is cornered.

Now we prove (ii).
Assume \(\tau\) is diagonal convex. Assume that \(n \in \NN\) is such that \(u,w \in \D_n \cap \infl_t(\tau)\), \(v \in \D_n\), and \(u \searrow v \searrow w\). We have \(v = (\ssee)^ku\), \(w = (\ssee)^\ell u\), for some \(1 \leq k \leq \ell\).
We have \(u \in \infl_t(z)\) for some \(z \in \tau\). Then, applying Lemma~\ref{inflmoves1}, we have \(v \in (\ssee)^k(\infl_t(z)) = \infl_t((\ssee)^kz)\), \(w \in (\ssee)^\ell(\infl_t(z)) = \infl_t((\ssee)^\ell z)\). Since \(w \in \infl_t(\tau)\), we have \((\ssee)^\ell z \in \tau\). By diagonal convexity of \(\tau\), it follows that \((\ssee)^k z \in \tau\). Thus \(\infl_t((\ssee)^kz) \subseteq \infl_t(\tau)\), so \(v \in \infl_t(\tau)\). Thus \(\infl_t(\tau)\) is diagonal convex.

Now assume \(\infl_t(\tau)\) is diagonal convex, and \(n \in \NN\) is such that \(u,w \in \D_n \cap \tau\), \(v \in \D_n\), and \(u \searrow v \searrow w\). We have \(v = (\ssee)^ku\), \(w = (\ssee)^\ell u\), for some \(1 \leq k \leq \ell\). Let \(z \in \infl_t(u) \subset \infl_t(\tau)\). 
 Then by Lemma~\ref{inflmoves1}, we have that \((\ssee)^k z \in (\ssee)^k \infl_t(u) = \infl_t((\ssee)^k z) = \infl_t(v)\), and similarly, \((\ssee)^\ell z \in \infl_t(w) \subseteq \infl_t(\tau)\). But, as \(\infl_t(\tau)\) is diagonal convex, it follows that \((\ssee)^kz \in \infl_t(\tau)\), which implies that \(v \in \tau\). Thus \(\tau\) is diagonal convex. This completes the proof of (ii).
 
 With (i),(ii) in hand, (iii) follows by Proposition~\ref{CSS=PDC}. With (iii) in hand, we need only check that \(\tau\) is thin if and only if \(\infl_t(\tau)\) is thin. This follows easily from Lemma~\ref{inflmoves1}.
\end{proof}

\begin{lemma}\label{onecolorskew}
Assume \(\tau\) is a nonempty skew shape, and \(\tau = \tau_1 \sqcup \cdots \sqcup \tau_k\), where \(\tau_k \NEarrow \cdots \NEarrow \tau_1\) are the connected components of \(\tau\). Then \(\infl_t(\tau) = \infl_t(\tau_1) \sqcup \cdots \sqcup \infl_t(\tau_k) \in \SSSS\), and \(\infl_t(\tau_k) \NEarrow \cdots \NEarrow \infl_t(\tau_1)\) are the connected components of \(\infl_t(\tau)\).
\end{lemma}
\begin{proof}
By definition of the dilation map, we have \(\infl_t(u) \NEarrow \infl_t(v)\) whenever \(u \NEarrow v\), so we have that \(\infl_t(\tau_1) \NEarrow \cdots \NEarrow \infl_t(\tau_k)\), and these are connected skew shapes by Lemma~\ref{skewtoskew}, so \(\infl_t(\tau) \in \SSSS\), as desired.
\end{proof}

\begin{lemma}\label{deltile}
Let \(\tau\) be a nonempty skew shape.
Then 
\(\Gamma_{\infl_t(\tau)} = \{\infl_t(u) \mid u \in \tau\}\), and \(\infl_t(\tau)\) is semicuspidal.
\end{lemma}
\begin{proof}
By Lemma~\ref{onecolorskew}, we have that \(\infl_t(\tau)\) is a skew shape, and if \(u\) is an \(\ssee\)-removable node in \(u\), then \(\infl_t(u)\) is an \(\ssee\)-removable hook in \(\infl_t(\tau)\) by Lemmas~\ref{inflmoves1} and~\ref{inflmoves2}. Then \(\infl_t(\tau \backslash \{u\}) = \infl_t(\tau)\backslash\{\infl_t(u)\}\), and it follows by induction that \(\{\infl_t(u) \mid u \in \tau\}\) is a tiling of \(\infl_t(\tau)\). The fact that \(\infl_t(\tau)\) is semicuspidal follows by Lemma~\ref{sciffcusptile}.
\end{proof}

\subsection{Connected imaginary semicuspidal skew shapes}

\begin{proposition}\label{conissemicusp}
Let \(\tau\) be a connected skew shape with \(\cont(\tau) = m\delta\). Then \(\tau\) is semicuspidal if and only if there exists \(t \in \Z_e\) and \(\mu \in \SSSS_{\tt c}(m)\) such that \(\tau = \infl_t(\mu)\).
\end{proposition}
\begin{proof}
The `if' direction is supplied by Lemma~\ref{deltile} and Lemma~\ref{skewtoskew}(iii). We focus now on the `only if' direction. 
By Lemma~\ref{sciffcusptile}, we may assume \((\gamma_1, \ldots, \gamma_m)\) is a cuspidal Kostant tableau for \(\tau\), where \(\gamma_i = \zeta^{(\delta, c^i)}\) for some \(c^i \in \N\). Write \(\tau' := \tau \backslash \gamma_m\). If \(\tau' \neq \varnothing\), then \(\tau'\) is semicuspidal by Lemma~\ref{sciffcusptile}.

We go by induction on \(m\). The base case \(m=1\) is immediate, since then we have that \(\tau=\gamma_1\) is cuspidal. Assume \(m=2\).
Assume first, by way of contradiction, that \(\res(c^1) \neq \res(c^2)\). By Lemma~\ref{xiuv}, there exist \((u,v) \in \textup{Rem}_\tau\tau\) such that \(\xi^\tau_{u,v} = \zeta^{(\delta, c^2)}\). Writing \(\tau' := \tau \backslash \xi^\tau_{u,v} = \zeta^{(\delta, c^1)} \), we have some \((z,w) \in \textup{Rem}_\tau{\tau'}\) such that \(\xi^{\tau'}_{z,w} = \zeta^{(\delta, c^1)}\). As \(\tau\) is connected, we have that \(u,v,z,w\) are trivially in the same connected component of \(\tau\). As the two ribbons have the same cardinality, the nodes \(u,v,z,w\) must be arranged as in one of the following cases:

{\em Case 1:} \(u \nearrow v \nearrow z \nearrow w\). As \(\tau\) is connected, this implies that \(z \in \{\ea v, \no v\}\), and thus
\begin{align*}
\res(c^2) = \res(z) = \res(v)+\overline 1 = \res(\zeta^{(\delta,c^1)}_\nnee) + \overline 1 = (\res(c^1) - \overline 1)+ \overline 1 = \res(c^1),
\end{align*} 
a contradiction. 

{\em Case 2:} \(z \nearrow w \nearrow u \nearrow v\). As \(\nu\) is connected, this implies that \(u \in \{\ea w, \no w\}\), which again forces \(\res(c^2) = \res(c^1)\), a contradiction as in Case 1.

{\em Case 3:} \(u \nearrow z \nearrow v \nearrow w\). Then by Lemma~\ref{remhookswap}(iii), we have that \((\ssee z,w) \in \textup{Rem}_\tau\nu\). Thus \(\xi^\tau_{(\ssee z, w)}\) is an \(\ssee\)-removable hook in \(\tau\). Note that \(\cont(\xi^\tau_{\ssee z, w}) = \delta\), so 
\(\xi^\nu_{(\ssee z, w)}\) is minimal in \(\tau\) since \(\gamma_2\) is minimal by Theorem~\ref{mainthmcusp}. Thus \(\xi^\nu_{\ssee z, w}\) is cuspidal by Lemma~\ref{minremiscusp}. We have \(\res(\ssee z) = \res(z) = \res(c^1)\). But then we must have by Proposition~\ref{allcusprib} that \(\xi^\tau_{\ssee z, w} = \zeta^{(\delta, \ssee z)} = \ssee(\zeta^{(\delta, z)})\). But then
\begin{align*}
w = (\xi^\tau_{\ssee, w})_\nnee = (\ssee(\zeta^{(\delta, z)}))_\nnee = \ssee w,
\end{align*}
a contradiction. 

{\em Case 4:} \(z \nearrow u \nearrow w \nearrow v\).  Then by Lemma~\ref{remhookswap}(iv), we have that \((z, \ssee w) \in \textup{Rem}_\tau\nu\), so \(\xi^\nu_{z, \ssee w}\) is an \(\ssee\)-removable hook in \(\nu\), which derives a contradiction along the same lines as Case 4. 

Thus, in any case we derive a contradiction, so \(\res(c^1) = \res(c^2) = t\) for some \(t \in \Z_e\). Then we have by Lemma~\ref{skewtoskew} that \(\tau = \infl_t(\mu)\) for some connected skew shape \(\mu \in \SSSS_{\tt c}(2)\), and we are done.

Now let \(\tau\) be a connected, cuspidal shape with content \(m\delta\), where \(m \geq 3\), and make the induction assumption on all \(m' < m\). First assume that \(\tau'\) is disconnected. 
Let \(\mu\) be any connected component of \(\tau'\). Then \(\mu = \sqcup_{i \in I} \gamma_i\) for some \(I \subset [1,m-1]\), and \(\mu \sqcup \gamma_m = \sqcup_{i \in I\cup\{m\}} \gamma_i\) is a connected semicuspidal skew shape by Lemma~\ref{sciffcusptile}. 
Then by the induction assumption we have \(\res(c^i) = \res(c^m)\) for all \(i \in I\). Applying this argument to all connected components of \(\tau'\) gives the result. 

Now assume that \(\tau'\) is connected. By the induction assumption we have \(\res(c^i) = \res(c^{m-1})\) for all \(i \in [1,m-1]\). If \(\gamma_{m-1} \sqcup \gamma_m\) is disconnected, then \((\gamma_1, \ldots, \gamma_{m-2}, \gamma_m, \gamma_{m-1})\) is a cuspidal Kostant tableau for \(\tau\). If \(\tau \backslash \gamma_{m-1}\) is disconnected, then it follows that \(\res(c^i) = \res(c^{m-1})\) for all \(i \in [1,m]\), giving the result. If \(\tau \backslash \gamma_{m-1}\) is connected, then by the induction assumption we have \(\res(c^m) = \res(c^1) = \res(c^{m-1})\), so \(\res(c^i) = \res(c^m)\) for all \(i \in I\).

Finally, assume \(\gamma_{m-1} \sqcup \gamma_m\) is connected. By the case for \(m=2\) above, we have that \(\res(c^{m-1}) = \res(c^m)\), so \(\res(c^i) = \res(c^m)\) for all \(i \in I\), completing the proof.
\end{proof}

\begin{lemma}\label{delconsim}
Let \(t_1, t_2 \in \ZZ_e\),  \(\mu, \nu \in \SSSS_{\tt c}\). Then \(\infl_{t_1}(\mu) \sim_e \infl_{t_2}(\nu)\) if and only if \(\mu \sim \nu\) and \(t_1 =t_2\).
\end{lemma}
\begin{proof}
First, note that if \(t_1 = t_2 = t\), then by the definition of the dilation map, there exists \(c \in \N\) such that \(\T_c(\mu) = \nu\) if and only if \(\T_{c_1y^t + c_2x^t}(\infl_t(\mu)) = \infl_t(\nu)\). The result follows. If \(t_1 \neq t_2\), then we note that \(\res(\infl_{t_1}(\mu)_{\ssww}) = t_1\), and \(\res(\infl_{t_2}(\nu)_{\ssww}) = t_2\), so \(\infl_{t_1}(\mu) \not \sim_e \infl_{t_2}(\nu)\), completing the proof.
\end{proof}

\subsection{Arbitrary imaginary semicuspidal skew shapes}
Similarity is an equivalence relation on \(\SSSS_{\tt c}\). Let \(\widetilde{\SSSS}_c \subset \SSSS_{\tt c}\) be a set of distinguished representatives from each similarity class. For \(k \leq m \in \NN\), we write
\begin{align*}
\SSSSS(k,m)&:= \{ (\btau, \bbep) \mid \bbep \in \Z_e^k, \btau \in \widetilde{\SSSS}_c^k, |\tau_1| + \cdots + |\tau_k| = m\};\\
\SSSSS(m)&:= \bigsqcup_{k \in \NN} \SSSSS(k,m).
\end{align*}

For each \((\btau, \bbep) \in \SSSSS(k,m)\), we choose a distinguished skew shape \(\zeta^{(\btau, \bbep)} \in \SSSS(m\delta)\) from the \(e\)-similarity class:
\begin{align*}
\zeta^{(\btau, \bbep)} \in [ \infl_{\varepsilon_1}(\tau_1), \ldots, \infl_{\varepsilon_k}(\tau_k)]_e.
\end{align*}
 
\begin{lemma}\label{arbdelsemi}
Let \(m \in \NN\). Then the set \(\{ \zeta^{(\btau, \bbep)} \mid (\btau, \bbep) \in \SSSSS(m)\}\) is a complete and irredundant set of semicuspidal skew shapes of content \(m \delta\), up to \(e\)-similarity.
\end{lemma}
\begin{proof}
Completeness of this set follows from Proposition~\ref{conissemicusp} and Corollary~\ref{concompsc}. Irredundancy follows from
Lemma~\ref{delconsim}.
\end{proof}

\begin{example}\label{e2dil}
Let \(e=2\) and recall the imaginary cuspidal ribbons \(\zeta^{\overline{0}}, \zeta^{\overline{1}}\) from Example~\ref{e2shapes}. In Figure~\ref{fig:e2dil}, we give an example of an element \((\btau, \bbep) \in \SSSSS(3,15)\), and display the corresponding dilated semicuspidal skew shape \(\zeta^{(\btau,\bbep)} \in \SSSS(15\delta)\).
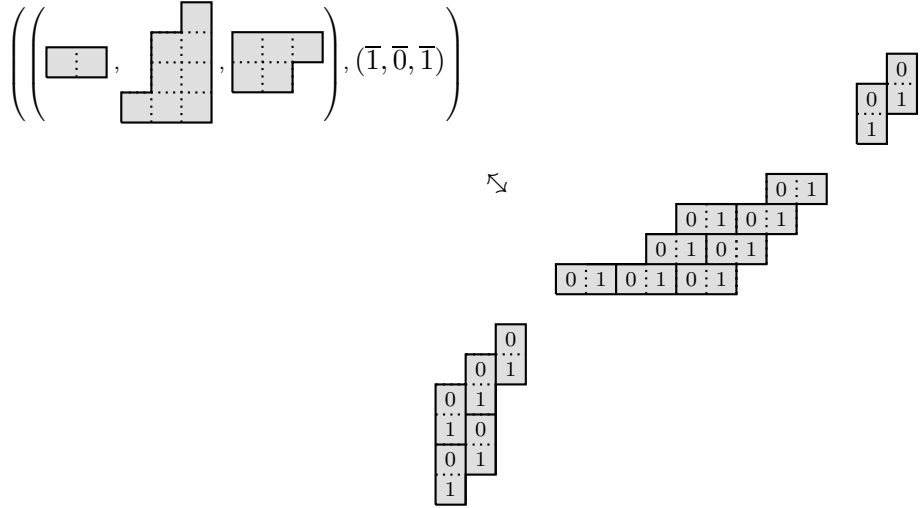
\begin{figure}[h]
\begin{align*}
{}
\substack{
\left(
\left(
\hackcenter{
\begin{tikzpicture}[scale=0.4]
%
\draw[thick, fill=lightgray!50] (0,0)--(2,0)--(2,1)--(0,1)--(0,0);
%
\draw[thick, dotted] (1,0)--(1,1);
\end{tikzpicture}
}
\;
,
\hackcenter{
\begin{tikzpicture}[scale=0.4]
%
\draw[thick, fill=lightgray!50] (0,0)--(3,0)--(3,4)--(2,4)--(2,3)--(1,3)--(1,1)--(0,1)--(0,0);
%
\draw[thick, dotted] (0,1)--(3,1);
\draw[thick, dotted] (1,2)--(3,2);
\draw[thick, dotted] (1,3)--(3,3);
\draw[thick, dotted] (1,0)--(1,3);
\draw[thick, dotted] (2,0)--(2,3);
\end{tikzpicture}
}
\;
,
\;
\hackcenter{
\begin{tikzpicture}[scale=0.4]
%
\draw[thick, fill=lightgray!50] (0,0)--(2,0)--(2,1)--(3,1)--(3,2)--(0,2)--(0,0);
%
\draw[thick, dotted] (0,1)--(2,1);
\draw[thick, dotted] (1,0)--(1,2);
\draw[thick, dotted] (2,0)--(2,2);
\end{tikzpicture}
}
\right)
,
\displaystyle (\overline 1, \overline 0,  \overline 1)
\right)
\\
{}
\\
{}
\\
{}
\\
{}
\\
{}
\\
{}
\\
{}
\\
{}
\\
{}
\\
{}
\\
{}
\\
{}
\\
{}
\\
{}
\\
{}
\\
{}
\\
{}
\\
{}
\\
{}
\\
{}
\\
{}
\\
{}
\\
{}
\\
{}
\\
{}
}
\;\;
\substack{
\displaystyle\rotatebox[origin=c]{-45}{$\leftrightarrow$}
\\
{}
\\
{}
\\
{}
\\
{}
\\
{}
\\
{}
\\
{}
\\
{}
\\
{}
\\
{}
\\
{}
}
\hspace{-10mm}
\hackcenter{
\begin{tikzpicture}[scale=0.4]
%
\draw[thick, fill=lightgray!50] (0,0)--(1,0)--(1,1)--(2,1)--(2,4)--(3,4)--(3,6)--(2,6)--(2,5)--(1,5)--(1,4)--(0,4)--(0,0);
\draw[thick, fill=lightgray!50] (4,7)--(10,7)--(10,8)--(11,8)--(11,9)--(12,9)--(12,10)--(13,10)--(13,11)--(11,11)--(11,10)--(8,10)--(8,9)--(7,9)--(7,8)--(4,8)--(4,7);
\draw[thick, fill=lightgray!50] (14,12)--(15,12)--(15,13)--(16,13)--(16,15)--(15,15)--(15,14)--(14,14)--(14,12);
%
\draw[thick] (1,0)--(1,4);
\draw[thick] (2,1)--(2,5);
\draw[thick, dotted] (0,1)--(1,1);
\draw[thick, dotted] (0,2)--(2,2);
\draw[thick, dotted] (0,3)--(2,3);
\draw[thick, dotted] (0,4)--(2,4);
\draw[thick, dotted] (1,5)--(3,5);
\draw[thick] (0,2)--(1,2);
\draw[thick] (1,3)--(2,3);
\draw[thick, dotted] (5,7)--(5,8);
\draw[thick, ] (6,7)--(6,8);
\draw[thick, dotted] (7,7)--(7,8);
\draw[thick, dotted] (8,7)--(8,9);
\draw[thick, dotted] (9,7)--(9,10);
\draw[thick, dotted] (10,7)--(10,10);
\draw[thick, dotted] (10,7)--(10,10);
\draw[thick, dotted] (11,8)--(11,11);
\draw[thick, dotted] (12,9)--(12,11);
\draw[thick] (5,8)--(10,8);
\draw[thick] (7,9)--(11,9);
\draw[thick] (9,10)--(12,10);
\draw[thick, ] (8,7)--(8,8);
\draw[thick, ] (9,8)--(9,9);
\draw[thick, ] (10,9)--(10,10);
\draw[thick, dotted] (14,13)--(15,13);
\draw[thick, dotted] (15,14)--(16,14);
\draw[thick, ] (15,13)--(15,14);
\node at (0.5,0.5){$\scriptstyle1 $};
\node at (0.5,1.5){$\scriptstyle0 $};
\node at (0.5,2.5){$\scriptstyle1 $};
\node at (0.5,3.5){$\scriptstyle0 $};
\node at (1.5,1.5){$\scriptstyle1 $};
\node at (1.5,2.5){$\scriptstyle0 $};
\node at (1.5,3.5){$\scriptstyle1 $};
\node at (1.5,4.5){$\scriptstyle0 $};
\node at (2.5,4.5){$\scriptstyle1 $};
\node at (2.5,5.5){$\scriptstyle0 $};
\node at (4.5,7.5){$\scriptstyle0 $};
\node at (5.5,7.5){$\scriptstyle1 $};
\node at (6.5,7.5){$\scriptstyle0 $};
\node at (7.5,7.5){$\scriptstyle1 $};
\node at (8.5,7.5){$\scriptstyle0 $};
\node at (9.5,7.5){$\scriptstyle1 $};
\node at (6.5+1,7.5+1){$\scriptstyle0 $};
\node at (7.5+1,7.5+1){$\scriptstyle1 $};
\node at (8.5+1,7.5+1){$\scriptstyle0 $};
\node at (9.5+1,7.5+1){$\scriptstyle1 $};
\node at (6.5+2,7.5+2){$\scriptstyle0 $};
\node at (7.5+2,7.5+2){$\scriptstyle1 $};
\node at (8.5+2,7.5+2){$\scriptstyle0 $};
\node at (9.5+2,7.5+2){$\scriptstyle1 $};
\node at (8.5+3,7.5+3){$\scriptstyle0 $};
\node at (9.5+3,7.5+3){$\scriptstyle1 $};
\node at (14.5,12.5){$\scriptstyle1 $};
\node at (14.5,13.5){$\scriptstyle0 $};
\node at (14.5+1,12.5+1){$\scriptstyle1 $};
\node at (14.5+1,13.5+1){$\scriptstyle0 $};
\end{tikzpicture}
}
\end{align*}
\caption{An element \((\btau, \bbep) \in \SSSSS(3,15)\); corresponding semicuspidal shape \(\zeta^{(\btau, \bbep)} \in \SSSS(15\delta)\)
}
\label{fig:e2dil}       
\end{figure}

\end{example}

\subsection{Main theorems, semicuspidal version}

\begin{theorem}\label{allsemicusp}
The set \(\{\zeta^{m \beta} \mid m \in \NN, \beta \in \Phi_+^\re\} \cup \{ \zeta^{(\btau, \bbep)} \mid m \in \NN, (\btau, \bbep) \in \SSSSS(m)\}\) represents a complete and irredundant set of semicuspidal skew shapes, up to \(e\)-similarity.
\end{theorem}
\begin{proof}
Follows from Corollary~\ref{realsemicor} and Lemma~\ref{arbdelsemi}.
\end{proof}

\begin{theorem}\label{mainthmscKost}
Let \(\tau\) be a nonempty skew shape. Then
there exists a unique semicuspidal strict Kostant tiling \(\Gamma_\tau^{sc}\) for \(\tau\) given by \(\Gamma_\tau^{sc} = \{\sqcup_{\gamma \in \Gamma_\tau \cap \SSSS(\beta)}\gamma \mid \beta \in \Psi, \kappa^{\Gamma_\tau}_\beta >0\}\), and \(\bkap^{\Gamma_\tau^{sc}} = \bkap^{\Gamma_\tau}\).
\end{theorem}
\begin{proof}
Let \(\beta \in \Psi, \kappa^{\Gamma_\tau}_\beta >0\). Then \(\sqcup_{\gamma \in \Gamma_\tau \cap \SSSS(\beta)}\gamma\) is a nonempty skew shape by Lemma~\ref{tabisskew}, and is semicuspidal by Lemma~\ref{sciffcusptile}. The fact that \(\Gamma_\tau^{sc}\) is strict Kostant follows from the definition of \(\Gamma_\tau^{sc}\) and the fact that \(\Gamma_\tau\) is Kostant. For uniqueness, assume that \(\Lambda\) is any strict semicuspidal Kostant tiling of \(\tau\). Let \(\lambda \in \Lambda\). By Lemma~\ref{sciffcusptile}, every \(\lambda \in \Lambda\) is tiled by cuspidal ribbons of content \(\psi(\cont(\lambda))\), so we may refine \(\Lambda\) to a cuspidal Kostant tiling of \(\tau\), which by Theorem~\ref{mainthmcusp} is equal to \(\Gamma_\tau\). By strictness of \(\Lambda\), it follows then that every \(\lambda \in \Lambda\) is a union of {\em all} \(\gamma \in \Gamma_\tau\) of the same content, so \(\Lambda = \Gamma_\tau^{sc}\), proving uniqueness. The final statement follows directly from the construction of \(\Gamma_\tau^{sc}\).
\end{proof}

\section{An application to representation theory of KLR algebras}\label{appSpecht}
The combinatorial study of cuspidality and Kostant tilings for skew shapes in this paper is motivated by a connection to the theory of cuspidal systems and Specht modules over KLR algebras. We explain the connection in this section.

\subsection{KLR algebras}
We continue with our choice of \(e \in \Z_{>1}\), associated positive root system \(\Phi_+\) of type \({\tt A}^{(1)}_{e-1}\), and convex preorder \(\succeq\). We additionally fix an arbitrary ground field \(\k\).
For every element in the positive root lattice \(\theta \in Q_+\), there is an associated \(\ZZ\)-graded \(\k\)-algebra \(R_\theta\), called a {\em KLR} algebra. This family of algebras categorifies the positive part of the quantum group \(U_q(\widehat{\mathfrak{sl}}_e)\), see \cite{KL1, KL2, Rouq}. As we will focus on the combinatorics surrounding these algebras and not the specifics of, say, the presentation of \(R_\theta\), we refer the interested reader to the aforementioned papers for such details.

\subsection{Representation theory of \(R_\theta\)}
We consider the category \(R_\theta\)-mod of finitely generated \(\ZZ\)-graded \(R_\theta\)-modules. We will use the \(\cong\) symbol to indicate a (degree-preserving) isomorphism of \(R_\theta\)-modules, and \(\approx\) to indicate an isomorphism of \(R_\theta\)-modules up to some grading shift.
For \(\theta_1,\ldots, \theta_k \in Q_+\), there is an inclusion 
\begin{align*}
R_{\theta_1, \ldots, \theta_k}:=R_{\theta_1} \otimes \cdots \otimes  R_{\theta_k} \to R_{\theta_1 + \cdots + \theta_k}, 
\end{align*} 
with accompanying induction and restriction functors
\begin{align*}
\Ind_{\theta_1,\ldots, \theta_k}^{\theta_1 + \cdots + \theta_k}: & \;R_{\theta_1, \ldots, \theta_k}\textup{-mod} \to R_{\theta_1 + \cdots + \theta_k}\textup{-mod},\\
\Res_{\theta_1,\ldots, \theta_k}^{\theta_1 + \cdots + \theta_k}: & \;R_{\theta_1 + \cdots + \theta_k}\textup{-mod} \to R_{\theta_1, \ldots, \theta_k}\textup{-mod},
\end{align*}
as defined, for instance, in \cite{KL1}.

\subsection{Cuspidal systems and classification of simple \(R_\theta\)-modules}\label{stratasec}
Following \cite{KCusp, McN, KM}, for \(m \in \NN\), \(\beta \in \Phi_+\), we say an \(R_{m\beta}\)-module \(M\) is {\em semicuspidal} provided that for all \(0 \neq \theta_1, \theta_2 \in Q_+\) with \(\theta_1 + \theta_2 = m\beta\), we have \(\Res_{\theta_1, \theta_2}^{m\beta} M \neq 0\) only if \(\theta_1\) is a sum of positive roots \(\preceq \beta\) and \(\theta_2\) is a sum of positive roots \(\succeq \beta\). We say moreover that \(M\) is {\em cuspidal} if \(m=1\) and the comparisons above are strict.
Cuspidal and semicuspidal modules are key building blocks in the representation theory of \(R_\theta\). 

For \(\theta \in Q_+\), we define a {\em root partition} \(\pi = (\bkap, \blam)\) as the data of a Kostant partition \(\bkap = (\kappa_\beta)_{\beta \in \Psi}\in \Xi(\theta)\), together with 
an \((e-1)\)-multipartition \(\blam\) of \(\kappa_\delta\). We set \(\Pi(\theta)\) to be the set of all root partitions of \(\theta\), and define the `forgetful' map \(\rho:\Pi(\theta) \to \Xi(\theta)\) via \(\rho((\bkap, \blam)) = \bkap\). The partial order \(\trianglerighteq\) on \(\Xi(\theta)\) induces a partial preorder on \(\Pi(\theta)\) via \(\rho\).

To each \(\pi \in \Pi(\theta)\), one may associate a certain proper standard module \(\Delta(\pi)\) which is an ordered induction product of simple semicuspidal modules, see for instance \cite[(6.5)]{KM}. The module \(\Delta(\pi)\) has a simple head \(L(\pi)\), and 
\(
\{ L(\pi) \mid \pi \in \Pi(\theta)\}
\)
is a complete and irredundant set of simple \(R_\theta\)-modules up to isomorphism and grading shift, as explained in \cite{KCusp, TW, KM}.

\subsection{Specht modules}\label{Spechtmodsec}
It is shown in \cite{BK} that level \(\ell\) cyclotomic quotients of KLR algebras are isomorphic to blocks of level \(\ell\) cyclotomic Hecke algebras associated to complex reflection groups. Of particular interest is the case where \(e\) is prime and \(\textup{char} \,\k = e\); in this situation a level one quotient of \(\bigoplus_{\height(\theta) = n} R_\theta\) is isomorphic to the symmetric group algebra \(\k \mathfrak{S}_{n}\). 

Along the lines of this connection, Kleshchev-Mathas-Ram describe in \cite{KMR}, for any \(\ell\)-multipartition \(\blam\) of multicharge \(\bc\) and content \(\theta\), the presentation of an associated {\em Specht module} \(S^{\blam} \in R_\theta\textup{-mod}\). Specht modules are cell modules in the cellular structure for cyclotomic quotients of \(R_\theta\) defined in \cite{HM}, and hence are key objects in the representation theory of these algebras.
In \cite{Muth}, this construction was extended to define {\em skew Specht} \(R_\theta\)-modules which are of primary interest in this section.

\subsubsection{Skew Specht modules}
Let \(\tau \in \SSSS(\theta)\). We define the {\em (row) skew Specht \(R_\theta\)-module} \(S^\tau\) using the presentation in \cite[Definition~4.5]{Muth}. We remark that, although the definition in that paper is given by considering \(\tau\) as the set difference of Young diagrams \(\lambda /\mu = \tau\), the presentation of the module depends only on the nodes in \(\tau\), and the choices of Young diagrams \(\lambda, \mu\) that realize \(\tau\) only serve to determine the overall grading shift of the module. As we are not invested in grading shifts in this paper, we will simply assume that the generating vector of \cite[Definition~4.5]{Muth} is placed in \(\ZZ\)-degree zero. We note then that \(S^\tau \cong S^\mu\) whenever \(\tau \sim_e \mu\).

The specifics of the presentation of \(S^\tau\) are not needed here, so we refer interested readers to \cite{KMR, Muth} for more information. Our investigation of these modules will rely primarily on the basic combinatorial facts provided in Propositions~\ref{kbasis} and~\ref{Spechtres} below.

\begin{remark}\label{highno}
Every (higher level) skew Specht module \(S^{\blam/\bmu}\) is isomorphic (up to grading shift) to some \(S^\tau\) described herein. Indeed, one may associate \(\blam/\bmu\) with a skew shape \(\tau\) as in \S\ref{otherformsec}. The resultant modules \(S^\tau\) and \(S^{\blam/\bmu}\) are then isomorphic up to grading shift.
\end{remark}

\begin{proposition}\label{kbasis} \cite[Theorem~5.12]{Muth}
Let \(\theta \in Q_+\). Then \(S^\tau\) has \(\k\)-basis in bijection with the set of Young tableaux for \(\tau\).
 \end{proposition}
 
In particular, Proposition~\ref{kbasis} implies that \(S^\tau =0\) if and only if \(\tau = \varnothing\).

\begin{proposition}\label{Spechtres}
Let \(0 \neq \theta, \theta_1, \ldots, \theta_k \in Q_+\), with \(\theta = \theta_1 + \cdots + \theta_k\). Let \(\tau \in \SSSS(\theta)\). Then \(\Res_{\theta_1, \ldots, \theta_k}^\theta S^\tau\) has a \(R_{\theta_1, \ldots, \theta_k}\)-module filtration with subquotients isomorphic (up to grading shift) to \(S^{\tau_1} \boxtimes \cdots \boxtimes S^{\tau_k}\), ranging over all tableaux \((\tau_1, \ldots, \tau_k)\) for \(\tau\) such that \(\tau_i \in \SSSS(\theta_i)\) for all \(i \in [1,k]\).
\end{proposition}
\begin{proof}
Follows from inductive application of \cite[Theorem~5.13]{Muth}.
\end{proof}

\subsection{Cuspidal Specht modules}
Our Definition~\ref{cuspdef} of cuspidality for skew shapes is motivated by a connection with cuspidal Specht modules, as detailed in the following proposition.

\begin{proposition}\label{cuspiscusp}
Let \(\tau\) be a skew shape. Then the Specht module \(S^\tau\) is cuspidal (resp. semicuspidal) if and only if the skew shape \(\tau\) is cuspidal (resp. semicuspidal).
\end{proposition}
\begin{proof}
We prove the statement for cuspidality; the semicuspidality statement is similar.
Let \(\beta \in \Phi_+\), and \(\tau \in \SSSS(\beta)\). 

\((\implies)\) Assume \(S^\tau\) is cuspidal. Let \((\tau_1, \tau_2)\) be a tableau for \(\tau\). Write \(\theta_1 = \cont(\tau_1)\), \(\theta_2 = \cont(\tau_2)\). Then \(\Res_{\theta_1, \theta_2}^\beta S^\tau \neq 0\), as it has a nonzero subquotient \(S^{\tau_1} \boxtimes S^{\tau_2}\) by Proposition~\ref{Spechtres}. But then by cuspidality of \(S^\tau\), we have that \(\theta_1\) is a sum of positive roots \(\prec \beta\), and \(\theta_2\) is a sum of positive roots \(\succ \beta\). Thus \(\tau\) is cuspidal by Definition~\ref{cuspdef}.

\((\impliedby)\) Assume \(\tau\) is cuspidal. Let \(0 \neq \theta_1, \theta_2 \in Q_+\) with \(\theta_1 + \theta_2 =  \beta\). If \(\Res_{\theta_1, \theta_2}^\beta S^\tau \neq 0\), then there must be a nonzero subquotient in the filtration described in Proposition~\ref{Spechtres}. Thus there exists a tableau \((\tau_1, \tau_2)\) for \(\tau\) with \(\tau_1 \in \SSSS(\theta_1)\) and \(\tau_2 \in \SSSS(\theta_2)\). By the cuspidality of \(\tau\), we have that \(\theta_1\) is a sum of positive roots \(\prec \beta\), and \(\theta_2\) is a sum of positive roots \(\succ \beta\). Thus \(S^\tau\) is cuspidal.
\end{proof}

In view of Proposition~\ref{cuspiscusp} and Remark~\ref{highno}, Theorems~\ref{allcusp} and~\ref{allsemicusp} give a full classification and construction (up to grading shift) of all cuspidal and semicuspidal skew Specht modules for the KLR algebra. 

\subsubsection{Simple semicuspidal modules}
Using representation-theoretic results, it is established in \cite[Proposition 8.5]{Muth} that when \(\succeq\) is a `balanced' convex preorder, every real simple cuspidal module is isomorphic to (a grading shift of) a certain ribbon Specht module. We extend this result to arbitrary convex preorders in the following proposition.

\begin{proposition}\label{propcuspSpechtsimp}
Let \(m \in \NN\), \(\beta \in \Phi_+^\re\). Then \(S^{\zeta^\beta}\) is the unique simple cuspidal \(R_\beta\)-module, and \(S^{\zeta^{m\beta}}\) is the unique simple semicuspidal  \(R_{m \beta}\)-module, up to grading shift.
\end{proposition}
\begin{proof}
By \cite[Theorem 5.2]{KM}, there is a unique simple cuspidal module \(L(\beta)\). By Theorem~\ref{allcusp}, \(\zeta^\beta \in \SSSS(\beta)\) is cuspidal. By Proposition~\ref{cuspiscusp} then, \(S^{\zeta^\beta}\) is cuspidal, so every simple factor of \ \(S^{\zeta^\beta}\) must be cuspidal, and thus all are isomorphic to \(L(\beta)\) up to some grading shifts. An extremal word argument, using \cite[Lemma 2.28]{KCusp} as in the proof of \cite[Lemma~8.3]{Muth} shows that this factor may occur only once. Thus  \(S^{\zeta^\beta} \approx L(\beta)\).

Let \(m \in \NN\). By \cite[Theorem 5.2]{KM} there is a unique simple cuspidal \(R_{m \beta}\)-module \(L(\beta^m)\) up to grading shift. Moreover, we have \(L(\beta^m) \cong \Ind_{\beta, \ldots, \beta}^{m \beta}(L(\beta)^{\boxtimes m})\). By Theorem~\ref{allsemicusp}, \(\zeta^{m\beta}\) is semicuspidal, and consists of \(m\) connected components, each of which is \(e\)-similar to \(\zeta^\beta\). By \cite[Theorem~5.15]{Muth}, we have then that \(S^{\zeta^{m \beta}} \approx \Ind_{\beta, \ldots, \beta}^{m\beta}((S^{\zeta^\beta})^{\boxtimes m})\). Then it follows from the first paragraph that \(S^{\zeta^\beta} \approx L(\beta^m)\), as desired.
\end{proof}

\subsection{Simple factors of skew Specht modules}

One of the central open problems in the representation theory of cyclotomic Hecke algebras, and by extension in the representation theory of KLR algebras, is the determination of the simple factors and decomposition numbers of Specht modules. 
The following proposition constitutes a tight upper bound (in the bilexicographic order on root partitions) for the simple factors of skew Specht modules.

\begin{proposition}\label{simplefactors}
Let \(\tau \in \SSSS(\theta)\). Then the Specht module \(S^\tau\) has a simple factor \(L(\pi)\) with \(\rho(\pi) = \bkap^{\Gamma_\tau}\), and \(\bkap^{\Gamma_\tau} \trianglerighteq \rho(\mu)\) for all simple factors \(L(\mu)\) of \(S^\tau\).
\end{proposition}
\begin{proof}
Assume \(L(\mu)\) is a simple factor of \(S^\tau\). By \cite[Theorem~6.8(v)]{KM}, \(\Res_{\rho(\mu)} S^\tau \neq 0\). Then by Proposition~\ref{Spechtres}, there exists a Kostant tiling \(\Lambda\) for \(\tau\) with \(\bkap^\Lambda = \rho(\mu)\). But then, by Theorem~\ref{mainthmcusp}, \(\bkap^{\Gamma_\tau} \trianglerighteq \rho(\mu)\).

By \cite[Theorem~6.8(v)]{KM}, if \(\bkap^{\Gamma_\tau}\triangleright \rho(\mu)\) for all simple factors \(L(\mu)\) of \(S^\tau\), then we must have \(\Res_{\bkap^{\Gamma_\tau}}S^\tau =0\). But \(\Gamma^{sc}_\tau\) is a tiling with \(\bkap^{ \Gamma^{sc}_\tau} = \bkap^{\Gamma_\tau}\) by Theorem~\ref{mainthmscKost}, so \(\Res_{\bkap^{\Gamma_\tau}}S^\tau \neq0\). Therefore \(S^\tau\) has some simple factor \(L(\pi)\) with \(\rho(\pi) = \bkap^{\Gamma_\tau}\).
\end{proof}

\subsection{Related questions and future work}

\subsubsection{Simple labels}
Simple \(R_\theta\)-modules have two differing sets of labels, coming from the stratified structure of the full affine KLR algebra, or alternatively from the cellular structure of cyclotomic quotients of the KLR algebra. Indeed, simple modules may be labeled by \(L(\pi)\) for \(\pi \in \Pi(\theta)\) via the stratified structure of \(R_\theta\)-mod described in \S\ref{stratasec}. 
Alternatively, \(R_\theta\)-modules have labels of the form \(D^{\blam}\), where \(\blam\) is a Kleshchev multipartition, and \(D^{\blam}\) is the simple head of the associated Specht module \(S^{\blam}\), see \cite{AM, HM}. Proposition~\ref{simplefactors} gives a bound on simple factors of Specht  modules in terms of root partitions, and is thus a rough step in the direction of understanding the connection between these approaches. We expect a much more delicate combinatorial process is required to match the Kleshchev multipartition \(\blam\) with the root partition \(\pi\) which labels the same simple module.

\subsubsection{Simple imaginary semicuspidal modules}
In general, simple imaginary semicuspidal \(R_{n \delta}\)-modules are not isomorphic to skew Specht modules, but do appear to arise as heads (or socles) of imaginary semicuspidal skew Specht modules. For a balanced convex preorder, evidence for this assertion appears in \cite{EK}, which relates some semicuspidal \(R_{n \delta}\)-modules to RoCK blocks of Hecke algebras via Morita equivalence. For arbitrary convex preorders, we expect a similar connection to hold with blocks which are Scopes equivalent to RoCK blocks.

\subsubsection{Other types}
Cuspidal modules for KLR algebras of all untwisted affine types were defined and studied in \cite{KCusp}. Specht modules, defined in the combinatorial setting of Young diagrams, albeit with a different treatment for residues, have been defined for the KLR algebra of affine type C in \cite{APS}. We expect that, with reasonable modifications, versions of Theorems~\ref{allcusp} and~\ref{allsemicusp} should hold in this combinatorial setting, allowing for a presentation of cuspidal and semicuspidal modules via (skew) Specht modules in affine type C.

\end{document}